\documentclass[a4paper,10pt]{article}
\usepackage[utf8]{inputenc}

\usepackage{graphicx}       
\usepackage{longtable} 
\usepackage{color}
\usepackage{amssymb,amsmath}
\usepackage{float}
\usepackage{enumitem}
\usepackage{subfigure}
\usepackage{units}
\usepackage{cases}
\usepackage{cite}
\usepackage{tikz}
\usepackage{amsthm}

\usepackage{hyperref}
\usepackage{cleveref}
\usepackage{authblk}

\newcommand{\vv}[1]{\overrightarrow{#1}}

\newcounter{TPFassumption}
\stepcounter{TPFassumption}
\newcounter{TPFproperties}
\stepcounter{TPFproperties}

\newtheorem{remark}{Remark}[section]
\newtheorem{proposition}{Proposition}[section]
\newtheorem{corollary}{Corollary}[section]
\newtheorem{theorem}{Theorem}[section]
\newtheorem{definition}{Definition}[section]
\def \a  {\alpha}

\def \b  {\beta}
\def \g  {\gamma}
\def \Gf {\mathcal{G}}

\def \d  {\delta}

\def \e  {\varepsilon}

\def \vr  {\varrho}

\def \Om {\Omega}

\def \t  {\tau}

\def \del {\nabla}

\def \H  {{\cal H}}
\def \Hdr  {{\mathcal{H}^{(d)}}}
\def \Him  {{\mathcal{H}^{(i)}}}
\def \Pim  {{p^{(i)}_c}}
\def \Pdr  {{p^{(d)}_c}}
\def \F   {{\mathcal{F}}}
\def \p  {\partial}
\def \R  {\mathbb{R}}
\def \N  {\mathbb{N}}
\def \Gf  {{\cal G}}

\def \sgn {{\rm sign}}

\def \W {{\bf {\cal W}}}

\newenvironment{rcases}
  {\left.\begin{aligned}}
  {\end{aligned}\right\rbrace}

\title{Fronts in two-phase porous flow problems: effects of hysteresis and dynamic capillarity}
\author[1]{K. Mitra}
\author[2]{T. K\"oppl}
\author[3]{I. S. Pop}
\author[4]{C. J. van Duijn}
\author[5]{R. Helmig}

\affil[1]{Eindhoven University of Technology, Department of Mathematics and Computer Science, Groene loper 5, 5612 AZ Eindhoven, Netherlands,
Hasselt University, Faculty of Science, Martelarenlaan 42, BE3500 Hasselt, Belgium, k.mitra@tue.nl}

\affil[2]{Department of Mathematics, University of Technology Munich, 
Boltzmannstra\ss e 3, 85748 Garching bei M\"unchen, Germany, koepplto@ma.tum.de}

\affil[3]{Hasselt University, Faculty of Science, Martelarenlaan 42, BE3500 Hasselt, Belgium, 
University of Bergen, Department of Mathematics, Norway, sorin.pop@uhasselt.be}

\affil[4]{University of Utrecht, Department of Earth Sciences, Princetonlaan 8a, 3584 CB Utrecht, Netherlands,
Eindhoven University of Technology, Department of Mechanical Engineering, PO Box, 513 5600 MB Eindhoven, Netherlands, c.j.v.duijn@tue.nl}

\affil[5]{Department of Hydromechanics and Modelling of Hydrosystems, University of Stuttgart, Pfaffenwaldring 61, 70569 Stuttgart, Germany
rainer.helmig@iws.uni-stuttgart.de}

\date{}                     
\begin{document}

\maketitle

\begin{abstract}
In this work, we study the behaviour of saturation fronts for two phase flow through a long homogeneous porous column. In particular, 
the model  includes   hysteresis and dynamic effects in the capillary pressure and hysteresis in the permeabilities. The analysis uses 
travelling wave approximation. Entropy solutions are derived for Riemann problems that are arising in this context. 
These solutions belong to a much broader class compared to the standard Oleinik solutions, where  hysteresis and 
dynamic effects are neglected. The relevant cases are examined and the corresponding solutions are categorized. 
They include non-monotone profiles, multiple shocks and self-developing stable saturation plateaus. 
Numerical results are presented that illustrate the mathematical analysis. 
Finally, we compare experimental results with our theoretical findings.
\end{abstract}

\section{Introduction}
Modelling of two phase flow through the subsurface is important for many practical applications, from groundwater modelling 
and oil and gas recovery to CO$_2$ sequestration. 
For this purpose the mass balance equations are used which read in the absence of source terms \cite{Bear1979,helmig1997multiphase} as follows:
\begin{equation}
\label{2PF_eq:massbalance}
\phi \frac{\p\left(\rho_\a S_\a \right)}{\partial t} + \del\cdot\left(\rho_\a v_\a \right) = 0,\; \a \in \left\{w,n \right\},
\end{equation}
where $\a=n$ denotes the non-wetting phase and $\a=w$ the wetting phase. Further, $\phi$ is the porosity, $S_\alpha$ 
and $\rho_\alpha$ the saturation and density of the phases. 
The phase-velocities $v_\alpha$ are given by the Darcy's law \cite{Bear1979,helmig1997multiphase},
\begin{equation}
\label{2PF_eq:Darcy}
v_\alpha = -\frac{k_{r\alpha}}{\mu_\alpha}K \left( \del p_\a 
- \rho_\a g\hat{e}_g \right),\; \a \in \left\{w,n\right\}.
\end{equation}
Here $K\left[ \unit{m}^2 \right]$ is the absolute permeability of the porous medium,  
$\mu_\alpha \left[Pa\cdot s\right]$ the viscosity and $k_{r\alpha}$ the relative permeability 
of each phase. Moreover, $p_\alpha\; \left[\unit{Pa} \right]$, 
$g\;\left[\unitfrac{m}{s^2} \right]$ and $\hat{e}_g$ stand for the phase pressure, the gravitational acceleration 
and the unit vector along gravity, respectively. 
Observe that the system \eqref{2PF_eq:massbalance}-\eqref{2PF_eq:Darcy} is not closed as there are more unknowns 
than equations, i.e. $S_\a$, $k_{r\a}$ and $p_\a$. Hence, one needs 
to take certain assumptions. Assuming incompressibility  results in $\rho_\a$ being constant. Moreover, by definition 
\begin{equation}
\label{2PF_eq:satBal}
S_w + S_n = 1.
\end{equation}
Commonly it is assumed that the relative permeabilities, as well as the phase pressure difference, are 
functions of the saturation of the wetting phase \cite{Bear1979,helmig1997multiphase},
\begin{equation}
k_{rn}=k_{rn}(S_w),\quad  k_{rw}=k_{rw}(S_w) \text{ and } p_n-p_w=p_c(S_w).\label{2PF_eq:standardModel}
\end{equation}
The function $p_c: (0,1]\to \R^+$ is referred to as the capillary pressure function. 
System \eqref{2PF_eq:massbalance}-\eqref{2PF_eq:standardModel} reduces to the hyperbolic Buckley-Leverett equation if this term is 
neglected, i.e. $p_c\equiv 0$. The model given by \eqref{2PF_eq:massbalance}-\eqref{2PF_eq:standardModel} works well under close to 
equilibrium conditions and when flow reversal does not take place. However, some more general cases cannot be explained by this model.

One of the first evidences of deviation from the standard model was reported in the 1931 paper by 
Richards \cite{richards1931capillary} where he concluded that the capillary pressure term is hysteretic in nature. 
Capillary hysteresis refers to the phenomenon that $p_c$ measured for a wetting phase infiltration process follows a curve, 
denoted here by $\Pim(S_w)$, which differs from $p_c$ measured for a drainage process, denoted by $\Pdr(S_w)$. If the 
process changes from infiltration to drainage or vice versa, then the $p_c$ follows scanning curves that are intermediate to 
$\Pim(S_w)$ and $\Pdr(S_w)$ \cite{beliaev2001theoretical}.  This is shown in detail in  \Cref{2PF_fig:Scanning} (left). Since then,
hysteresis has been studied experimentally \cite{morrow1965capillary,poulovassilis1970hysteresis,zhuang2017analysis}, 
analytically \cite{Cao2015688,VANDUIJN2018232,el2018traveling,SCHWEIZER20125594,ratz2014hysteresis} and numerically 
\cite{Zhang2017,papafotiou2010numerical,ratz2014hysteresis,schneider2018stable,brokate2012numerical}. Variety of models have been
proposed to incorporate the effects of hysteresis, such as independent and dependent domain models 
\cite{philip1964similarity,Maulem_hys,parker1987parametric} and interfacial area models 
\cite{HASSANIZADEH1990169,hassanizadeh1993toward,pop2009horizontal,niessner2008model}. A comprehensive study of these models 
can be found in \cite{Kmitra2017}. However, in this paper we will use the play-type hysteresis model  \cite{beliaev2001theoretical,Kmitra2017} 
that approximates scanning curves as constant saturation lines. This model is comparatively simple to treat analytically 
\cite{Cao2015688,SCHWEIZER20125594,ratz2014hysteresis,VANDUIJN2018232}, it has a physical basis \cite{beliaev2001theoretical,schweizer2005laws} 
and it can be extended to depict the realistic cases accurately \cite{Kmitra2017}.

\begin{figure}[H]
\begin{minipage}{.45\textwidth}
\centering
\includegraphics[scale=.3]{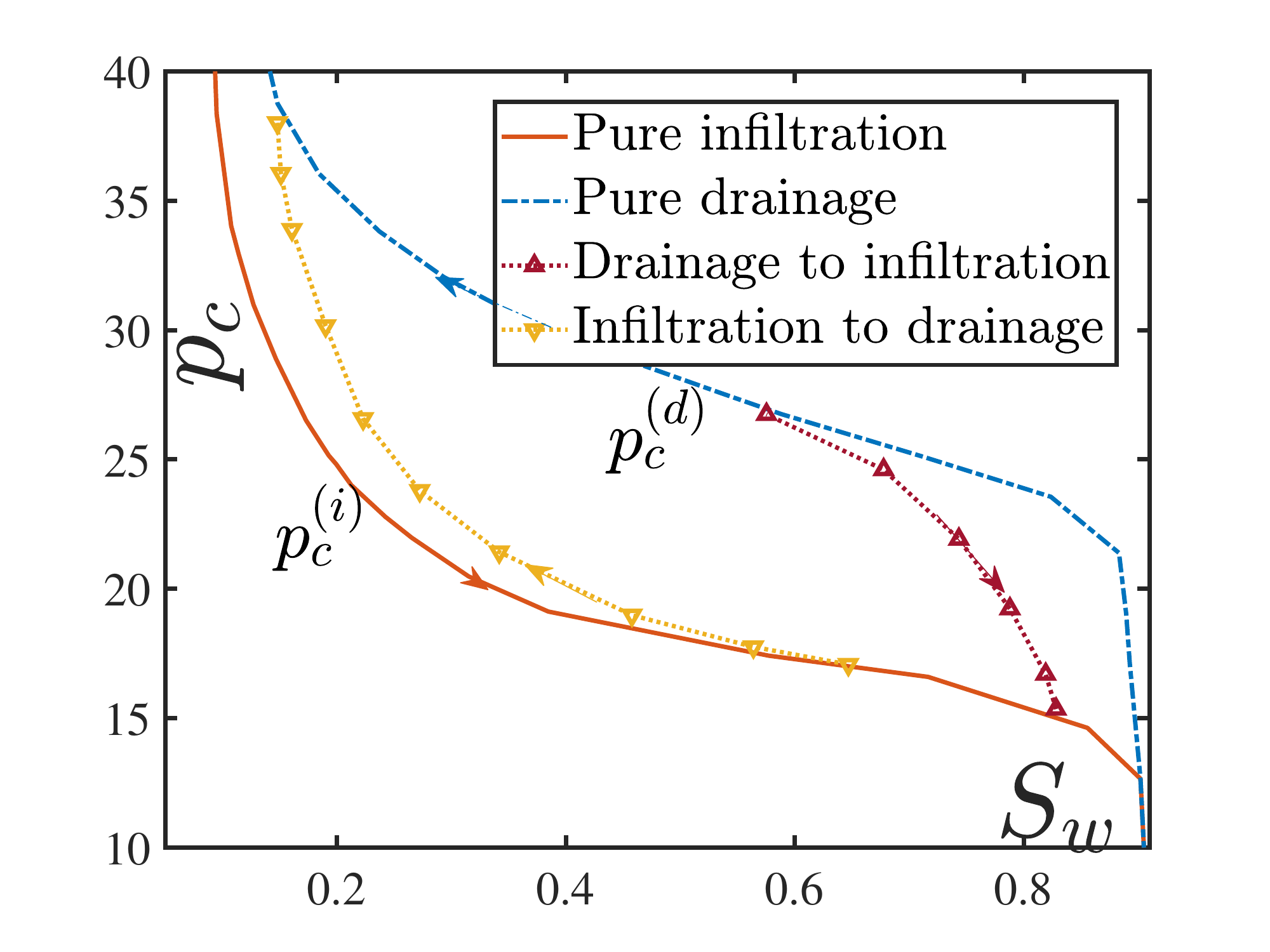}
\end{minipage}
\begin{minipage}{.45\textwidth}
\centering
\includegraphics[scale=.3]{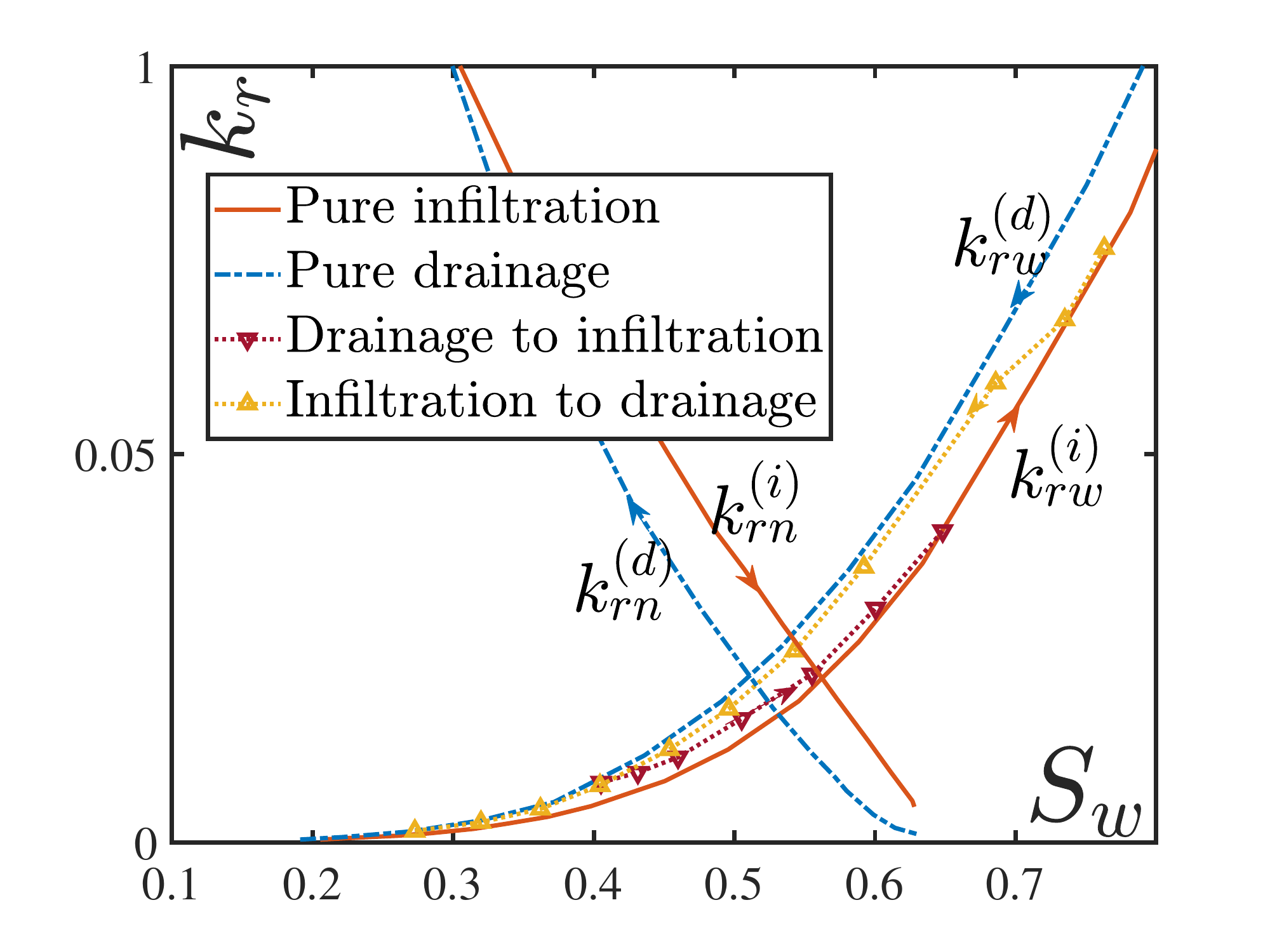}
\end{minipage}
\caption{(left) Hysteresis of capillary pressure and scanning curves. 
The plots drawn use data points from Figure 4 and 5 of \cite{morrow1965capillary}. (right) 
Hysteresis of relative permeabilities. Experimental data from \cite{topp1966hysteretic} are used for t
he $k_{rw}$ plots and the corresponding scanning curves. Plots for $k_{rn}$ show data from \cite{gladfelter1980effect}. 
The curves are scaled in the $y$ direction. }\label{2PF_fig:Scanning}
\end{figure}

A similar hysteretic behaviour is observed for the relative permeabilities too, although to a lesser
extent. Hysteresis of the non-wetting
phase relative permeability in the two phase case (oil and water for example) is reported in 
\cite{gladfelter1980effect,braun1995relative,killough1976reservoir}. The wetting permeability $k_{rw}$  also exhibits 
hysteresis \cite{topp1966hysteretic,poulovassilis1970hysteresis} but the effect is less pronounced, see  \Cref{2PF_fig:Scanning} (right).

Another effect that cannot be explained by the standard model is the occurrences of overshoots. More precisely, 
in infiltration experiments through initially low saturated soils it is observed that if the flow rate is large 
enough then the saturation at an interior point is larger than that on the boundary even in the absence of internal sources 
\cite{dicarlo2004experimental,bottero2011nonequilibrium,shiozawa2004unexpected}. This cannot be explained by a second order model 
such as \eqref{2PF_eq:massbalance}-\eqref{2PF_eq:standardModel}  \cite{egorov2003stability,van2004steady,schweizer2012instability}.
Hence, based on thermodynamic considerations the dynamic capillary model was proposed in \cite{hassanizadeh1993thermodynamic}. 
Since then the dynamic capillary term has been measured  experimentally \cite{camps2010experimental,kalaydjian1992dynamic} and it was 
used successfully  to explain overshoots 
\cite{cuesta2000infiltration,van2007new,van2013travelling,VANDUIJN2018232,mitra2018wetting,ratz2014hysteresis,SCHWEIZER20125594}. Also the 
well-posedness of the dynamic capillarity model has been proved \cite{cao2016degenerate,cao2015uniqueness,mikelic2010global,bohm1985diffusion} 
and numerical methods have been investigated \cite{karpinski2017analysis,cao2018convergence,karpinski2017_2,cao2018error,ewing1978time}.

\begin{figure}
\begin{center}
\includegraphics[width=0.5\textwidth]{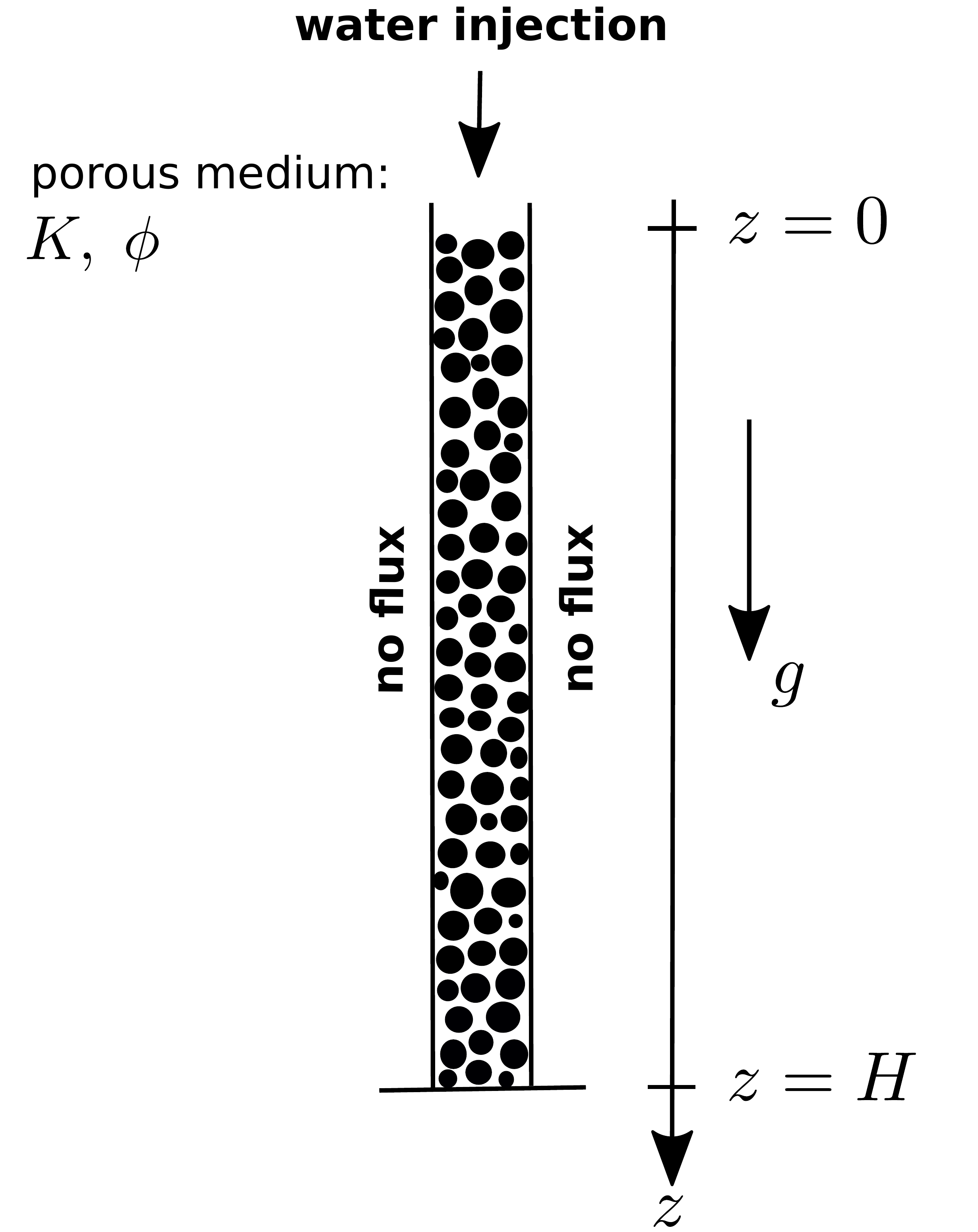}
\end{center}
\caption{\label{2PF_fig:column} Setup of an infiltration experiment. At the inlet of a column having the height 
$H$ water is injected by a constant rate. The main axis
of the column is orientated such that it is aligned with the gravity vector.}
\end{figure}

In this paper we are interested in studying how the flow behaviour is influenced if one considers the non-equilibrium effects, i.e. hysteresis
and dynamic capillarity. For this purpose, we study the system in a one-dimensional setting. The one dimensional case is relevant when one 
spatial direction is dominant; it approximates flow through viscous fingers 
\cite{ratz2014hysteresis,glass1989mechanism,rezanezhad2006experimental} and it can explain results from the standard experimental 
setting shown in  \Cref{2PF_fig:column}  \cite{dicarlo2004experimental,bottero2011nonequilibrium,shiozawa2004unexpected}. In this study, 
the behaviour of the fronts is investigated by traveling wave (TW) solutions. The TW solutions can approximate the saturation and pressure 
profiles in an infiltration experiments through a long column, and the existence conditions of the TWs act as the entropy conditions for 
the corresponding hyperbolic model when the viscous terms are disregarded. For the unsaturated case ($p_n=0$) TW solutions with dynamic 
effect were analysed in \cite{cuesta2000infiltration}. For the two phase case it was shown rigorously in 
\cite{van2007new,van2013travelling,spayd2011buckley} that non-monotone travelling waves and non-standard entropy solutions are existing 
if one includes dynamic capillarity effect. Similar analysis but for higher order viscous terms containing spatial derivatives were performed 
in \cite{bedjaoui2002diffusive,el2017dispersive}.  The existence of TW solutions for the unsaturated case when dynamic capillarity and 
capillary hysteresis are present was proved in \cite{VANDUIJN2018232,mitra2018wetting} and criteria for non-monotonicity and reaching full 
saturation were stated. It is evidenced in \cite{schneider2018stable,hilfer2014saturation} that hysteresis can explain stable 
saturation plateaus but  it cannot initiate them. In \cite{ratz2014hysteresis,el2018traveling} it is shown how both hysteresis and 
dynamic capillarity  are required to explain the growth of viscous fingers. The entropy conditions for Buckley-Leverett equation 
considering hysteresis in only permeability were derived in \cite{plohr2001modeling,schaerer2006permeability,bedrikovetsky1996mathematical}. 
However, hysteresis and nonlinearities were not included in the viscous term.  This is taken into consideration in \cite{abreu2019relaxation}
where the authors add a dynamic term to model permeability hysteresis, while disregarding hysteresis and dynamic  
effects in capillary pressure.  The behaviour of TW for a non-monotone flux function in the presence of a third order term was 
described in \cite{shearer2015traveling}.  

In our current work we build upon  \cite{cuesta2000infiltration,van2007new,van2013travelling,VANDUIJN2018232,mitra2018wetting}  
to describe the behaviour of fronts in the two-phase case when dynamic capillarity and both type of hysteresis are included.
The models that are used in our analysis are introduced in \Cref{2PF_sec:MathModel}. \Cref{2PF_sec:CaseA} discusses the existence 
of TWs when hysteresis and dynamic effects are included in the capillary pressure but not in the permeabilities. Entropy 
conditions are derived and they reveal that there can be non-classical shocks. In \Cref{2PF_sec:CaseB}, the analysis is extended 
to include hysteresis in the permeabilities.  This makes self-developing stable saturation plateaus and a broader class of entropy 
solutions possible. \Cref{2PF_sec:NumRes} presents numerical results that support our analytical findings. Finally, we make some 
concluding remarks in \Cref{2PF_sec:Conclude} and compare the results with experiments.

\section{Mathematical model}\label{2PF_sec:MathModel}

This section is dedicated to the formulation of a mathematical model that can be used to describe an
infiltration process of a fluid into a homogeneous porous column. An example for such an infiltration process is the
injection of water into a dry sand column (see Figure \ref{2PF_fig:column}).

\subsection{Governing equations}
Here we consider the one-dimensional situation where the flow problem is defined on an interval $\left(0,H\right)$.  This 
simplification is justified by the fact 
that the walls of the porous column, in which the fluids are injected are impermeable and that saturation is in general 
almost constant across the
section area of the column. The axis is pointing in the direction of gravity. The medium is assumed to be homogeneous 
and the fluids are incompressible. Under these constraints, \eqref{2PF_eq:massbalance}-\eqref{2PF_eq:Darcy} simplify to
\begin{equation}
\label{2PF_eq:2phase1stForm}
\phi \frac{\p S_\a}{\partial t} + \dfrac{\p v_\a}{\p z}=0,\; v_\a=-\dfrac{ k_{r\a}}{\mu_\a}K \left(\dfrac{\p p_\a}{\p z}
-\rho_\a g \right),\; \a \in \left\{w,n \right\},
\end{equation}
where $t$ and $z$ are denoting the time and space variables, respectively. To further simplify the model, after summing 
\eqref{2PF_eq:2phase1stForm} for the two phases and using \eqref{2PF_eq:satBal} we observe that the total velocity 
\begin{equation}
\label{eq:vtot}
v(z,t) = v_w(z,t) + v_n(z,t)
\end{equation}
is constant in space. In addition to that, we assume that $v$ is also constant in time, which occurs, e.g. if a constant 
influx (injection rate) is prescribed at the inlet $z=0$. This gives
\begin{equation}
\frac{\partial v}{\partial t} = \frac{\partial v}{\partial z}= 0, \text{ or } v\left( z,t \right) \equiv v \text{ for } 
z\in (0,H) \text{ and } t>0.\label{2PF_eq:TotalVel}
\end{equation}
From \eqref{2PF_eq:2phase1stForm}-\eqref{2PF_eq:TotalVel} one finds
\begin{equation}
v_w = \frac{k_{rw}}{k_{rw}+ \frac{\mu_w}{\mu_n}k_{rn}} v + \dfrac{K}{\mu_n} \frac{k_{rw} k_{rn}}{k_{rw}+ 
\frac{\mu_w}{\mu_n}k_{rn}} \left( \frac{ \partial p_c }{ \partial z} + \left(\rho_w-\rho_n \right)g \right).\label{2PF_eq:vwExp}
\end{equation}
At this stage we define the fractional flow function
\begin{equation}
f:=  \frac{k_{rw}}{k_{rw}+ \frac{\mu_w}{\mu_n}k_{rn}}, 
\label{2PF_eq:Deff}
\end{equation}
and the function
\begin{equation}
h:= \frac{k_{rw} k_{rn} }{k_{rw} + \frac{\mu_w}{\mu_n} k_{rn}} = k_{rn} f.
\label{2PF_eq:Defh}
\end{equation}
 Substituting these relations and definitions into \eqref{2PF_eq:2phase1stForm} for $\alpha = w$ yields the transport equation for the wetting phase
\begin{equation}
\frac{\partial S}{\partial t} + \frac{v}{\phi} \frac{\partial}{\partial z} \left[ f +
\frac{K \left(\rho_w-\rho_n \right) g}{v\mu_n} h + \frac{K}{v\mu_n} h
\frac{\partial p}{\partial z} \right] = 0,\label{2PF_eq:2phase2ndForm}
\end{equation}
where we used the notation
\begin{equation}
S:=S_w \text{ and } p:=p_n-p_w.
\end{equation}
Note that, $f$ and $h$ are functions of $S$ and possibly of $p$, as shown below.

\subsection{Modelling hysteresis and dynamic capillarity}
To incorporate hysteresis and dynamic capillarity in the model, one needs to extend capillary pressure and relative permeability 
given in the closure relationship \eqref{2PF_eq:standardModel}.  
\subsubsection{Capillary pressure}
The following expression is used to extend the capillary pressure:
 \begin{equation}
p\in \frac{1}{2}(\Pim(S)+\Pdr(S))-\frac{1}{2}(\Pdr(S)-\Pim(S))\cdot\sgn\left(\frac{\p S}{\p t}\right)-\t \frac{\p S}{\p t},\label{2PF_eq:Combined}
\end{equation}
where $\text{sign}(\cdot)$ denotes the multi-valued signum graph
\begin{equation}
\text{sign} \left( \xi \right) = 
\begin{cases}
 1, &\text{ for } \xi > 0,\\
 \left[-1,1 \right], &\text{ for } \xi = 0,\\
 -1, & \text{ for } \xi < 0,
\end{cases}\label{2PF_eq:DefSign}
\end{equation}
see \cite{beliaev2001theoretical,hassanizadeh1993thermodynamic,SCHWEIZER20125594}. 
The second and third term in the right hand side of \eqref{2PF_eq:Combined} describe, respectively, 
capillary hysteresis \cite{beliaev2001theoretical} and dynamic capillarity \cite{hassanizadeh1993thermodynamic}. Further, 
$\t\geq 0$ denotes the dynamic capillary coefficient. It models relaxation or damping in the capillary pressure. Although in practice
$\t$ may depend on $S$ \cite{camps2010experimental,bottero2011nonequilibrium}, here we assume it to be constant. The case of non-constant 
$\t$ is considered in \cite{VANDUIJN2018232,mitra2018wetting}.
The capillary pressure functions $p_c^{\left(j \right)}$, $j\in \{i,d\}$, fulfill 
\cite{Bear1979,helmig1997multiphase,morrow1965capillary}:
\begin{enumerate}[label=(P\theTPFproperties)]
  \item $p_c^{\left(j \right)}: \left(0,1 \right] \rightarrow \left[0,\infty \right),\; p_c^{\left(j \right)} \in
  C^1\left( \left( 0,1 \right] \right),\; p_c^{\left(j \right)} \left( 1 \right) = 0.$ Moreover, 
  ${ p_c^{\left(j \right)} }^\prime\left( S\right) < 0 \text{ and } p_c^{\left( i \right)}\left(S \right) < p_c^{\left( d \right)}\left(S \right) 
  \text{ for } S \in \left(0,1 \right).$\label{2PF_P:Pc}
\stepcounter{TPFproperties}
\end{enumerate}
Here, and later in this paper, a prime denotes differentiation with respect to the argument.
In the absence of dynamic effects, i.e. $\t=0$, expression \eqref{2PF_eq:Combined} implies 
$$
p=\begin{cases}
\Pim(S) &\text{ when } \p_t S>0,\\
\Pdr(S) &\text{ when } \p_t S<0.
\end{cases}
$$
This is precisely what is seen from water infiltration/drainage experiments \cite{morrow1965capillary}. 
When $\frac{\p S}{\p t}=0$, $p$ is between $\Pim(S)$ and $\Pdr(S)$. For this reason, the hysteresis described by \eqref{2PF_eq:Combined} 
is called play-type hysteresis: i.e. the scanning curves between $\Pim(S)$ and $\Pdr(S)$ are vertical.
 
Before discussing the case $\t>0$, we introduce for convenience the sets
\begin{align}
&\Him:=\{(S,p):S\in (0,1],\; p<\Pim(S)\},\\
&\Hdr:=\{(S,p):S\in (0,1],\; p>\Pdr(S)\},\\
&\H:=\{(S,p):S\in (0,1],\; \Pim(S)\leq p\leq \Pdr(S)\},\label{2PF_eq:SetHdef}
\end{align}
and the strip $\mathcal{W}=\Him\cup\H\cup\Hdr=\{0<S\leq 1\}$.
In \cite{mitra2018wetting} it is shown that pressure expression \eqref{2PF_eq:Combined} can be written as
\begin{equation}
\label{2PF_eq:Psi}
\frac{\partial S}{\partial t} = \frac{1}{\tau}  \mathcal{F}\left( S, p \right):= \frac{1}{\tau}
\begin{cases}
 p_c^{(d)}\left( S \right) - p & \text{ if } (S,p)\in \Hdr,\\[-.1em]
 0 & \text{ if } (S,p)\in \H, \\[-.1em]
 p_c^{(i)}\left( S \right) - p & \text{ if } (S,p)\in \Him.
\end{cases}
\end{equation}

\subsubsection{Relative permeability}
To make the effect of hysteresis explicit in the relative permeabilities  
we need to incorporate a dependence on both $S$ and $\frac{\p S}{\p t}$. This dependence should satisfy
\begin{equation}
\label{2PF_eq:KraHys}
k_{r\a}\left(S,\frac{\partial S}{\partial t}\right) =
\begin{cases}
k^{(i)}_{r\a}\left( S \right) &\text{ if } \frac{\partial S}{\partial t} > 0, \\
k^{(d)}_{r\a}\left( S \right) &\text{ if } \frac{\partial S}{\partial t} < 0,
\end{cases} \text{ for }  \a\in\{w,n\}.
\end{equation}
Here $k^{(i)}_{r\a},k^{(d)}_{r\a}:[0,1]\to \R$ are the infiltration and drainage relative permeabilities 
obtained from experiments \cite{gladfelter1980effect,braun1995relative,killough1976reservoir,topp1966hysteretic,poulovassilis1970hysteresis}.
In line with the experimental outcomes, we assume here for $j\in \{i,d\}$,
\begin{enumerate}[label=(P\theTPFproperties)]
 \item $k^{(j)}_{rw} \in C^2\left( \left[0,1 \right] \right),\;{k^{(j)}_{rw}}^\prime\left( S \right)>0$ 
 for $0< S \leq 1,\;k^{(j)}_{rw}\left( 0 \right) = 0$
 and $k^{(j)}_{rw}$ is strictly convex. Moreover, for $0<S<1$, $k^{(i)}_{rw}(S)<k^{(d)}_{rw}(S)$.\stepcounter{TPFproperties}\label{2PF_P:Krw}
 \item $k^{(j)}_{rn} \in C^2\left( \left[0,1 \right] \right),\;{k^{(j)}_{rn}}^\prime\left( S \right)<0$ 
 for $0 \leq S < 1,\;k^{(j)}_{rn}\left( 1 \right) = 0$
 and $k_{rn}$ is strictly convex. Moreover, for $0<S<1$, $k^{(d)}_{rn}(S)<k^{(i)}_{rn}(S)$.\stepcounter{TPFproperties}\label{2PF_P:Krn}
\end{enumerate} 
Note the reverse ordering in $k_{rw}$ and $k_{rn}$ when switching from infiltration to drainage. 
This is demonstrated experimentally in \cite{braun1995relative,gladfelter1980effect,topp1966hysteretic}, see also \Cref{2PF_fig:Scanning}.
  
In \cite{zhuang2017advanced}, a play-type approach has been proposed to model $k_{r\a}$ where
\begin{align}
  k_{r\a}\in \frac{1}{2}( k_{r\a}^{(d)}(S)+ k_{r\a}^{(i)}(S)) 
  - \frac{1}{2}(k_{r\a}^{(d)}(S)-k_{r\a}^{(i)}(S))\cdot\sgn\left(\frac{\p S}{\p t}\right).\label{2PF_eq:PlaytypeForK}
\end{align}
However, this model is ill-posed in the unregularised case as for $\frac{\p S}{\p t}=0$ the 
relative permeabilities are  undetermined, i.e. the relative permeabilities have no equation to determine them when $\frac{\p S}{\p t}=0$. 
This is different for the capillary pressure \eqref{2PF_eq:Combined} 
because $p$ satisfies equation \eqref{2PF_eq:2phase2ndForm} as well. With the permeabilities we take an approach inspired by 
\cite{plohr2001modeling,schaerer2006permeability,bedrikovetsky1996mathematical}.
Here, inherited from the capillary pressure, the hysteresis is of the play-type as well, but now depending on 
$S$ and $p$, rather than on $S$ and $\frac{\p S}{\p t}$. We propose the following model: for $\a\in\{w,n\}$
\begin{equation}
k_{r\a}=k_{r\a}(S,p)=
\begin{cases}
k_{r\a}^{(d)}(S) &\text{  if } (S,p)\in \Hdr,\\
\bar{k}_{r\a}(S,p) &\text{  if } (S,p)\in \H,\\
k_{r\a}^{(i)}(S) &\text{  if } (S,p)\in \Him.
\end{cases}\label{2PF_eq:Perhys}
\end{equation}
Here $\bar{k}_{r\a}:\H\to [0,\infty)$ is a given function that satisfies
\begin{enumerate}[label=(P\theTPFproperties)]
 \item $\bar{k}_{r\a}\in C^2(\H)$ such that $k_{r\a}\in C(\mathcal{W})$ for $\a\in\{w,n\}$ and 
 $\p_p \bar{k}_{rw}>0$, $\p_p \bar{k}_{rn}<0$ in $\H$.\label{2PF_P:Zeta}\stepcounter{TPFproperties}
\end{enumerate}
Observe that, this implies $\bar{k}_{r\a}(S,p^{(j)}_c(S))=k_{r\a}^{(j)}(S)$ for $j\in \{i,d\}$. 
For the moment we leave the choice of $\bar{k}_{r\a}$ unspecified, except for properties \ref{2PF_P:Zeta}, 
as it neither influences the entropy conditions nor the critical $\t$ values introduced afterwards.

\begin{remark}
In the computations one needs to specify an expression for $\bar{k}_{r\a}$. In \Cref{2PF_sec:NumRes} we use
\begin{equation}
\bar{k}_{r\a}(S,p)=k_{r\a}^{(i)}(S) + (k_{r\a}^{(d)}(S)-k_{r\a}^{(i)}(S))\left(\frac{p- \Pim(S)}{\Pdr(S)-\Pim(S)}\right).
\label{2PF_eq:ExprsnKbar}
\end{equation}
This expression is obtained by considering $\sgn$ in \eqref{2PF_eq:Combined} and \eqref{2PF_eq:PlaytypeForK} 
as a function that can be eliminated. With $\t=0$ in \eqref{2PF_eq:Combined}, this results in \eqref{2PF_eq:ExprsnKbar}. 
Since the fraction \eqref{2PF_eq:ExprsnKbar} is bounded by 0 and 1 for $(S,p)\in \H$, we have 
$\lim_{S\searrow 0}k_{r\a}(S,p)=k_{r\a}^{(i)}(0)=k_{r\a}^{(d)}(0)$ and similar for $S\nearrow 1$.
\label{2PF_rem:Kbar}
\end{remark}
Observe that \eqref{2PF_eq:Perhys} is consistent with \eqref{2PF_eq:KraHys} as from \eqref{2PF_eq:Psi}, 
$\frac{\partial S}{\partial t}>0$ iff $p< \Pim(S)$ and $\frac{\partial S}{\partial t} < 0$ iff $p> \Pdr(S)$.
Moreover, the scanning curves for $k_{r\a}$ have constant $S$. Although  not true in general, see for instance \Cref{2PF_fig:Scanning}, 
we restrict ourselves to play-type for both $p$ and $k_{r\a}$. An extension describing non-vertical scanning curves is 
discussed in  \cite{Kmitra2017}.

Using \eqref{2PF_eq:Perhys} and \eqref{2PF_eq:Deff},\eqref{2PF_eq:Defh}, 
the nonlinearities $f$ and $h$ are expressed in terms of $S$ and $p$ as well:
\begin{equation}
\zeta(S,p)=
\begin{cases}
\zeta^{(d)}(S) &\text{  if } (S,p)\in \Hdr,\\
\bar{\zeta}(S,p) &\text{  if } (S,p)\in \H,\\
\zeta^{(i)}(S) &\text{  if } (S,p)\in \Him,
\end{cases}\; \text{ for } \zeta\in \{f,h\}.
\label{2PF_eq:fhhys}
\end{equation}
From \ref{2PF_P:Krw}-\ref{2PF_P:Zeta}  we deduce for $f$ and $h$:
\begin{enumerate}[label=(P\theTPFproperties)]
 \item $f\in C(\mathcal{W})$, $\bar{f}\in C^2(\H)$ and $\p_p \bar{f}>0$ in $\H$. 
 For $j\in \{i,d\}$, $f^{(j)}\in C^2\left( \left[0,1 \right] \right)$, ${f^{(j)}}^\prime(S)>0$ for 
 $0<S < 1,\;f^{(j)}\left( 0 \right) = 0,\;f^{(j)}\left( 1 \right) = 1$. 
 Moreover, for $0<S<1$, $f^{(i)}(S)<f^{(d)}(S)$.\label{2PF_P:f}\stepcounter{TPFproperties}
 \item $h\in C(\mathcal{W})$, $\bar{h}\in C^2(\H)$,  $h^{(j)}\in C^2\left( \left[0,1 \right] \right)$ and 
 $h^{(j)}(0)=h^{(j)}(1)=0$ for $j\in \{i,d\}$.\label{2PF_P:h}\stepcounter{TPFproperties}
\end{enumerate}
Observe that, in general no ordering holds between $h^{(i)}$ and $h^{(d)}$. 
Typical curves for $f^{(j)}$ and $h^{(j)}$ are shown in  \Cref{2PF_fig:fhF}.
The equations \eqref{2PF_eq:2phase2ndForm}, \eqref{2PF_eq:Combined} and \eqref{2PF_eq:fhhys} are the complete set of equations for our model.
 
\subsection{Dimensionless formulation}
Let $H\; \left[\unit{m} \right]$ be the characteristic length, $p_r\; \left[\unit{Pa} \right]$ the characteristic pressure,
$t_r=\;\frac{\phi H}{v} \left[\unit{s} \right]$ the characteristic time and $\t_r\;\left[\unit{Pa\cdot s} \right]$ 
the characteristic dynamic capillary constant. Inspired by the J-Leverett model \cite{leverett1941capillary}, 
we take as characteristic pressure $p_r=\sigma\sqrt{\frac{\phi}{K}}$, $\sigma$ being the surface tension between the two phases. 
Alternatively, one could consider $p_r=(\rho_n-\rho_w)g H$ which is a more common choice for the Richards equation with gravity. Setting
$$
\tilde{z} := \frac{z}{H},\quad \tilde{t} := \frac{t}{t_r},\quad \tilde{\psi} := \frac{\psi}{p_r} \text{ and } \tilde{\t} = \frac{\tau}{\tau_r}
$$
where $\psi\in \{p,\Pim,\Pdr\}$, and defining the dimensionless numbers
$$
N_g := \frac{K\left(\rho_w - \rho_n \right) g}{v \mu_n}\;\text{(gravity number) and }
N_c : = \frac{K p_r}{v \mu_n H}\;\text{(capillary number)},
$$
we obtain  from \eqref{2PF_eq:2phase2ndForm} the dimensionless transport equation
\begin{equation}
\frac{\partial S}{\partial \tilde{t}} +  \frac{\partial}{\partial \tilde{z}}\left( f + N_g h 
+ N_c h \frac{\partial \tilde{p}}{\partial \tilde{z}} \right) = 0.\label{2PF_eq:dimlessTE}
\end{equation}
The closure relation \eqref{2PF_eq:Combined} becomes
\begin{equation}
\label{2PF_eq:dimlessCl}
\tilde{p}\in \frac{1}{2}(\tilde{p}_c^{(i)}(S)+\tilde{p}_c^{(d)}(S))-\frac{1}{2}(\tilde{p}_c^{(d)}(S)
-\tilde{p}_c^{(i)})\cdot\sgn\left(\frac{\p S}{\p \tilde{t}}\right)-\tilde{\t}\dfrac{ \t_r}{p_r t_r} \frac{\p S}{\p \tilde{t}}.
\end{equation}
Now choosing $\t_r=N_c p_r t_r=p_r^2\frac{\phi K}{v^2 \mu_n}$,  the Leverett scaling for $p_r$ gives
\begin{equation}
\t_r=\frac{\sigma^2\phi^2}{\mu_n v^2} \text{ implying } \tilde{\t}=\dfrac{\mu_n v^2}{\sigma^2 \phi^2}\t. \label{2PF_eq:ScalingOfTau}
\end{equation}
This choice leaves us with a characteristic dynamic coefficient that is independent of the length scale of the problem. 
This is precisely the scaling used in \cite{van2007new,van2013travelling,graf2018vanishing} that is consistent with the hyperbolic limit.
Realistic values of dimensional and scaled quantities are given in \cite{manthey2008dimensional}. 

Dropping the $\tilde{\phantom{a}}$ sign from the notation, we are left with the dimensionless system
\begin{subequations}\label{2PF_eq:FullRegularisedSystem}
\begin{numcases}{(\mathcal{P})}
&$\dfrac{\partial S}{\partial t} +  \dfrac{\partial}{\partial z}\left( F(S,p) + 
N_c h(S,p) \dfrac{\partial p}{\partial z} \right)=0$, \label{2PF_eq:TE}\\[.5em]
&$\dfrac{\partial S}{\partial t} =  \dfrac{1}{N_c\tau} \mathcal{F}\left( S, p \right)$,\label{2PF_eq:hyst_psi}
\end{numcases}
\begin{equation}
\text{where  } F= f+ N_g h.\label{2PF_eq:DefF}
\end{equation}
\end{subequations}
This system can be seen as a regularisation of the hyperbolic Buckley-Leverett equation with gravity. 
Here the regularisation involves hysteresis and dynamic capillarity. Compared to the usual second order parabolic regularisation, 
yielding shocks that satisfy the standard Oleinik conditions \cite{oleinik1957discontinuous}, different (non-parabolic) regularisations 
may yield shocks that violate these conditions, see e.g. \cite{van2007new,lefloch2002hyperbolic}. Such shocks are called non-classical.

One of the main issues of this paper is to show the existence of non-classical shocks originating from System $(\mathcal{P})$. 
To this end we proceed as in \cite{van2007new} and study the existence of travelling wave (TW) solutions of $(\mathcal{P})$ 
that connect a left state $S_B$ to a right state $S_T$ in the presence of both hysteresis and dynamic capillarity.
Travelling waves for the model with only dynamic capillarity are analysed in \cite{van2013travelling,spayd2011buckley}. 
For the case of unsaturated flow, i.e. Richards equation with a convex flux function, existence and qualitative properties 
of travelling waves are considered in detail in \cite{cuesta2000infiltration,VANDUIJN2018232,mitra2018wetting}.

For the purpose of travelling waves we consider System $(\mathcal{P})$ in the domain $-\infty<z<\infty$. 
Then the capillary number $N_c$ can be removed from the problem by the scaling
$$
z:=z\slash N_c \text{ and  } t:=t\slash N_c.
$$
This yields the $N_c$ independent formulation
\begin{subequations}\label{2PF_eq:SystemPtilde}
\begin{numcases}{(\tilde{\mathcal{P}})}
&$\dfrac{\partial S}{\partial t} +  \dfrac{\partial}{\partial z}\left( F(S,p) + h(S,p) \dfrac{\partial p}{\partial z} \right)=0,$ 
\label{2PF_eq:TE2}\\[.5em]
&$\dfrac{\partial S}{\partial t} =  \dfrac{1}{\tau} \mathcal{F}\left( S, p \right)$,\label{2PF_eq:hyst_psi2}
\end{numcases}
\end{subequations}
with $-\infty<z<\infty$ and $t>0$. This is the starting point for the TW analysis.

\begin{remark}\label{2PF_rem:fhF}
Using the Brooks-Corey type expression, e.g. see \cite{brooks1966properties}, 
\begin{equation}
k_{rw}(S)=S^q \text{ and } k_{rn}(S)=(1-S)^q,
\end{equation}
with $q=2$, the nonlinearities \eqref{2PF_eq:Deff}, \eqref{2PF_eq:Defh} and  \eqref{2PF_eq:DefF} become
$$
f(S)=\frac{S^2}{S^2+ M (1-S)^2},\;h(S)=(1-S)^2 f(S),\; F(S)=S^2\frac{(1+ N_g (1-S)^2)}{S^2+ M (1-S)^2},
$$ 
where $M=\frac{\mu_w}{\mu_n}$ denotes the viscosity ratio. 
A plot is shown in \Cref{2PF_fig:fhF}. Some elementary calculations give
\begin{itemize}
\item[$(a)$] \emph{Monotonicity:} If $N_g\leq M$ then $F'(S)>0$ for all $0<S<1$ and if $N_g>M$ 
then there exists a unique $S_F\in (0,1)$ such that $F'(S)>0$ for all $0<S<S_F$ and $F'(S)<0$ for $S_F<S<1$. Since $F(1)=1$, clearly $F(S_F)>1$.
\item[$(b)$] \emph{Inflection points:} $f(S)$ has only one inflection point in $(0,1)$ 
whereas, $F(S)$ has at most two. To see this for $f(S)$, note that $f''(S)=P(S)Q(S)$ with $Q(S)$ being a positive function 
and $P(S)=M -(3M+3)S^2 + (2M+2)S^3$. Since $P(0)=M$, $P(1)=-1$ and $P'(S)<0$ for $S\in (0,1)$, the result follows.
\end{itemize}
These properties of $f$ and $F$ will be used when discussing the different cases of travelling waves.
\end{remark}

\begin{figure}[H]
\begin{center}
\includegraphics[width=0.5\textwidth]{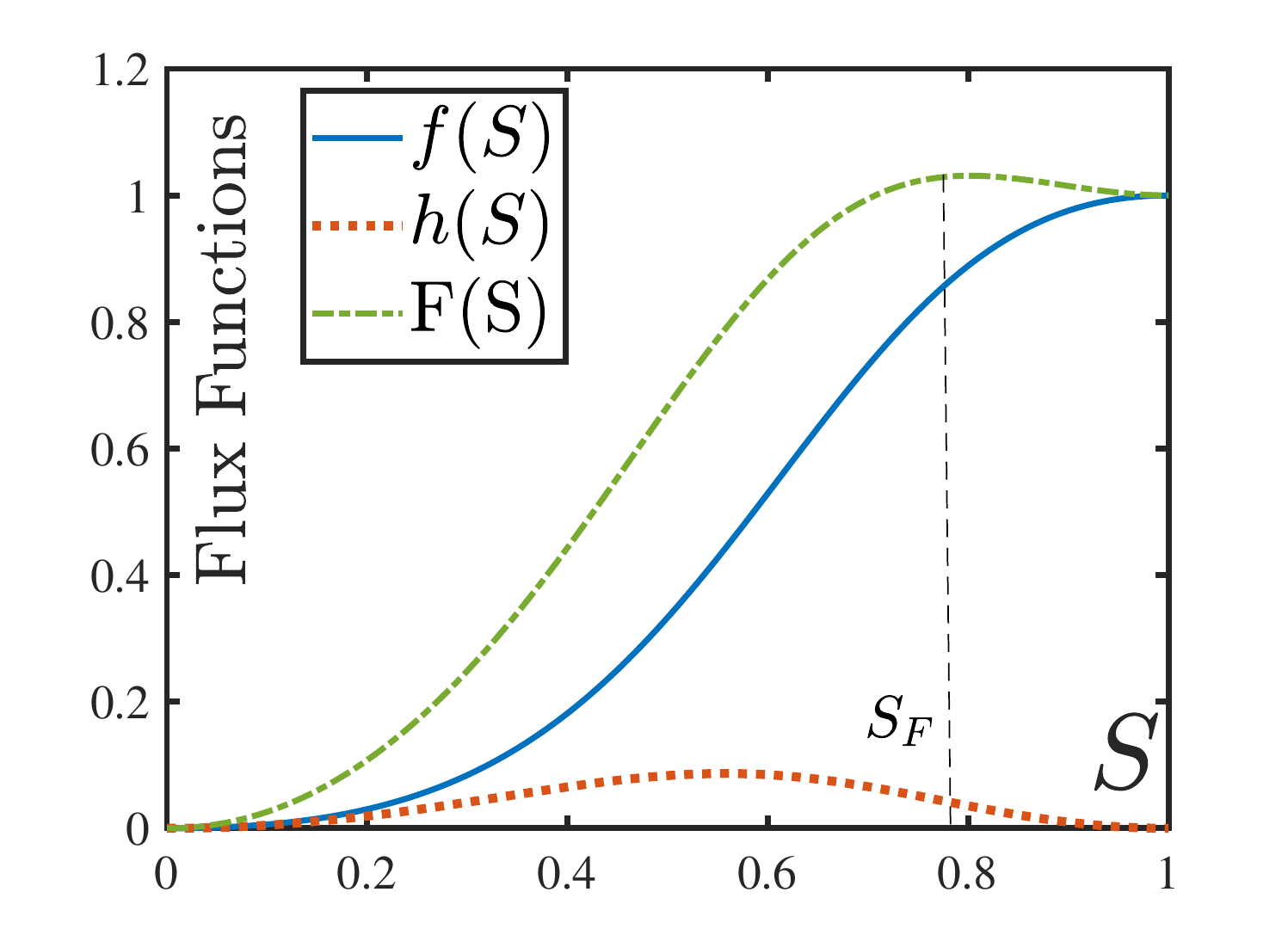}
\end{center}
\caption{\label{2PF_fig:fhF} The functions $f(S),\;h(S)$ and $F(S)$ as given in Remark \ref{2PF_rem:fhF}. Here $M=2$ and $N_g=4$.}
\end{figure}

\subsection{Travelling wave formulation}
\label{2PF_sec:TWderivation}

Having derived the non-dimensional hysteretic two-phase flow System $(\tilde{\mathcal{P}})$, we investigate under 
which conditions travelling wave solutions exist. These are solutions of the form
$$
S(z,t)=S(\xi),\; p(z,t)=p(\xi),\; \text{ with } \xi=ct-z,\label{2PF_eq:ansatz}
$$
where $S$ and $p$ are the wave profiles of saturation and pressure and $c\in \R$ the wave-speed. We seek travelling waves that satisfy
\begin{align}
\begin{cases}
\lim\limits_{\xi\to -\infty}S(\xi)=S_B,\; \lim\limits_{\xi\to \infty}S(\xi)=S_T,\\
\lim\limits_{\xi\to -\infty}p'(\xi)=\lim\limits_{\xi\to \infty}p'(\xi)=0,\label{2PF_eq:BC1}
\end{cases}
\end{align}
where $S_B$ corresponds to an `initial' saturation and $S_T$ to the injected saturation. The choice of $p'(\pm\infty)=0$ ensures that 
the diffusive flux vanishes at $\xi=\pm\infty$.
Substituting
\eqref{2PF_eq:BC1} into \eqref{2PF_eq:TE2} and \eqref{2PF_eq:hyst_psi2}, and integrating \eqref{2PF_eq:TE2} one obtains
\begin{subequations}\label{2PF_eq:ODEsys}
\begin{align}
\label{2PF_eq:ODES2}
&cS - \left( F(S,p) - h(S,p)p^\prime \right) = A, \\
\label{2PF_eq:ODES1}
&cS^\prime =  \frac{1}{\tau} \mathcal{F}\left( S, p \right), 
\end{align}
\end{subequations}
where $\xi \in \mathbb{R}$ and $A$ is a constant of integration. 

As was shown in \cite{VANDUIJN2018232} for the Richards equation, \eqref{2PF_eq:BC1} and \eqref{2PF_eq:ODEsys} do not automatically 
guarantee the existence of $\lim_{\xi\to \pm\infty} p(\xi)$. But if $p(\pm\infty)$ is well-defined then \eqref{2PF_eq:ODES1} and the 
existence of $S(\pm\infty)$ forces $\lim_{\xi\to \pm\infty}\F(S(\xi),p(\xi))=0$. Recalling that $\F(S,p)=0$ iff $(S,p)\in\H$ we then have 
\begin{equation*}
\lim\limits_{\xi\to -\infty}p(\xi)=p_B\in [\Pim(S_B),\Pdr(S_B)],\; \lim\limits_{\xi\to \infty}p(\xi)=p_T\in [\Pim(S_T),\Pdr(S_T)].
\end{equation*}
We show later that $p_B$, interpreted as the initial pressure,  can sometimes be chosen independently, whereas, $p_T$, when existing, 
is always fixed by the choice of $S_B,S_T$ and $p_B$.
Following the steps in \cite{van2013travelling,VANDUIJN2018232,mitra2018wetting}  we obtain the Rankine-Hugoniot condition for wave-speed $c$, i.e.
\begin{equation}
c = \frac{F(S_T,p_T) - F(S_B,p_B)}{S_T-S_B}. \label{2PF_eq:WavespeedRH}
\end{equation}
With this, system \eqref{2PF_eq:ODEsys} can be rewritten as a dynamical system,
\begin{subequations}\label{2PF_eq:ODE}
\begin{numcases}{\mathrm{(TW)}}
&$S^\prime = \dfrac{1}{c\tau} \mathcal{F} \left( S, p \right)$, \label{2PF_eq:ODES} \\[.5em]
&$p^\prime = \mathcal{G}\left(S,p \right)$.\label{2PF_eq:ODEp}
\end{numcases}
\end{subequations}
where
\begin{equation}
\mathcal{G}\left(S,p\right) := \frac{  F(S,p)- \ell(S)  }{h\left( S,p \right)} \text{ with } \ell(S):=F(S_B,p_B) +c(S-S_B).\label{2PF_eq:DefG}
\end{equation}
Note that when $F$ is non-monotone (e.g. $N_g>M$ in \Cref{2PF_rem:fhF}), the wave-speed $c$ can be positive or 
negative depending on the values of $S_B$ and $S_T$.

We study all possible solutions of system \hyperref[2PF_eq:ODE]{(TW)} for $\t>0$. 
They serve as viscous profiles of admissible shocks of the limiting Buckley-Leverett equation.
Existence conditions for solutions of \hyperref[2PF_eq:ODE]{(TW)} act as admissibility/entropy conditions for the corresponding shocks.

\noindent
The solutions of \hyperref[2PF_eq:ODE]{(TW)} are investigated under three different scenarios.
\begin{itemize}
\item[A:] No hysteresis in relative permeabilities, i.e. $\zeta^{(i)}=\zeta^{(d)}$ for $\zeta\in \{f,h\}$. 
Furthermore, $N_g$ is sufficiently small so that $F$ satisfies properties stated for $f^{(j)}$ in \ref{2PF_P:f}. 
For $F$ as in Remark \ref{2PF_rem:fhF} this is satisfied if $N_g\leq M$. 
 \item[B:]  $N_g$ and $\t$ sufficiently small; relative permeabilities are hysteretic. 
\end{itemize}
A third scenario where $N_g$ is large so that $F$ is non-monotone is discussed briefly at the end of \Cref{2PF_sec:CaseA}.

\section{No  relative permeability hysteresis and small $N_g$ (Scenario A)}
\label{2PF_sec:CaseA}
In the absence of relative permeability hysteresis, the functions $f,\;h,\;F$ and $\Gf$ depend on $S$ only.  
We explicitly state the properties of $F$ as a result of \ref{2PF_P:f}, \ref{2PF_P:h} and \Cref{2PF_rem:fhF}.

\begin{enumerate}[label=(A\theTPFassumption)]
 \item $F\in C^2\left([0,1]\right),\;{F}^\prime(S)>0$ for $0<S < 1,\;F(0) = 0,\;F\left( 1 \right) = 1$. 
 Moreover, a unique $S_o\in (0,1)$ exists such that \label{2PF_ass:F1}
$$
F''(S_o)=0,\; F''(S)>0 \text{ for } 0<S<S_o \text{ and } F''(S)<0 \text{ for } S_o<S<1.
$$ \stepcounter{TPFassumption}
\end{enumerate}

\begin{figure}[H]
\centering
\includegraphics[width=0.95\textwidth]{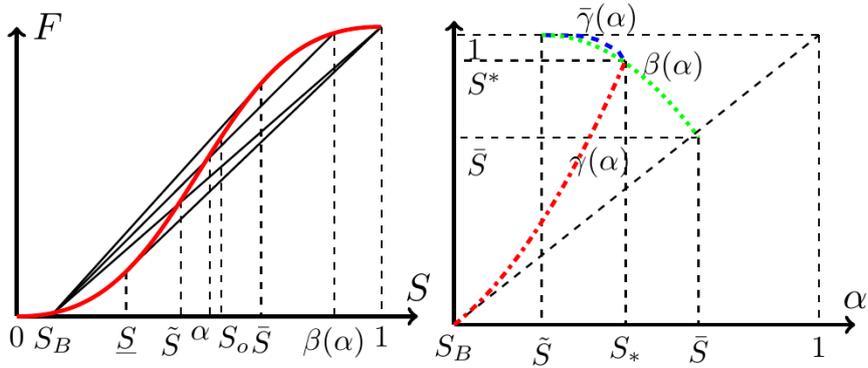}
\caption{(left) The saturations $S_B,\;\underline{S},\;\bar{S},\;\tilde{S},\;\a$ and $\b(\a)$ for Scenario A. (right) 
The functions $\b$, $\g$ and $\bar{\g}$ (assuming \eqref{2PF_eq:GfnotIntegrableAt1}) and the definitions of $S_*$ and $S^*$ 
for $S_B<\underline{S}$.}\label{2PF_fig:Sdefs}
\end{figure}

\subsection{Preliminaries}
\label{2PF_sec:prelim}
Throughout this paper we restrict ourselves to relatively small values of $S_B$. Specifically, we assume
\begin{equation}
0<S_B<S_o.\label{2PF_eq:SbSt}
\end{equation}

First let us take $S_B\leq \underline{S}$, where $\underline{S}$ is the saturation at which 
$F'(\underline{S})=\frac{1-F(\underline{S})}{1-\underline{S}}$. The convex-concave behaviour of $F$ implies $\underline{S}<S_o$.
For later purpose, and with reference to \Cref{2PF_fig:Sdefs} (left), we introduce the additional saturations 
$S_B<\tilde{S}<\bar{S}<1$, where $\tilde{S}$ is the saturation at which $F(S)$ intersects the line connecting 
$(S_B,F(S_B))$ and $(1,1)$, and where $\bar{S}$ is the saturation for which $F'(\bar{S})=\frac{F(\bar{S})-F(S_B)}{\bar{S}-S_B}$.  
Then to each $\a\in [\tilde{S},\bar{S}]$ corresponds a unique $\b\in [\bar{S},1]$ such that $(\b,F(\b))$ is the third 
intersection point between the graph of $F$ and the chord through $(S_B,F(S_B))$ and $(\a,F(\a))$, see \Cref{2PF_fig:Sdefs} (left). 
This defines the function 
\begin{align}
\begin{cases}
\b: [\tilde{S},\bar{S}]\to [\bar{S},1],\; \b(\bar{S})=\bar{S},\; \b(\tilde{S})=1,\\
\b(\a) \text{ is strictly decreasing}. 
\end{cases}
\end{align}
Later in this section a second function $\g=\g(\a)$ is introduced as one of the roots of the equation
\begin{equation}
\int_{S_B}^{\g(\a)}\Gf(S;S_B,\a)dS=0 \text{ for } S_B\leq \a\leq \bar{S}.\label{2PF_eq:DefGamma}
\end{equation}
Here $\Gf(S;S_B,\a)$ is the expanded notation of $\Gf$ from \eqref{2PF_eq:DefG} for the $p$ independent case:
$$ \Gf(S;S_B,\a)=\tfrac{F(S)-\ell(S;S_B,\a)}{h(S)} \text{ with } \ell(S;S_B,\a)=F(S_B)+ \tfrac{F(\a)-F(S_B)}{\a-S_B}(S-S_B).$$ 
A typical sketch of $\Gf(S;S_B,\a)$ for different values of $\a$ is shown in \Cref{2PF_fig:Gf}. Note that
\begin{align}
\Gf(S;S_B,\a) \text{ decreases with respect to } \a\in [S_B,\bar{S}] \text{ and }\label{2PF_eq:GfisAlphaDecreasing}\\
\Gf(S;S_B,\a)
\begin{cases}
\begin{rcases}
<0 &\text{ for } S_B<S<\a\\
>0 &\text{ for } S>\a
\end{rcases} &\text{ when } S_B<\a<\tilde{S},\\
\begin{rcases}
<0 &\text{ for } S_B<S<\a\\
>0 &\text{ for } \a<S<\b(\a)\\
<0 &\text{ for } \b(\a)<S<1
\end{rcases} &\text{ when } \tilde{S}<\a<\bar{S}.\\
\end{cases}\label{2PF_eq:PropGfsign}
\end{align}
Since,
\begin{align}
\Gf(S;S_B,\a)=\begin{cases}
\mathcal{O}(\tfrac{1}{k_{rn}(S)}) &\text{ when } \a\not= \tilde{S},\\[.2em]
\mathcal{O}(\tfrac{1-S}{k_{rn}(S)}) &\text{ when } \a= \tilde{S},
\end{cases}\label{2PF_eq:GfBehaveNear1}
\end{align}
as $S\nearrow 1$, we have for most practical applications 
\begin{align}
\Gf(S;S_B,\a) \text{ is non-integrable near $S=1$ for each } S_B\leq \a\leq \bar{S}.\label{2PF_eq:GfnotIntegrableAt1}
\end{align}
This is the case for Brooks-Corey permeabilities with $q\geq 2$, see \Cref{2PF_rem:fhF}.
\begin{figure}[H]
\begin{center}
\includegraphics[width=0.95\textwidth]{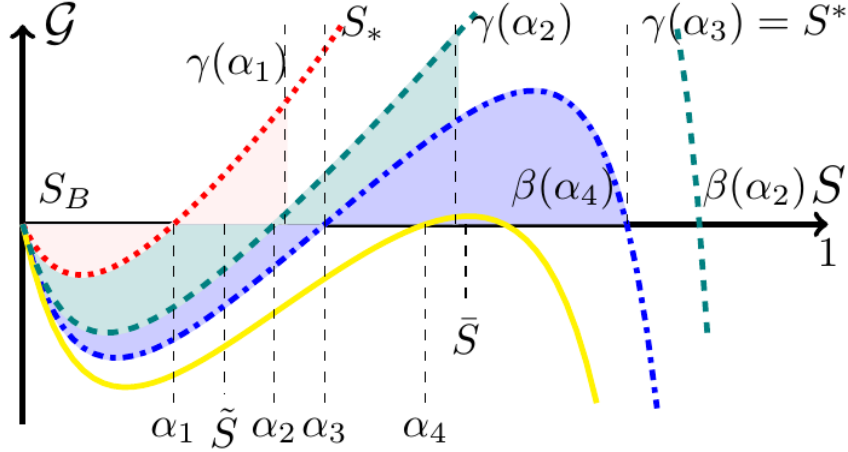}
\caption{The plot of $\Gf(S;S_B,\a)$ for different values of $\a$. Here $\a_1<\tilde{S}<\a_2<\a_3=S_*<\alpha_4<\bar{S}<S^*$. 
Values of $\g$ and $\b$ are also shown. Note that, $\b(\a_1)$ and $\g(\a_4)$ do not exist.}\label{2PF_fig:Gf}
\end{center}
\end{figure}

Returning to equation \eqref{2PF_eq:DefGamma}, we note that $\g=S_B$ is the trivial solution. Properties 
\eqref{2PF_eq:PropGfsign} and \eqref{2PF_eq:GfBehaveNear1} imply the existence of a second (non-trivial) 
solution $\g=\g(\a)$ for $\a\geq S_B$. It satisfies $\g(S_B)=S_B$, $\g(\a)$ increases, $\g(\a)>\a$ for $\a>S_B$. 
Moreover, if \eqref{2PF_eq:GfnotIntegrableAt1} is satisfied then $\g(\tilde{S})<1$. This shows the existence of 
$\g(\a)$ in a right neighbourhood of $S=\tilde{S}$.  The solution in this case exists up to $\a=S_*\in (\tilde{S},\bar{S})$ 
where $\g(\a)$ and $\b(\a)$ intersect: $\g(S_*)=\b(S_*)=:S^*$. Further, if  
\eqref{2PF_eq:GfnotIntegrableAt1} holds, then a third solution $\g=\bar{\g}$ exists for $\tilde{S}<\a<S_*$. 
It decreases in $\a$ with $\bar{\gamma}(\tilde{S})=1$ and $\bar{\gamma}(S_*)=S^*$. When  
\eqref{2PF_eq:GfnotIntegrableAt1} is not satisfied, the existence of $S_*$ and a third solution depends on the specific form of 
$k_{rn}(S)$. The solutions of \eqref{2PF_eq:DefGamma} and the function $\b(\a)$ are sketched in \Cref{2PF_fig:Sdefs} (right). 

For $S_B\in (\underline{S},S_o)$, $\b(\a)$ and $\g(\a)$ can similarly be defined, although the domain where $\b(\a)$ is defined 
is different. In this case the intersection of $\b(\a)$ and the second solution $\g(\a)$ is guaranteed irrespective 
of \eqref{2PF_eq:GfnotIntegrableAt1} since $\int^{\bar{S}}_{S_B} \Gf(S;S_B,\bar{S})dS<0$ and  $\int^{1}_{S_B} \Gf(S;S_B,1)dS>0$. 
This is because $\Gf(S;S_B,\bar{S})<0$ for $S_B<S<\bar{S}$ and $\Gf(S;S_B,1)>0$ for $S_B<S<1$.
Since we use the second solution $\g=\g(\a)$ only, we summarize its properties in the following proposition.

\begin{proposition}
Assume either \eqref{2PF_eq:GfnotIntegrableAt1} or $S_B\in(\underline{S},S_o)$. 
Let $\g$ be the increasing (unique) solution  of \eqref{2PF_eq:DefGamma}. Then it is defined in the interval $[S_B,S_*]$ 
where $S_*\in (\tilde{S},\bar{S})$ is such that $\g(S_*)=\b(S_*)=: S^*$. Further, $\g(S_B)=S_B$, $\g(\a)>\a$ for $\a>S_B$ 
and $\g(\a)<\b(\a)$ for $\a<S_*$ in the common domain of definition of $\b$ and $\g$.\label{2PF_prop:PropertiesGammaFunc}
\end{proposition}

\begin{remark}
For simplicity, we assume \eqref{2PF_eq:GfnotIntegrableAt1} for the rest of the discussion. 
This guarantees the existence of a $(S_*,S^*)$ pair. The methods presented in this paper can also be applied to analyse 
the case when $\b(\a)$ and $\g(\a)$ are not intersecting. The results are briefly discussed  in \Cref{2PF_sec:Ap1}.\label{2PF_rem:BetaGamma}
\end{remark}

Next we turn to system \hyperref[2PF_eq:ODE]{(TW)} where, for the time being, we take
\begin{equation}
S_B<S_T\leq \bar{S}.\label{2PF_eq:SbSt2}
\end{equation}
Since \hyperref[2PF_eq:ODE]{(TW)} is autonomous, it is convenient to represent solutions as orbits in the $(S,p)$-plane, 
or rather, in the strip $\{(S,p): 0\leq S\leq 1,\; p\in\R\}$. Moreover, orbits are same for any shift in the independent variable $\xi$. 
Therefore we may set without loss of generality, see  \cite{VANDUIJN2018232,mitra2018wetting},
\begin{equation}
S(0)=\frac{1}{2}(S_B+S_T).\label{2PF_eq:DefOfzeta0}
\end{equation}
\begin{figure}[H]
\begin{center}
\includegraphics[width=0.75\textwidth]{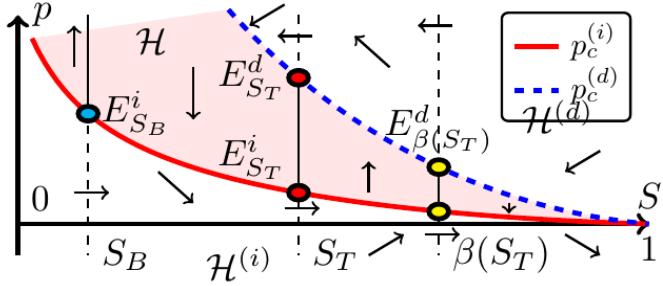}
\caption{The $S$-$p$ phase plane and the direction of orbits for  Scenario A with $\tilde{S}\leq S_T\leq \bar{S}$. 
The regions $\H,\;\Him,\;\Hdr$ and the equilibrium lines are marked. }\label{2PF_fig:OrbitDirection}
\end{center}
\end{figure}
Equilibrium points of \hyperref[2PF_eq:ODE]{(TW)} are
\begin{align*}
&E_{K}^{j}\equiv (K,p^{(j)}_c(K)), \text{ where } K\in \{S_B,S_T\} \text{ and } j\in \{i,d\}.
\end{align*}
If $\tilde{S}\leq S_T<\bar{S}$, a third pair exists for $K=\b(S_T)$. The points $E^j_K$ and the direction of the orbits
are indicated in \Cref{2PF_fig:OrbitDirection}. By the special nature of the function $\mathcal{F}$, we have in fact that all 
points of the segments $\overline{E^i_K E^d_K}$ are equilibrium points. Boundary conditions \eqref{2PF_eq:BC1} are satisfied 
if an orbit connects the segments $\overline{E^i_{S_B} E^d_{S_B}}$ and $\overline{E^i_{S_T} E^d_{S_T}}$. 
As shown in \cite{mitra2018wetting}, an orbit can leave $\overline{E^i_{S_B} E^i_{S_B}}$ only from the lowest point $E^i_{S_B}$. 
Then it enters region $\Him$ where it moves monotonically with respect to $S$ as a consequence of the sign in the right hand 
side of equation \eqref{2PF_eq:ODES}: if $p<\Pim(S)$ we have $S'>0$.

Due to this monotonicity one can alternatively describe an orbit leaving $E^i_{S_B}$ as a function of the 
saturation as long as it belongs to $\Him$. For given $\t>0$ and $S_T$ satisfying \eqref{2PF_eq:SbSt2}, let $w(S)=w(S;\t,S_T)$ 
denote this function. Then
\begin{subequations}
\begin{align}
&w(S_B;\t,S_T)=\Pim(S_B)\label{2PF_eq:WinitialValue}\\
&\text{ and } w(S;\t,S_T)<\Pim(S) \text{ in a right neighbourhood of } S_B \label{2PF_eq:WbelowPim}.
\end{align}
\end{subequations}
As in \cite{VANDUIJN2018232,mitra2018wetting}, we deduce from \hyperref[2PF_eq:ODE]{(TW)} that $w$ should satisfy
\begin{equation}
 w'(S;\t, S_T)=\frac{c\t\Gf(S;S_B,S_T)}{\Pim(S)-w(S;\t, S_T)} \text{ for } S>S_B.\label{2PF_eq:fWtauGDE}
 \end{equation}
Using techniques from \cite{VANDUIJN2018232,mitra2018wetting}, one can show that initial value problem 
\eqref{2PF_eq:fWtauGDE}, \eqref{2PF_eq:WinitialValue} has a unique local solution $w(S;\t, S_T)$  that satisfies \eqref{2PF_eq:WbelowPim}.
\begin{remark}
Conversely one recovers the orbit $(S(\xi),p(\xi))$ by substituting $w$ into \eqref{2PF_eq:ODES}. Using \eqref{2PF_eq:DefOfzeta0} this gives
$$
\xi=c\t \int^{S(\xi)}_{\tfrac{1}{2}(S_B+S_T)} \frac{d\vr}{\Pim(\vr)-w(\vr;\t,S_T)} d\vr \text{ and } p(\xi)=w(S(\xi);\t,S_T).
$$
\end{remark}
Rewriting \eqref{2PF_eq:fWtauGDE} as
$$
(\Pim-w)(w-\Pim)'+ (\Pim-w)\Pim'=c\t \Gf(S;S_B,S_T),
$$
we find recalling \ref{2PF_P:Pc} that
\begin{equation}
((\Pim-w)^2)'=2(\Pim-w) \Pim'- 2c\t \Gf(S;S_B,S_T)\leq -2c\t \Gf(S;S_B,S_T) \text{ in } \{w<\Pim\}.\label{2PF_eq:VboundPointWise}
\end{equation}
Integrating this inequality from $S_B$ to $S$ gives the lower bound
\begin{align}
\begin{rcases}
&w(S;\t,S_T)>\Pim(S)-\sqrt{2c\t \Phi(S)} \text{ in } \{w<\Pim\},\\
&\text{ where } \Phi(S)=\Phi(S;S_B,S_T):=-\int_{S_B}^S \Gf(\vr;S_B,S_T)d\vr.
\end{rcases}\label{2PF_eq:WlowerBoundDefPhi}
\end{align}
With $S_T$ satisfying \eqref{2PF_eq:SbSt2}, properties \eqref{2PF_eq:PropGfsign}-\eqref{2PF_eq:GfnotIntegrableAt1} 
and \Cref{2PF_prop:PropertiesGammaFunc} imply
\begin{subequations}\label{2PF_eq:PropOfPhi}
\begin{align}
\begin{rcases}
&\Phi(S)>0 \text{ for } S_B<S<\g(S_T)\\
&\Phi(S_B)=\Phi(\g(S_T))=0
\end{rcases} &\text{ when } S_B<S_T\leq S_*,\label{2PF_eq:PropOfPhi1}\\
\text{ and }
\begin{rcases}
&\Phi(S)>0 \text{ for } S_B<S<1\\
&\lim\limits_{S\nearrow 1}\Phi(S)=+\infty
\end{rcases} &\text{ when } S_*<S_T\leq \bar{S}.\label{2PF_eq:PropOfPhi2}
\end{align}
\end{subequations}

Observe that, depending on $S_B$, $S_T$ and $\t$, the interval where $w(S)<\Pim(S)$ is either $(S_B,1]$ if $w(S)$ and $\Pim(S)$ do not intersect, or $(S_B,S_i)$ with $S_i\leq 1$ in case there is an 
intersection at $S=S_i$. In the latter case, it follows immediately from
\eqref{2PF_eq:VboundPointWise} that we must have,
\begin{proposition}
Suppose there exists $S_i\in (S_B,1)$ such that $w(S)< \Pim(S)$ for $S_B<S<S_i$ and $w(S_i)=\Pim(S_i)$. 
Then $\Gf(S_i;S_B,S_T)\geq 0$.\label{2PF_prop:WequalsPimWhenGfpos}
\end{proposition}

Hence, if the orbit exits through the capillary pressure curve $\Pim$, it can only do so at points where $\Gf\geq 0$. 
From the discussion above, one defines
\begin{equation}
S_m(\t,S_T)=\sup\{S\in (S_B,1): w(\vr;\t,S_T)<\Pim(\vr) \text{ for all } S_B<\vr<S\},
\end{equation}
which is the upper limit of the interval on which $w$ exists. Then we have 
 \begin{proposition}
\begin{itemize}
\item[(a)]  If $S_B<S_T\leq S_*$, then $S_T\leq S_m(\t,S_T)<\g(S_T)$ for all $\t>0$;
\item[(b)] If $S_*<S_T\leq \bar{S}$ and $w(\b(S_T);\t,S_T)< \Pim(\b(S_T))$, then $S_m(\t,S_T)=1$ and $\lim\limits_{S\nearrow 1}w(S)=-\infty$.
\end{itemize}\label{2PF_prop:PropertiesOfSm}
 \end{proposition}
\begin{proof}
\emph{(a)} The lower bound follows from \Cref{2PF_prop:WequalsPimWhenGfpos}. 
To show the upper bound, observe that if $S_m(\t,S_T)\geq \g(S_T)$, then $w(\g(S_T))\leq  \Pim(\g(S_T))$. 
This directly contradicts the strict inequality in \eqref{2PF_eq:WlowerBoundDefPhi} since $\Phi(\g(S_T))=0$.  

\emph{(b)} Since $\Gf(\cdot;S_B,S_T)<0$ in $(\b(S_T),1)$, \Cref{2PF_prop:WequalsPimWhenGfpos} 
and \eqref{2PF_eq:WlowerBoundDefPhi}, \eqref{2PF_eq:PropOfPhi1} imply $S_m(\t,S_T)=1$. Since \eqref{2PF_eq:WlowerBoundDefPhi} 
holds for all $S<1$ and since $w'<0$ in a left neighbourhood of $S=1$, let us suppose that $\lim_{S\nearrow 1} w(S;\t,S_T)=-L$ ($L>0$). 
Then equation \eqref{2PF_eq:fWtauGDE} and property \eqref{2PF_eq:GfnotIntegrableAt1} give $w'\not\in L^1$ near $S=1$, 
contradicting the boundedness of $w$.\qed
\end{proof}
In \Cref{2PF_fig:Prop2Cases} we sketch the behaviour of $w(S;\t,S_T)$ in $\Him$. 
The existence of orbits as in \Cref{2PF_fig:Prop2Cases} (left) is a direct consequence of the behaviour of the 
lower bound \eqref{2PF_eq:WlowerBoundDefPhi}. Orbits as in \Cref{2PF_fig:Prop2Cases} (right)
 need more attention since the case $S_m(\t,S_T)<\b(S_T)$, represented by $\t_3$, remains to be discussed.  
 We make the behaviour as sketched in \Cref{2PF_fig:Prop2Cases} (right) precise in a number of steps. 
\begin{figure}[H]
\centering
\includegraphics[width=0.85\textwidth]{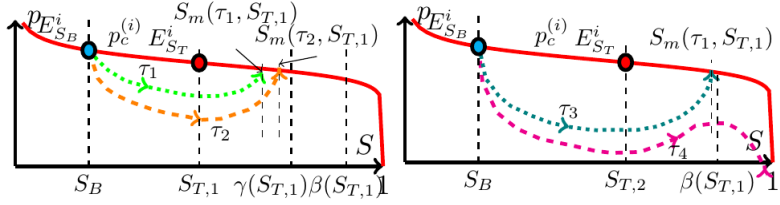}
\caption{Sketch of orbits represented by $w(S;\t,S_T)$. (left) $S_{T,1}\in (S_B,S_*]$, $\t_1<\t_2$; (right) 
$S_{T,2}\in (S_*,\bar{S}]$, $\t_3<\t_4$.}\label{2PF_fig:Prop2Cases}
\end{figure}
We start with the following
\begin{remark}
In the context of this section the wave-speed \eqref{2PF_eq:WavespeedRH} reduces to
$$
c=\frac{F(S_T)-F(S_B)}{S_T-S_B}.
$$
From assumption \ref{2PF_ass:F1} it follows that there is a one-to-one correspondence between $c$ and 
$S_T\in [S_B,\bar{S}]$. Writing $c=c(S_T)$, we have $c(S_B)=F'(S_B)$, $c(\bar{S})=F'(\bar{S})$ and 
$\frac{d c}{d S_T}>0$ in $[S_B,\bar{S}]$.\label{2PF_rem:CwithST}
\end{remark}

Next, we give a general monotonicity result.
\begin{proposition}[Monotonicity]
Let $S_B$ satisfy \eqref{2PF_eq:SbSt}.  
\begin{itemize}
\item[(a)]   For a fixed $S_T\in (S_B,\bar{S}]$ and any pair $0<\t_1<\t_2$,
$$ w(\cdot;\t_2,S_T)< w(\cdot;\t_1,S_T) \text{ in } \{w(\cdot;\t_1,S_T)<\Pim(\cdot)\}$$
and
$$ S_m(\t_1,S_T)<S_m(\t_2,S_T) \text{ if } S_T<S_m(\t_2,S_T)\leq \b(S_T);$$
\item[(b)] For fixed $\t>0$ and any pair $S_B<S_{T,1}<S_{T,2}\leq \bar{S}$,
$$ w(\cdot;\t,S_{T,2})< w(\cdot;\t,S_{T,1}) \text{ in } \{w(\cdot;\t,S_{T,1})<\Pim(\cdot)\}$$
and
$$ S_m(\t,S_{T,1})<S_m(\t,S_{T,2}) \text{ if } S_{T,2}\leq  S_m(\t,S_{T,2})\leq \b(S_{T,2}).$$
\end{itemize} 
\label{2PF_prop:Wtau}
\end{proposition}
 
\begin{proof}
We argue as in \cite{van2007new,van2013travelling}. The key idea is to introduce the function 
\begin{equation}
u=\frac{(\Pim-w)}{\sqrt{d}} \text{ with } d=c(S_T)\t.\label{2PF_eq:DefUtau}
\end{equation}
Using \eqref{2PF_eq:fWtauGDE} one obtains for $u$ the equation
\begin{equation}
{u}'(S;\t,S_T)=\frac{1}{\sqrt{d}}\Pim'(S)-\frac{\Gf(S;S_B,S_T)}{u(S;\t,S_T)}.\label{2PF_eq:Utau}
\end{equation}
Clearly, $u|_{S_B}=0$ and $u>0$ in a right neighbourhood of $S_B$. Since $\Gf|_{S_B}=0$ as well, 
one finds from \eqref{2PF_eq:Utau} and the sign of $u$
\begin{equation*}
u'(S_B;\t,S_T)=-\frac{\Pim'(S_B)}{2}\left[\sqrt{\dfrac{1}{d}-4\tfrac{\Gf'(S_B;S_B,S_T)}{(\Pim'(S_B))^2}}-\dfrac{1}{\sqrt{d}}\right]>0
\end{equation*}
since $\Gf'(S_B;S_B,S_T)<0$. Using \eqref{2PF_eq:GfisAlphaDecreasing}, \Cref{2PF_rem:CwithST} and some elementary algebra
\begin{subequations}
\begin{align}
&u'(S_B;\t_1,S_T)<u'(S_B;\t_2,S_T) &\text{ in case }  (a),\label{2PF_eq:UtauAtSb1}\\
&u'(S_B;\t,S_{T,1})<u'(S_B;\t,S_{T,2}) &\text{ in case }  (b)\label{2PF_eq:UtauAtSb2}.
\end{align}
\end{subequations}

\emph{(a)} From \eqref{2PF_eq:UtauAtSb1}, $u_1(\cdot):=u(\cdot;\t_1,S_T)<u(\cdot;\t_2,S_T)=: u_2(\cdot)$ in a right n
eighbourhood of $S_B$. We claim that $u_1$ and $u_2$ do not intersect in $\{u_1>0\}$. Suppose, to the contrary, 
there exists $S_i>S_B$ such that $u_1(S)<u_2(S)$ for $S_B<S<S_i$ and $u_1(S_i)=u_2(S_i)$. Thus ${u_1}'(S_i)\geq {u_2}'(S_i)$. 
Evaluating \eqref{2PF_eq:Utau} at $S_i$ gives 
$$
{u_1}'(S_i)=\frac{\Pim'(S_i)}{\sqrt{d_1}}-\frac{\Gf(S_i;S_B,S_{T,1})}{u_1(S_i)}<\frac{\Pim'(S_i)}{\sqrt{d_2}}
-\frac{\Gf(S_i;S_B,S_{T,2})}{u_2(S_i)}={u_2}'(S_i),
$$
a contradiction.

 If $S_T<S_m(\t_2,S_T)\leq \b(S_T)$, the $u$-monotonicity gives $S_m(\t_1,S_T)\leq S_m(\t_2,S_T)$. We rule out the equality 
 by contradiction. Suppose $S_m(\t_1,S_T)=S_m(\t_2,S_T)=:S_m$. Then
 $$
u_1<u_2 \text{ in } (S_B,S_m).
 $$
Integrating equation \eqref{2PF_eq:Utau} from $S_T$ to $S_m$ gives
\begin{equation}
u_2(S_T)-u_1(S_T)=(\Pim(S_T)-\Pim(S_m))\left(\tfrac{1}{\sqrt{d_2}}- \tfrac{1}{\sqrt{d_1}}\right)+ \int^{S_m}_{S_T} \Gf \left(\tfrac{1}{u_2} - \tfrac{1}{u_1}\right).\label{2PF_eq:WhySmIsUnique}
\end{equation}
Since $\Gf>0$ in $(S_T,S_m)$ for $S_T<S_m\leq \b(S_T)$, the term in the right of \eqref{2PF_eq:WhySmIsUnique} is negative, yielding a contradiction.

\emph{(b)} Using \eqref{2PF_eq:GfisAlphaDecreasing} this part is demonstrated along the same lines. Details are omitted. \qed
\end{proof} 
\begin{remark}
To complement \Cref{2PF_prop:Wtau}, we further state that
$$
 S_m(\t_1,S_T)=S_T \text{ if } S_m(\t_2,S_T)=S_T \text{ and }  S_m(\t_2,S_T)=1 \text{ if } S_m(\t_1,S_T)=1.
$$
The statements follow directly from the ordering of the orbits.
\end{remark}

So far we have shown the monotonicity of the orbits. However, the question of continuous variation is still open. 
This is addressed  in the following results.
\begin{proposition}[Continuous dependence of $w$]
Let $v=\Pim -w$. In the context of \Cref{2PF_prop:Wtau} and with $\Phi$ defined in \eqref{2PF_eq:WlowerBoundDefPhi} we have
\begin{itemize}
\item[(a)] $0<v^2(S;\t_2,S_T) -v^2(S;\t_1,S_T)<2c(\t_2-\t_1)\Phi(S)$ for $S_B<S\leq S_m(\t_1,S_T)$;
\item[(b)] $0<v^2(S;\t,S_{T,2}) -v^2(S;\t,S_{T,1})<2(c(S_{T,2})\,\Phi(S;S_B,S_{T,2})-c(S_{T,1})\,\Phi(S;S_B,S_{T,1}))$ 
for $S_B<S\leq S_m(\t,S_{T,1})$;
\end{itemize}\label{2PF_prop:ContinuousDependence}
\end{proposition}

\begin{proof}
As shown earlier in this section, $v$ satisfies the equation
$$
(v^2)'=2v \Pim' -2c\t \Gf \text{ in } \{v>0\}. 
$$
Integrating this equation and using \Cref{2PF_prop:Wtau} and $\Phi$ from \eqref{2PF_eq:WlowerBoundDefPhi} gives the desired inequalities. 
\qed
\end{proof}

\begin{corollary}[Continuous dependence of $S_m$]
Let $\t_0>0$ and $S_{T_0}$ be fixed such that $S_m(\t_0,S_{T_0})\leq \b(S_{T_0})$. Then for any small $\e>0$, 
there exists $\d=\d(\e;\t_0,S_{T_0})$ so that $|S_m(\t,S_T)-S_m(\t_0,S_{T_0})|<\e$ if $\max\{|\t-\t_0|, |S_T-S_{T_0}|\}<\d$ 
and $S_m(\t,S_T)<\b(S_T)$.\label{2PF_cor:ContinuousDependence}
\end{corollary}
\begin{proof}
We only demonstrate continuity with respect to $\t$. Proving the continuity with respect to $S_T$ follows the same lines. 
We therefore take $S_T=S_{T_0}$ and drop its dependence from the notation for simplicity. 
Consider first $\t>\t_0$ and $S_{T_0}<\b(S_{T_0})$. Recalling $v(S_m(\t_0);\t_0)=0$, \Cref{2PF_prop:ContinuousDependence} gives
$$
0<v(S_m(\t_0);\t)<\sqrt{2c(\t-\t_0)\Phi(S_m(\t_0))},
$$
where $\Phi(S_m(\t_0))>0$ by \eqref{2PF_eq:PropOfPhi} and \Cref{2PF_prop:PropertiesOfSm}. For any given (small) 
$\e>0$ and with reference to \Cref{2PF_fig:ContinuousDependence} choosing $\d<\frac{\Pim(S_m(\t_0))-\Pim(S_m(\t_0)+\e)}{2c\Phi(S_m(\t_0))}$ we have
$$
w(S_m(\t_0),\t)>\Pim(S_m(\t)+\e) \text{ for all } \t-\t_0 <\d.
$$
Since $w'>0$, this implies the continuity of $\t>\t_0$.

Next let $\t<\t_0$ and $S_{T_0}\leq \b(S_{T_0})$. Now we have from \Cref{2PF_prop:ContinuousDependence}
\begin{equation}
0<v(S_m(\t);\t_0)<\sqrt{2c(\t_0-\t)\Phi(S_m(\t))}.\label{2PF_eq:vSmBound}
\end{equation}
Since $v(\cdot,\t_0)\in C([S_B,S_m(\t_0)])$, $v(S_m(\t_0),\t_0)=0$ and $v(\cdot, \t_0)>0$ in $(S_B,S_m(\t_0))$, 
the continuity of $S_m$ follows directly from \eqref{2PF_eq:vSmBound}.
\qed
\end{proof}

\begin{figure}[H]
\begin{center}
\includegraphics[width=0.85\textwidth]{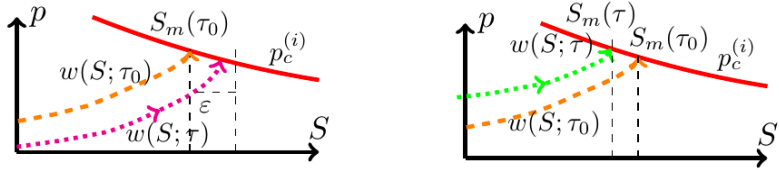}
\caption{Behaviour of $w$ close to $S_m$: (left) $\t>\t_0$ and (right) $\t<\t_0$.}
\label{2PF_fig:ContinuousDependence}
\end{center}
\end{figure}
Now we are in a position to describe how $S_m$ behaves for different combinations of $S_T$ and $\t$.
\begin{proposition}
Let $S_B$ satisfy \eqref{2PF_eq:SbSt} and  fix $S_T\in (S_B,\bar{S}]$. Then there exists a $\t_m(S_T)>0$ such that 
$$
S_m(\t,S_T)=S_T \text{ for all } 0<\t\leq \t_m(S_T) \text{ and } S_m(\t,S_T)>S_T\text{ for all } \t> \t_m(S_T).
$$\label{2PF_prop:Tau_m}
\end{proposition}

\begin{proof} 
The proof is based on Proposition 2.1 of \cite{mitra2018wetting} and Lemma 4.5 of \cite{van2007new}. 
Let us define the function $\ell_r(S)=\Pim(S)+ r(S-S_T)$ for $r>0$ and the constant $\underline{P}=\min_{S\in(S_B,\bar{S})}\{-\Pim'(S)\}>0$.
We show that for $\t$ small enough there exists an $r>0$ for which the curves $w(S)$ and $\ell_r(S)$ do not intersect. Specifically
\begin{equation}
\ell_{r}(S)<w(S;\t,S_T)<\Pim(S) \text{ for all } r>r^-\in (0,\underline{P}) \text{ and } S\in (S_B,S_T).\label{2PF_eq:BetweenEllrAndPim}
\end{equation}
This directly shows that $w(S_T)=\Pim(S_T)=\ell_r(S_T)$ meaning $S_m(\t,S_T)=S_T$.

Assuming the contrary, let $S_i\in (S_B,S_T)$ be the coordinate at which  $w(S;\t,S_T)$ and $\ell_r(S)$ 
intersect for the first time. Since $w(S_B)=\Pim(S_B)>\ell_r(S_B)$, one gets
\begin{align}
&(\Pim-w)(S_i)=r(S_T-S_i) \text{ and } \tfrac{c\t\Gf(S_i)}{r(S_T-S_i)}=w'(S_i)\leq {\ell_r}'(S_i)= \Pim'(S_i)+ r.\label{2PF_eq:EllmIntersection}
\end{align}
Recalling that $\Gf(S_T)=0$ and taking $m_0(S_T)=\sup_{S\in[S_B,S_T]} \Gf'(S;S_B,S_T)<\infty$ we get
 $$
 \frac{\Gf(S_i)}{S_i-S_T}\leq m_0(S_T).
 $$
Combining this with \eqref{2PF_eq:EllmIntersection} results in the inequality $r^2 -r \underline{P}  + c\t m_0\geq 0$. 
The roots of the quadratic expression on the left hand side of this inequality motivates us to define
\begin{equation}
\bar{\t}_m=\frac{\underline{P}^2}{4c m_0(S_T)} \text{ and } 
r^\pm=\frac{\underline{P}}{2}\left[1\pm \sqrt{1-\tfrac{\t}{\bar{\t}_m}}\right].
\label{2PF_eq:DefTauBarM}
\end{equation}
It directly follows that the inequality in \eqref{2PF_eq:EllmIntersection} 
is not satisfied if $0<\t<\bar{\t}_m$ and $r\in (r^-,r^+)\subset (0, \underline{P})$. 
In this case one has $w(S)> \ell_{r}(S)$ for $S\in (S_B,S_T)$ and consequently, \eqref{2PF_eq:BetweenEllrAndPim} holds. 
Note that $\t>\bar{\t}_m$ does not necessarily imply that $S_m(S_T,\t)>S_T$. For this purpose, we define
\begin{align}
\t_m(S_T):=\sup\{\t: S_m(S_T,\t)=S_T\}\geq \bar{\t}_m>0.\label{2PF_eq:DefTauM}
\end{align}
Using \cite[Proposition 4.2(b)]{VANDUIJN2018232}, which states that
\begin{equation}
w(S;\t,S_T)\to -\infty \text{ as } \t\to \infty \text{ for all } S\in (S_B,S_T], \label{2PF_eq:WtauToInfty}
\end{equation}
 we get from \Cref{2PF_cor:ContinuousDependence}, $\t_m(S_T)<\infty$.\qed
\end{proof}
We consider now the case $\t>\t_m(S_T)$. \Cref{2PF_prop:PropertiesOfSm} guarantees that 
$S_m(S_T,\t)<\g(S_T)\leq \b(S_T)$ if $S_B<S_T\leq S_*$. However, for $S_T> S_*$ it is unclear whether $S_m(S_T,\t)$ 
is bounded by $\b(S_T)$ or not. We show below that a $\t_c=\t_c(S_T)$ exists in this case such that 
$S_m(S_T, \t) \in (S_T,\b(S_T)]$ if $\t \in (\t_m(S_T), \t_c(S_T)]$ implying from \Cref{2PF_prop:PropertiesOfSm} 
that $S_m(S_T,\t)=1$ for all $\t>\t_c(S_T)$. 
\begin{proposition}
Let $S_B$ satisfy \eqref{2PF_eq:SbSt}. Then the following holds:
\begin{itemize}
\item[(a)] For each $S_T\in (S_*,\bar{S})$, there exists a unique $\t_c=\t_c(S_T)$ such that
$$
S_m(\t_c,S_T)=\b(S_T).
$$
\item[(b)] The function $\t_c(\cdot)$ is strictly decreasing and continuous on $[S_*,\bar{S}]$. 
One has $\t_c(S_T)\to \infty$ as $S_T\searrow S_*$ and $\t_c(\bar{S})=\bar{\t}=\t_m(\bar{S})>0$.
\end{itemize}\label{2PF_pros:PropOFTauC}
\end{proposition}
\begin{figure}[H]
\begin{center}
\includegraphics[width=0.85\textwidth]{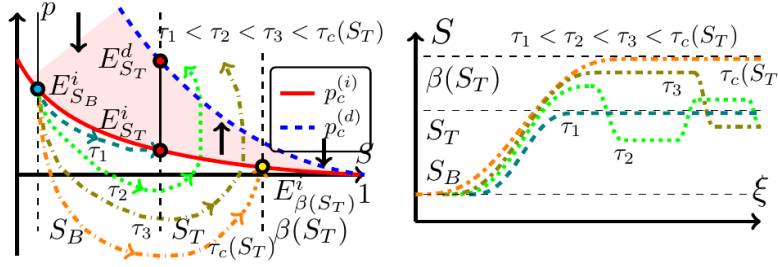}
\end{center}
\caption{(left) Ordering of the orbits in the $S$-$p$ phase plane for $S_*<S_T<\bar{S}$ and $\t\leq \t_c(S_T)$. (right) 
The behaviour of the orbits in the $\xi$-$S$ plane for $\t\leq \t_c(S_T)$.}\label{2PF_fig:CaseSgeqSstar}
\end{figure}

\begin{proof}
\emph{(a)}  
Suppose no $\t_c(S_T)$ exists such that $S_m(\t_c,S_T)=\b(S_T)$, meaning $S_m(\t,S_T)<\b(S_T)$ for all $\t>0$.
Combined with \eqref{2PF_eq:WtauToInfty}, this implies that for large enough $\t$, a $S_1\in [S_T,\b(S_T)]$ 
exists for which $w(S_1)=0$. From \eqref{2PF_eq:fWtauGDE} it is evident that $w(S_T)\leq w(S)$, in particular $w(S_T)<w(S_1)=0$. 
Moreover, \eqref{2PF_eq:WlowerBoundDefPhi} gives the lower bound 
$w(S)\geq w(S_T)\geq \Pim(S_T)-\sqrt{2c\t\Phi(S_T)}\geq -\sqrt{2c\t\Phi(S_T)}$ for all $S\in [S_B,S_m(\t,S_T)]$.
Multiplying both sides of \eqref{2PF_eq:Utau}  by $u$, integrating from $S_B$ to $S_1$ and using the above inequality we get
\begin{align*}
-\dfrac{1}{2}\Pim^2(S_B)&=\int_{S_B}^{S_1} (c\t \Gf(S;S_B,S_T) + \Pim'(S)w(S;c,\t) )dS\\
&\leq  -c\t \Phi(S_1) +(\Pim(S_B)-\Pim(S_1))\sqrt{2c\t\Phi(S_T)}.
\end{align*}
Since $\Phi(S_1)>0$ (as stated in \eqref{2PF_eq:PropOfPhi2}), this leads to a contradiction for $\t\to \infty$. 
Hence, $S_m(\t,S_T)=\b(S_T)$ for some $\t>0$. The uniqueness follows from \Cref{2PF_prop:Wtau}. 

\emph{(b)} The monotonicity and continuity follows from \Cref{2PF_prop:Wtau,2PF_prop:ContinuousDependence} 
and \Cref{2PF_cor:ContinuousDependence}. To show the limit for $S_T\searrow S_*$, assume that $\lim_{S_T\searrow S_*} \t_c(S_T) 
= \t_\infty <\infty$.  Let then $\t>\t_\infty$. \Cref{2PF_prop:PropertiesOfSm} implies that $S_m(\t,S_*)<\b(S_*)$. 
Choose an $S_T>S_*$ such that $\b(S_T)\geq S_m(\t,S_*)$. Since $\t>\t_c(S_T)$, we get that 
$w(S_m(\t,S_*);\t,S_T)\leq w(\b(S_T);\t,S_T)\leq \Pim(\b(S_T))$, implying 
$w(S_m(\t,S_*);\t,S_*)-w(S_m(\t,S_*);\t,S_T)\geq \Pim(S_m(\t,S_*))-\Pim(\b(S_T))$. 
This gives a contradiction when $S_T\searrow S_*$ since the right hand side goes to $\Pim(S_m(\t,S_*))-\Pim(\b(S_*))>0$, 
whereas the left hand side converges to 0 from \Cref{2PF_cor:ContinuousDependence}.

The existence of a $\bar{\t}>0$ is a consequence of the continuity of 
$\t_c$ with $\bar{\t}=\t_m(\bar{S})\geq \bar{\t}_m(\bar{S})$ following from \Cref{2PF_prop:Tau_m}. \qed
\end{proof}

After the preliminary statements we are in a position to consider the solvability of \hyperref[2PF_eq:ODE]{(TW)} 
for different ranges of $S_T$. 

\subsection{Problem \hyperref[2PF_eq:ODE]{(TW)} with $S_B<S_T\leq \bar{S}$}\label{2PF_sec:Ap1}
\begin{figure}[H]
\begin{center}
\includegraphics[width=0.85\textwidth]{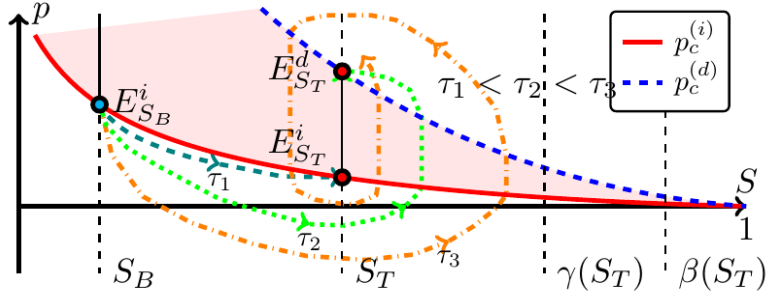}
\caption{The different cases of $S_B<S_T< S_*$. The orbits are plotted for $0<\t_1<\t_i<\t_2<\t_d<\t_3$.}
\label{2PF_fig:CaseStleqSstar}
\end{center}
\end{figure}

We investigate the existence of an orbit connecting $(S_B,\Pim(S_B))$ and the segment $\overline{E^i_{S_T}E^d_{S_T}}$. Defining
\begin{equation}
\t_{j}=\frac{({p^{(j)}_c}'(S_T))^2}{4c \Gf'(S_T;S_B,S_T)}>0, \quad j\in\{i,d\},\label{2PF_eq:ExpressionEigenValues}
\end{equation}
the eigenvalues of the \hyperref[2PF_eq:ODE]{(TW)} system associated with the equilibrium points $E^j_{S_T}$, $j\in\{i,d\}$ are 
$$
\lambda^j_\pm=\frac{{p^{(j)}}'(S_T)}{2c\t}\left[1\pm\sqrt{1-\frac{\t}{\t_j}}\right] \text{ implying } \begin{cases}
E^j_{S_T} \text{ is stable sink for } \t\leq \t_j,\\
E^j_{S_T} \text{ is stable spiral sink for } \t>\t_j.
\end{cases}
$$
This immediately gives $\t_i\geq \t_m$ as no monotone orbit can connect with $E^i_{S_T}$  for $\t>\t_i$. 
The general behaviour of the orbits for $S_T\in (S_B,S_*]$ are stated in

\begin{theorem}
Under the assumptions of Scenario A, consider $S_B$ satisfying \eqref{2PF_eq:SbSt}, $S_T\in (S_B,S_*]$
and $\t_i<\t_d$. Let $(S,p)$ be the orbit  originating from $(S_B,\Pim(S_B))$ satisfying \hyperref[2PF_eq:ODE]{(TW)}. 
Then, with reference to  \Cref{2PF_fig:CaseStleqSstar}, as $\xi\to \infty$ one gets
\begin{itemize}[label=$\bullet$]
\item[(a)] If $0<\t\leq \t_i$, then either $S\to S_T$ and $p\to \Pim(S_T)$ monotonically with respect to 
$\xi$ through $\Him$ (when $\t\leq \t_m$) or the orbit $(S,p)$ goes around $\overline{E^i_{S_T}E^d_{S_T}}$ 
finitely many times and ends up in either $E^i_{S_T}$ or $E^d_{S_T}$ (when $\t_m<\t\leq \t_i$).
\item[(b)] If $\t_i<\t\leq \t_d$, $(S,p)\to E^d_{S_T}$ after finitely many turns around $\overline{E^i_{S_T}E^d_{S_T}}$.
\item[(c)] If $\t_d<\t$, then $(S,p)$ revolves infinitely many times around $\overline{E^i_{S_T}E^d_{S_T}}$
while approaching it and \eqref{2PF_eq:BC1} is satisfied.
\end{itemize}
\label{2PF_theo:SleqSstar}
\end{theorem}

\begin{figure}[h!]
\begin{center}
\includegraphics[width=0.85\textwidth]{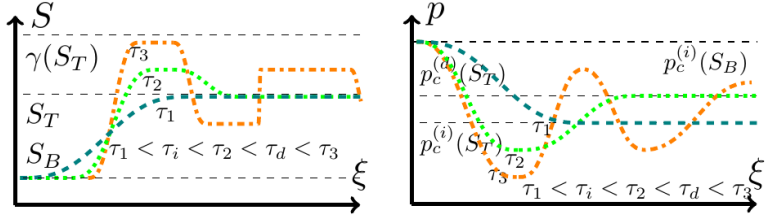}
\end{center}
\caption{Typical behaviour of $S(\xi)$  (left)  and  $p(\xi)$ (right) for different values of $\t$. 
Here profiles for three different $\t$ values are plotted satisfying $0<\t_1<\t_i<\t_2<\t_d<\t_3$. 
The $\xi=0$ coordinate is fixed by \eqref{2PF_eq:DefOfzeta0}.}\label{2PF_fig:Theo1PSxDiagram}
\end{figure}
These statements are demonstrated by arguments from  \cite[Theorem 2.1 and Lemma 2.1 \& 2.2]{mitra2018wetting}. 
We omit the details here. In \Cref{2PF_theo:SleqSstar} we have taken $\t_i<\t_d$ without loss of generality. 
In the $\t_i>\t_d$ case, the roles of the equilibrium points $E^i_{S_T}$ and $E^d_{S_T}$ are reversed. 
The typical behaviour of the $S$ and the $p$ profiles with respect to $\xi$ is given in \Cref{2PF_fig:Theo1PSxDiagram}. 
Both $S$ and $p$ are monotone for $\t<\t_i$, whereas for $\t_i<\t<\t_d$ they have finite number of local extrema and 
$p(+\infty)=\Pdr(S_T)$. For $\t>\t_d$, $S$ has infinitely many decaying local extrema, whereas $p$ has no limit. 
In particular, each $S$ maximum corresponds to a saturation overshoot.  On the other hand, the oscillations in $p$ 
become wider, in line with the assumption $\lim_{\xi\to \infty}p'(\xi)=0$. In this case, the segment $\overline{E^i_{S_T}E^d_{S_T}}$ 
becomes an $\omega$-limit set of the orbit. 

Turning to the case, $S_T\in (S_*,\bar{S})$, we define the two functions which will be used extensively below
\begin{definition}
The functions $\hat{S}_B,\;\check{S}_B:[0,\infty)\to (0,1]$ are such that
\begin{equation*}
\check{S}_B(\t)=\begin{cases}
(\t_c)^{-1}(\t) &\text{ for } \t>\bar{\t},\\
\bar{S} &\text{ for } 0\leq \t\leq \bar{\t},
\end{cases}
\text{ and } \hat{S}_B(\t)=\b(\check{S}_B(\t)).\label{2PF_def:Supdown_i}
\end{equation*}
\end{definition}

Observe that, ${\check{S}_B}(\t)$ is a strictly decreasing function whereas ${\hat{S}_B}(\t)$ is a strictly 
increasing function for $\t>\bar{\t}$ and $\hat{S}_B(\t)=\check{S}_B(\t)=\bar{S}$ for $\t\leq \bar{\t}$. 
This is sketched in \Cref{2PF_fig:Bifurcation1}. Numerically computed $\check{S}_B(\t)$ and  $\hat{S}_B(\t)$
functions are shown in \Cref{2PF_fig:NumResBifurcationDiagram}. 

\begin{remark}
The case when $\b(\a)$ does not intersect $\g(\a)$ is treated in a similar way. However, since orbits may intersect the line segment $\{S=1,p\leq 0\}$ in this case, a multivalued extension of $\Pim$ at $S=1$ needs to be introduced, see \cite{mitra2018wetting,VANDUIJN2018232} 
for further details. With this, one shows that the function $\t_c(S_T)$ is well-defined in $[\tilde{S},\bar{S}]$. 
Then a $\t_{\!_B}>0$ exists such that ${\check{S}_B}(\t_{\!_B})=\tilde{S}$ and ${\hat{S}_B}(\t_{\!_B})=1$. 
The subsequent results remain valid if $\hat{S}_B$ and $\check{S}_B$ are extended by
$$
\check{S}_B(\t_{\!_B})=\tilde{S} \text{ for  } \t>\t_{\!_B}, \text{ and }  \hat{S}_B(\t)=1 \text{ for  } \t>\t_{\!_B}.
$$
\end{remark}
With this in  mind, we define the following sets:
\begin{align}
&\mathcal{A}=\{(S_T,\t):S_B<S_T<\bar{S},\; \t<\t_c(S_T)\},\nonumber\\
&\mathcal{B}=\{(S_T,\t):\t>\bar{\t},\; \check{S}_B(\t)<S_T<\hat{S}_B(\t)\},\nonumber\\
&\mathcal{C}=\{(S_T,\t):\bar{S}<S<S^*,\; \t<\t_c(\b^{-1}(S_T))\}. \label{2PF_eq:solution_classes}
\end{align}
\begin{figure}
\begin{center}
\includegraphics[scale=0.65]{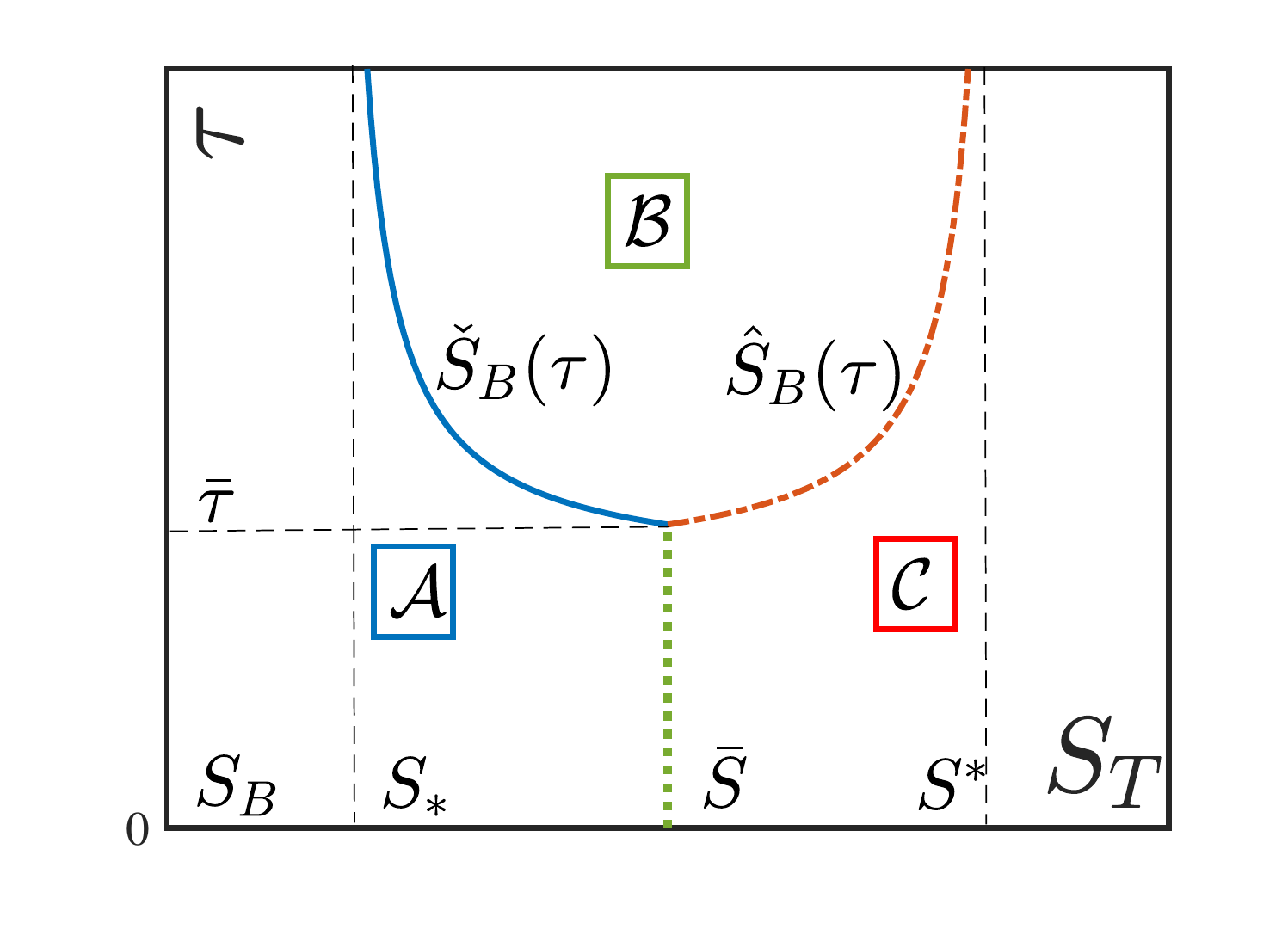}
\end{center}\caption{The sets $\mathcal{A}$, $\mathcal{B}$ and $\mathcal{C}$ and the functions $\check{S}_B(\t)$, 
$\hat{S}_B(\t)$.}\label{2PF_fig:Bifurcation1}
\end{figure}
Observe that, if $S_T<\bar{S}$ then only regions $\mathcal{A}$ and $\mathcal{B}$ are relevant. 
With $S_o$ defined in \ref{2PF_ass:F1}, for  $(S_T,\t)\in \mathcal{A}$ one has
\begin{proposition}
For a fixed $S_B\in (0,S_o)$ and $(S_T,\t)\in \mathcal{A}$, the orbit $(S,p)$ entering $\Him$ from $E^i_{S_B}$ 
behaves according to statements (a),(b) and (c) of \Cref{2PF_theo:SleqSstar}. 
\end{proposition}
We discuss the remaining situations, $(S_T,\t)\in\mathcal{B}$ and $(S_T,\t)\in\mathcal{C}$ in the next section.

\subsection{$(S_T,\t)\not\in \mathcal{A}$}\label{2PF_sec:notInA}
Since $S_T>\check{S}_B(\t)$, a TW cannot connect $S_B$ and $S_T$. However, a different class of waves is possible when $(S_T,\t)\in \mathcal{B}$. 

\begin{proposition}
For a fixed $S_B\in (0,S_o)$ and $(S_T,\t)\in \mathcal{B}$, consider the system 
\begin{align}\begin{cases}
S'=\dfrac{1}{c_d\t}\F(S,p),\\[.2em]
p'=\Gf(S;\hat{S}_B(\t),S_T)
,\end{cases} \text{ with } c_d=\frac{F(\hat{S}_B(\t))-F(S_T)}{\hat{S}_B(\t) - S_T}. \label{2PF_eq:ModifiedPeq}
\end{align}
For this system an orbit $(S_d, p_d)$ exists that connects $E^d_{\hat{S}_B(\t)}$ for $\xi\to -\infty$ to 
$\overline{E^i_{S_T}E^d_{S_T}}$ for $\xi\to \infty$.\label{2PF_prop:DraniageWaveExists}
\end{proposition}
\begin{proof}
Upon inspection of the eigen-directions for the system \eqref{2PF_eq:ModifiedPeq} around the equilibrium point $E^d_{\hat{S}_B(\t)}$  
one concludes that there is indeed  
an orbit $(S_d,p_d)$ that connects to $E^d_{\hat{S}_B(\t)}$ as $\xi\to -\infty$ from the set $\Hdr$ defined in \eqref{2PF_eq:SetHdef}. 
Moreover, from the direction of the orbits in this case, as shown in \Cref{2PF_fig:CaseBorbits} (left), 
it is apparent that after leaving $E^d_{\hat{S}_B(\t)}$, $S_d$ decreases monotonically  till the orbit either 
hits the curve $p=\Pdr(S)$ for some $S\leq S_T$ or exits $\{S>S_B\}$ through the line $S=S_B$. 
We prove that it is not possible for the orbit to escape through $S=S_B$.
\begin{figure}[h!]
\begin{center}
\includegraphics[width=0.975\textwidth]{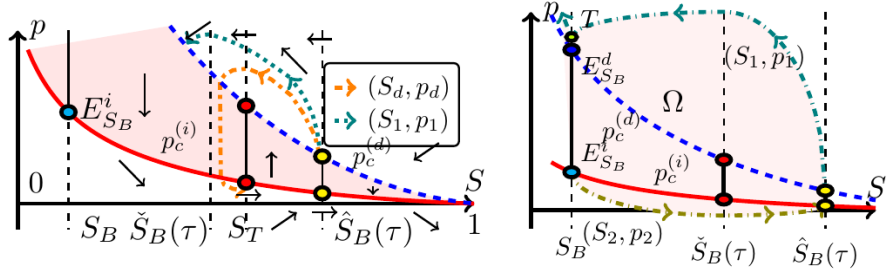}
\end{center}
\caption{(left) The direction of orbits for the system \eqref{2PF_eq:ODES}, \eqref{2PF_eq:ModifiedPeq} and the orbits
$(S_d,p_d)$ and $(S_1,p_1)$. Here the orbit $(S_d,p_d)$ connects $E^d_{\hat{S}_B(\t)}$ 
and $E^i_{S_T}$. (right) The domain $\Om$ used in the divergence argument for the hypothetical case where $(S_1,p_1)$ 
crosses the line $S=S_B$.}\label{2PF_fig:CaseBorbits}
\end{figure}
To show this, consider the orbit $(S_1,p_1)$ that satisfies the original \hyperref[2PF_eq:ODE]{(TW)} equations and enters 
$\Hdr$ from $E^d_{\hat{S}_B(\t)}$. We show that this orbit cannot cross the line $S=S_B$. The 
divergence argument presented in \cite{cuesta2000infiltration,VANDUIJN2018232,mitra2018wetting} is used for this purpose. 
To elaborate, assume that $(S_1,p_1)$ intersects the line $S=S_B$ at $T$. Consider the region $\Om$, enclosed by the segments 
$\overline{E^i_{S_B}T}$, $\overline{E^i_{\hat{S}_B(\t)}E^d_{\hat{S}_B(\t)}}$, the orbit $(S_1,p_1)$  and the orbit $(S_2,p_2)$ that 
satisfies \hyperref[2PF_eq:ODE]{(TW)} and connects $E^i_{S_B}$ to $E^i_{\hat{S}_B(\t)}$, see \Cref{2PF_fig:CaseBorbits} 
(right). Introducing the vector-valued function $\vv{R}(S,p)=(\frac{1}{c\t}\F(S,p),\Gf(S;S_B,S_T))$ and deduces from \eqref{2PF_eq:Psi},
$$
\mathrm{div} \vv{R}=\frac{1}{c\t}\frac{\p \F}{\p S}(S,p)=\frac{1}{c\t}
\begin{cases}
\Pim'(S) &\text{ in } \Him,\\[-.2em]
0  &\text{ in } \H,\\[-.2em]
\Pdr'(S) &\text{ in } \Hdr.
\end{cases}
$$
This gives a contradiction when the divergence theorem is applied  to $\vv{R}$ in the domain $\Om$: the integral 
of $\vv{R}$ over $\p\Om$ is non-negative whereas $\int_\Om \mathrm{div} \vv{R} <0$ from \ref{2PF_P:Pc} and  
\Cref{2PF_fig:CaseBorbits} (right). Hence, the orbit $(S_1,p_1)$ intersects $\Pdr(S)$ at some $S \in (S_B,\check{S}_B(\t)]$. 

The wave-speed corresponding to the orbit $(S_d,p_d)$ satisfies 
\begin{equation}
c_d<\tfrac{F(\hat{S}_B(\t))-F(S_B)}{\hat{S}_B(\t)-S_B} =c_i,
\end{equation}
$c_i$ being the speed of both $(S_1,p_1)$ and $(S_2,p_2)$ waves. Hence, by the continuity of the orbits with respect to $c$, 
as shown in \Cref{2PF_prop:Wtau}, it is evident that $(S_d,p_d)$ intersects $\Pdr(S)$ for some $S_d>S_B$. From here, 
the rest of the proof is identical to the proof of \Cref{2PF_theo:SleqSstar}, and follows the arguments in 
\cite[Theorem 2.1 and Lemma 2.1 \& 2.2]{mitra2018wetting}.
\end{proof}
From the results of \Cref{2PF_theo:SleqSstar} we further state 
\begin{corollary}
The orbit $(S_d,p_d)$ can monotonically go to $E^d_{S_T}$ only when $\t\leq \frac{(\Pdr'(S_T))^2}{4c_d \Gf'(S_T;\hat{S}_B(\t),S_T)}$. 
For $\t$ large enough, the orbit $(S_d,p_d)$ goes around  $\overline{E^i_{S_T}E^d_{S_T}}$ infinitely many times while approaching it, 
and $\lim\limits_{\xi\to \pm\infty}p_d'(\xi)=0$.\label{2PF_cor:DrainageOscillation}
\end{corollary}
Observe that, if $(S_T,\t)\in \mathcal{C}$ then travelling waves do not exist between $S_T$ and $\hat{S}_B(\t)$ 
since both are in the concave part of $F$ with $S_T>\hat{S}_B(\t)$. Thus  we have exhausted all the possibilities of 
connecting $S_B$ and $S_T$ with \Cref{2PF_theo:SleqSstar} and \Cref{2PF_prop:DraniageWaveExists}.

\subsection{Entropy solutions to hyperbolic conservation laws}
\begin{figure}[H]
\begin{center}
\includegraphics[width=0.97\textwidth]{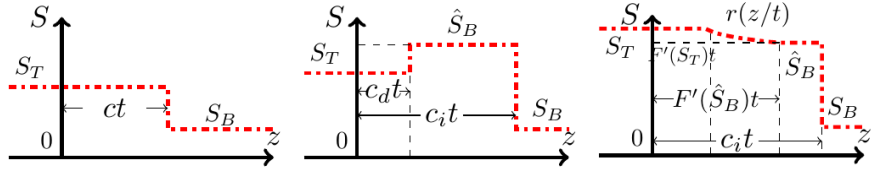}
\end{center}
\caption{ The entropy solutions for (left) $(S_T,\t)\in \mathcal{A}$, (center) $(S_T,\t)\in \mathcal{B}$ and (right) 
$(S_T,\t)\in \mathcal{C}$. Note that the solutions in the center and the right figures include non-classical shocks.}
\label{2PF_fig:EntropySol}
\end{figure}
Under the conditions of Scenario A, we consider the Riemann problem
\begin{subequations}\label{2PF_eq:LimitingBL}
\begin{align}
&\dfrac{\p S}{\p t} + \dfrac{\p F(S)}{\p z}=0 \text{ in }\R\times[0,\infty) \label{2PF_eq:RiemannScenarioAgde}\\
&\text{ with } S(z,0)=
\begin{cases}
S_T &\text{ for } z<0,\\
S_B&\text{ for } z>0.
\end{cases}
\end{align}
\end{subequations} 
In the context of the viscous model discussed in this paper, we consider the Buckley-Leverett 
equation \eqref{2PF_eq:RiemannScenarioAgde} as the limit of System \eqref{2PF_eq:FullRegularisedSystem} for $N_c\searrow 0$. 
As a consequence, we only take into account those shock solutions of \eqref{2PF_eq:RiemannScenarioAgde} that have a viscous
profile in the form of a travelling wave satisfying \hyperref[2PF_eq:ODE]{(TW)}. Such shocks are called admissible because 
they arise as the $N_c\to 0$ limit of TWs. In this sense, the entropy condition for shocks satisfying \eqref{2PF_eq:RiemannScenarioAgde} 
are equivalent to existence conditions for travelling waves satisfying \hyperref[2PF_eq:ODE]{(TW)}. This may lead to non-classical shocks 
violating the well-known Oleinik entropy conditions, see e.g. \cite{van2007new}.

Here, we assume
\begin{equation}
0<S_B<S_T<1,\label{2PF_eq:CaseAentropySbSt}
\end{equation}
which is more general compared to \eqref{2PF_eq:SbSt} where the additional constraint of $S_B<S_o$ was imposed. 
This generalization is possible since $S_B>S_o$ simply implies that the sets $\mathcal{A},\;\mathcal{B}$ are empty. 
Our analysis can also be applied to derive the entropy conditions for the case $S_B>S_T$, however, for simplicity we 
restrict our discussion to \eqref{2PF_eq:CaseAentropySbSt}.  

\subsubsection*{\underline{$(S_T,\t)\in \mathcal{A}$}}\vspace{-.2em}
As in the usual Buckley-Leverett case (i.e. without dynamic capillarity and hysteresis in the regularised models) 
the solution is given by
\begin{equation}
S(z,t)=\begin{cases}
S_T &\text{ for } z<ct,\\
S_B &\text{ for } z>ct,
\end{cases} \text{ where } c=\frac{F(S_T)-F(S_B)}{S_T - S_B}.\label{2PF_eq:CaseAentropyStauinA}
\end{equation}
Here, the shock satisfies the classical Oleinik condition.

\subsubsection*{$\underline{(S_T,\t)\in \mathcal{B}}$}\vspace{-.2em}
In this case the admissible solution is composed of two shocks: an infiltration shock from $S_B$ to $\hat{S}_B(\t)$, followed by a drainage shock from $\hat{S}_B(\t)$ to $S_T$.
\begin{equation}
S(z,t)=\begin{cases}
S_T &\text{ for } z<c_dt,\\
\hat{S}_B(\t) &\text{ for } c_d t<z<c_i t,\\
S_B &\text{ for } z>c_i t,
\end{cases} \text{ with } 
\begin{cases}
 c_i=\frac{F(\hat{S}_B(\t))-F(S_B)}{\hat{S}_B(\t)-S_B},\\[.2em]
 c_d=\frac{F(\hat{S}_B(\t))-F(S_T)}{\hat{S}_B(\t)-S_T}.
 \end{cases}\label{2PF_eq:CaseAentropyStauinB}
\end{equation}
 Note that this solution violates the Oleinik condition \cite{oleinik1957discontinuous}. Both shocks are under compressive \cite{lefloch2002hyperbolic}.
\subsubsection*{$\underline{(S_T,\t)\in \mathcal{C}}$}\vspace{-.2em}
The solution in this case violates again the Oleinik entropy condition. It consists of an 
infiltration shock from $S_B$ to $\hat{S}_B(\t)$ followed by a rarefaction wave from $\hat{S}_B(\t)$  to $S_T$,
\begin{equation}
S(z,t)=\begin{cases}
S_T &\text{ for } z<F'(S_T)t,\\
r(z/t) &\text{ for } F'(S_T) t<z<F'(\hat{S}_B(\t)) t,\\
\hat{S}_B(\t) &\text{ for } F'(\hat{S}_B(\t)) t<z<c_i  t,\\
S_B &\text{ for } z>c_i t,
\end{cases} \label{2PF_eq:SolCaseC}
\end{equation}
with $r(\cdot)$ satisfying
\begin{equation}
F'(r(\zeta))=\zeta, \text{ for } F'(S_T)\leq \zeta \leq F'(\hat{S}_B(\t)).\label{2PF_eq:RWforScenarioA}
\end{equation}
Since $F$ is concave for $S\in [\hat{S}_B(\t),S_T]$, $F'$ is monotone implying that $r(\cdot)$ is well-defined. We observe that in the last two cases the solution features a plateau-like region. This plateau appears
and grows in time since the speeds of the drainage shock and of the end point of the rarefaction wave are lesser 
than the speed of the infiltration shock. Interestingly, the saturation of the plateau only depends on $\Pim$ and
not on $\Pdr$. To be more specific, although the viscous profile consisting of a travelling wave connecting 
$E^d_{\hat{S}_B}$ and $\overline{E^i_{S_T}E^d_{S_T}}$ depends on $\Pdr$, the shock solution resulting from it, 
in the hyperbolic limit, does not.  However, the role of the drainage curve in the entropy solutions become evident
in Scenario B, which is discussed in the next section.

In the absence of hysteresis and for linear higher order terms, which correspond to constant $k$ and linear $p_c$-$S$ 
dependence, in \cite[Section 6]{van2007new} it is proved that the non-standard entropy conditions discussed here are 
entropy dissipative for the entropy $U(s)=\frac{1}{2}s^2$. However, such an analysis is beyond the scope of this paper. 
The solution profiles for the Riemann problem are shown in \Cref{2PF_fig:EntropySol}.

\subsection*{Extension to the non-monotone $F$ case}
The analysis so far can be extended to the case where $N_g$ is large resulting in $F$ being non-monotone. If $S_F\in (0,1)$ is the saturation where $F(S)$ attains its maximum (see \Cref{2PF_rem:fhF} and \Cref{2PF_fig:fhF}), then the results obtained so far cover the case when $S_T$ and $S^*$ are below $S_F$. However, if $S_T>S_F$ then  the TW study has to be conducted also from a $S_T$ perspective, not only from the $S_B$ one. In this scenario, since fronts having negative speeds and thus moving towards $S_T$ become possible, one has to consider the functions $\hat{S}_T(\t)$, $\check{S}_T(\t)$ for a fixed $S_T$, similar to $\hat{S}_B(\t)$, $\check{S}_B(\t)$ from \Cref{2PF_def:Supdown_i} for fixed $S_B$. Due to the symmetry in the behaviour of the fronts approaching $S_B$, respectively $S_T$, some of the results obtained so far extend straightforwardly to the non-monotone case. However, a detailed analysis is much more involved and therefore left for future research because of the following two reasons:
\begin{enumerate}[label=(\alph*)]
\item Depending on the relative positions of $S_B$, $\hat{S}_B$, $S_T$ and $\hat{S}_T$, there are many sub-cases to consider. In this case up to three shocks are possible, traveling both forward and backward. Which of these shocks are admissible and how they are connected requires further analysis. 
\item For a non-monotone $F$, when considering the hyperbolic limit in the absence of hysteresis or dynamic effects, the entropy solutions may include rarefaction waves with endpoints moving in opposite directions, forward and backward. When capillary hysteresis is included, preliminary numerical results have provided solutions incorporating two rarefaction waves, one with endpoints travelling backward and another one with endpoints travelling forward, and a stationary shock at $z=0$. Such solutions still need to be analysed further.
\end{enumerate}

\section{Hysteretic  relative permeabilities and small $N_g$ (Scenario B)}\label{2PF_sec:CaseB}
For Scenario B, the flux function $F(S,p)$ is composed of $F^{(j)}=f^{(j)}+ N_g h^{(j)}$ for $j\in\{i,d\}$ 
and $\bar{F}=\bar{f} + N_g \bar{h}$ such that
\begin{equation}
F(S,p)=
\begin{cases}
F^{(d)}(S) &\text{  if } (S,p)\in \Hdr,\\
\bar{F}(S,p) &\text{  if } (S,p)\in \H,\\
F^{(i)}(S) &\text{  if } (S,p)\in \Him.
\end{cases}
\label{2PF_eq:FhysRelation}
\end{equation}
It has the following properties
\begin{enumerate}[label=(A\theTPFassumption)]
 \item $F\in C(\W)$, $\bar{F}\in C^2(\H)$, $\p_p F>0$ in $\H$ and $F^{(i)}$, $F^{(d)}$ satisfy 
 properties stated for $F$ in \ref{2PF_ass:F1}. Additionally, $F^{(d)}(S)>F^{(i)}(S)$ for $0<S<1$.
  \label{2PF_ass:F3}
\stepcounter{TPFassumption}
\end{enumerate}
In this scenario, $S_B$ can be taken in the entire interval $(0,1)$ and $p_B$ can be chosen 
independently as long as $(S_B,p_B)\in \H$, i.e.
\begin{equation}
0<S_B<1 \text{ and } p_B\in [\Pim(S_B), \Pdr(S_B)].\label{2PF_eq:PerHysSb}
\end{equation}
This is different from Scenario A where $S_B$ is restricted to the interval $(0,S_o)$ and $p_B$ is fixed to  $p_B=\Pim(S_B)$.

We first introduce some notation.
\begin{definition}
For $k\in\{B,T\}$ let $E_k= (S_k,p_k)$ and $U_k= (S_k,F(S_k,p_k))$ (see \Cref{2PF_fig:hysPlot1} (left)). 
We define the saturations $\bar{S}_j$, $j\in \{i,d\}$ as the $S$-coordinates of the tangent points to $F^{(j)}(S)$ 
from $U_B$ such that $\bar{S}_i\geq S_B$ and $\bar{S}_d\leq S_B$.\label{2PF_def:FhysQuatities}
\end{definition}
Observe that, the saturations $\bar{S}_j$, for $j\in \{i,d\}$, are functions of $U_B$. The properties 
of $F^{(j)}$ further ensure that they are well defined. If $S_B$ is such that ${F^{(i)}}''(S_B)\leq 0$ and $p_B= \Pim(S_B)$ 
then $\bar{S}_i = S_B$. Similarly if ${F^{(d)}}''(S_B)\geq 0$ and $p_B= \Pdr(S_B)$ then $\bar{S}_d = S_B$. 
 
The existence of travelling waves is analysed for the following two cases:
\begin{align*}
\text{ Case (i): } S_B<S_T\leq \bar{S}_i, \text{ and } \text{ Case (ii) : } \bar{S}_d\leq S_T<S_B.
\end{align*}

\begin{figure}[H]
\begin{center}
\includegraphics[width=0.95\textwidth]{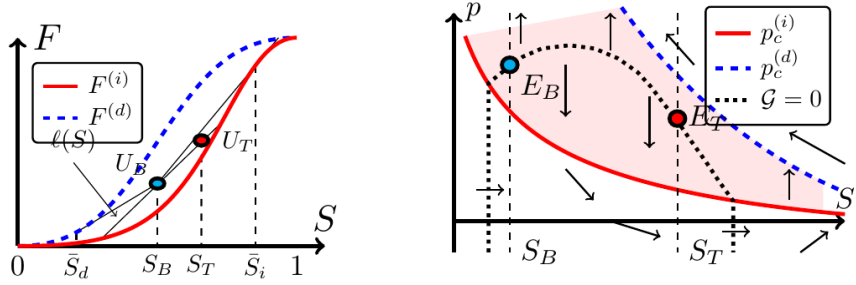}
\end{center}
\caption{(left) The graphs of $F^{(i)}$ and $F^{(d)}$, together with the saturations $\bar{S}_i,\;\bar{S}_d$ and 
the points $U_B,\;U_T$. (right) The orbit directions for Case (i) for two equilibrium points $E_B$ and $E_T$. 
The black dotted curve represents points where $\Gf(S,p)=0$, implying $p'=0$.}\label{2PF_fig:hysPlot1}
\end{figure}
\noindent
Regarding the choice of $p_T$, we have the following 
\begin{proposition}
Let $S_B$ and $S_T$ be as in Case (i) or Case (ii). Then any solution of \hyperref[2PF_eq:ODE]{(TW)} that connects $E_B$ and $E_T$ can only exist if $p_T=\Pim(S_T)$ or $p_T=\Pdr(S_T)$.\label{2PF_prop:pTrestriction}
\end{proposition}
\begin{proof}
Since $E_T$ is an equilibrium point, $\F(S_T,p_T)=0$, which implies that $p_T\in [\Pim(S_T),\Pdr(S_T)]$. 
The directions of the orbits for  $p_T$ in this interval are displayed in \Cref{2PF_fig:hysPlot1} (right). 
We proceed by introducing the set 
$$
\H_0=\{(S,p): S\in (0,1),\; p\in \R \text{ such that } \Gf(S,p)=0\}.
$$
It corresponds to the black dotted curve in \Cref{2PF_fig:hysPlot1} (right). Let $\ell=\ell(S)$, defined in \eqref{2PF_eq:DefG}, be the line passing through $U_B$ and $U_T$. If $\ell$  intersects  $F^{(i)}$ at $S=S_H$, then the vertical half-line $\{(S_H,p):p<\Pim(S_H)\}$ lies in $\H_0$ due to the definition of $F$ in \eqref{2PF_eq:FhysRelation}.  
Concerning $F^{(d)}$, $\ell$ has either zero, one or two intersection points, see \Cref{2PF_fig:fHysTW} (left). In the latter case, as before, $\H_0$ contains one or two vertical half-lines as 
shown in the (right) plot of \Cref{2PF_fig:fHysTW}. However, this aspect plays no major role in the analysis below. 

Every point in the set $\H_0\cap \H$ is an equilibrium point. However, all points in the set $\H_0\cap \mathrm{int}(\H)$ (the interior of $\H$ being referred to as $\mathrm{int}(\H)$ here)
are unstable and as follows from \Cref{2PF_fig:hysPlot1} (right), no orbit can reach these points as $\xi\to \infty$. 
This eliminates all other possibilities to reach $E_T$ as $\xi\to \infty$ except for $p_T=\Pim(S_T)$ and $p_T=\Pdr(S_T)$. 
\end{proof}

We now consider the two cases  separately. 

\subsection{Case (i): $S_B<S_T\leq \bar{S}_i$}
The main result of this section is
\begin{proposition}
Assume \eqref{2PF_eq:PerHysSb} and let $S_T\in (S_B,\bar{S}_i]$, $p_T=\Pim(S_T)$ and  $F^{(i)}(S_T)>F(S_B,p_B)$. 
Then a $\t^*_i(S_T)>0$ exists such that for all $\t<\t^*_i(S_T)$ there is an orbit satisfying \hyperref[2PF_eq:ODE]{(TW)} 
and connecting $E_B$ to $E_T$.  \label{2PF_prop:CaseFhysi}
\end{proposition}
\begin{proof}
Consider the orbit $(S_{(i)},p_{(i)})$ that leaves $E_B$ vertically through the half-line $\{S=S_B,p<p_B\}$. 
The directions of the orbits in $\H$ imply that $(S_{(i)},p_{(i)})$  intersects $\Pim(S)$ and enters $\Him$ 
(the region under the graph of $\Pim$) at some finite $\xi\in \R$, see \Cref{2PF_fig:fHysTW} (right). In $\Him$ 
its motion is governed by the system
\begin{equation}
\begin{cases}
{S_{(i)}}'=\frac{1}{c_{(i)}\t}(\Pim(S_{(i)})-p_{(i)}),\;\\
 {p_{(i)}}'=\Gf_i(S_{(i)}):=\Gf(S_{(i)},\Pim(S_{(i)})),
 \end{cases} \text{ with } c_{(i)}=\tfrac{F^{(i)}(S_T)-F(S_B,p_B)}{S_T-S_B}>0.
 \label{2PF_eq:CaseiGDE}
\end{equation} 
Note that, $\Gf_i(S)=(F^{(i)}(S)-\ell(S))\slash h^{(i)}(S)$. The system \eqref{2PF_eq:CaseiGDE} has exactly 
the same structure as \hyperref[2PF_eq:ODE]{(TW)} described in \Cref{2PF_sec:CaseA}. Defining $\tau_i^*(S_T)$ 
similar to $\tau_m$ in \Cref{2PF_prop:Tau_m}, the result follows directly.
\end{proof}

\begin{remark} 
Observe that, the construction fails if $F^{(i)}(S_T)<F(S_B,p_B)$ which is intuitive since the overall 
process is not infiltration in this case. If one prescribes a flux $F=F_T$ at $\xi\to \infty$ which is 
less than $F(S_B,p_B)$, then \Cref{2PF_prop:pTrestriction,2PF_prop:CaseFhysi} forces the saturation at 
$\xi\to \infty$ to be $S_T=(F^{(d)})^{-1}(F_T)< S_B$, reducing the problem to Case (ii). However, if one fixes 
the saturation $S_T$ so that  $F(S_B,p_B)>F^{(i)}(S_T)$, then we get a frozen profile with a $p_T\in (\Pim(S_T),\Pdr(S_T))$ 
that satisfies $F(S_T,p_T)=F(S_B,p_B)$. This is explained further in \Cref{2PF_sec:NumresScenarioB}. We set $\t^*_i(S_T)=\infty$ 
in this case.\label{2PF_rem:FSTsmaller}
\end{remark}

\Cref{2PF_prop:CaseFhysi} implies the following: 
\begin{corollary}
Under the assumptions of \Cref{2PF_prop:CaseFhysi}, let $S_{(i)}(\underline{\xi})=S$ 
for some $S\in (S_B,S_T]$ and $\underline{\xi}\in\R$. Define $\underline{w}(S;\t):=p_{(i)}(\underline{\xi})<\Pim(S)$. 
Then $\lim\limits_{\t\to 0} \underline{w}(S;\t)=\Pim(S)$.\label{2PF_cor:WgoesToPim}
\end{corollary}

\begin{figure}[H]
\begin{center}
\includegraphics[width=0.95\textwidth]{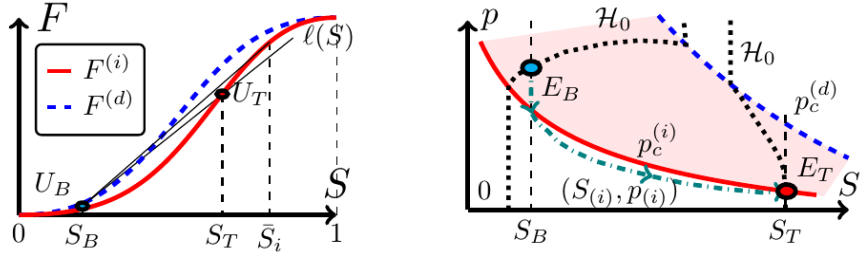} 
\end{center}
\caption{$U_B$, $U_T$, $\ell(S)$ used in \Cref{2PF_prop:CaseFhysi} in the $S$-$F$ plane. (right) 
The $S$-$p$ plane and the orbit $(S_{(i)},p_{(i)})$ for Case (i) with $F^{(i)}(S_T)>F(S_B,p_B)$ and $\t<\t^*_i(S_T)$.}
\label{2PF_fig:fHysTW}
\end{figure}

Here, $\underline{w}$ is the counterpart of $w$ defined in \Cref{2PF_sec:CaseA} for Scenario A. 
The proof of \Cref{2PF_cor:WgoesToPim} is based on the inequality  \eqref{2PF_eq:WlowerBoundDefPhi} which 
is satisfied in this case by $\underline{w}$. From \Cref{2PF_cor:WgoesToPim} one obtains that for Case (i), 
if $\t\searrow 0$, meaning that if the dynamic capillarity is vanishing, then the orbit follows either the scanning curve, 
here the line segment $\{S=S_B, \Pim(S_B)<p<p_B\}$, or the infiltration curve $\Pim$. The result is analogous to the results 
for capillary hysteresis given in \cite[Section 3]{VANDUIJN2018232}.

\subsection{Case (ii): $\bar{S}_d\leq S_T<S_B$ and stability of plateaus}
The counterpart of \Cref{2PF_prop:CaseFhysi} for Case (ii) is (see also \Cref{2PF_fig:fHysTWdr}),
\begin{proposition}
Assume \eqref{2PF_eq:PerHysSb} and let $S_T\in [\bar{S}_d, S_B)$, $p_T=\Pdr(S_T)$ and $F^{(d)}(S_T)<F(S_B,p_B)$. 
Then a $\t^*_d(S_T)>0$ exists such that for all $\t<\t^*_d(S_T)$ there is an orbit $(S_{(d)},p_{(d)})$ satisfying 
\hyperref[2PF_eq:ODE]{(TW)} and connecting $E_B$ to $E_T$. Moreover, for a fixed $S_{(d)}=S\in [S_T,S_B)$, one 
has $p_{(d)}\to \Pdr(S)$ as $\t\to 0$.\label{2PF_prop:CaseFhysii}
\end{proposition}
\begin{figure}[H]
\begin{center}
\includegraphics[width=0.75\textwidth]{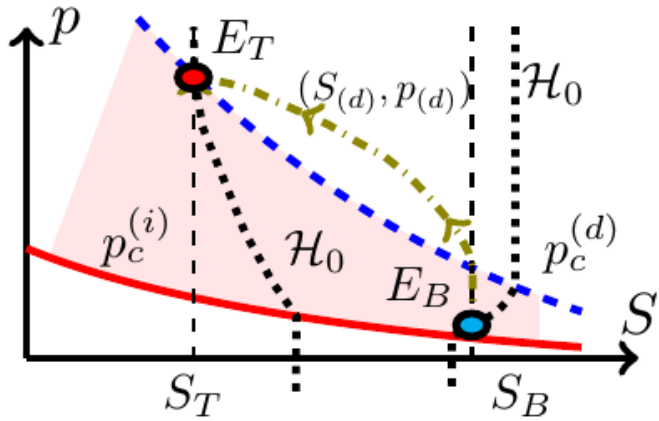}
\end{center}
\caption{The orbit $(S_{(d)},p_{(d)})$ for Case (ii) with $F^{(d)}(S_T)<F(S_B,p_B)$ and $\t<\t^*_d(S_T)$. }
\label{2PF_fig:fHysTWdr}
\end{figure}
Finally, we investigate a special case related to the development of stable saturation plateaus in 
infiltration experiments. For $S_B\in (0,1)$, and $S_T\in(S_B,1)$ a stable plateau is formed when an infiltration wave, 
from $S_B$ to $S_P\in (S_T,1)$, followed by a drainage wave, from $S_P$ to $S_T$, both have the same speed resulting
in the width of the plateau to remain constant. This is different from the plateaus described in 
\eqref{2PF_eq:CaseAentropyStauinB} where the speeds of the infiltration and the drainage fronts are 
necessarily different. The existence of stable saturation plateaus has been widely  
studied experimentally \cite{dicarlo2004experimental,shiozawa2004unexpected,gladfelter1980effect}
and numerically \cite{schneider2018stable,hilfer2014saturation}. Although  results regarding stability 
of the plateau are available \cite{schneider2018stable,hilfer2014saturation}, the mechanism behind its development 
is still not well understood. Here, we give an example where our analysis predicts that such a plateau will develop. 
Specifically, it occurs when $\t>\t^*_i(S_T)$ and a direct monotone orbit from $E_B$ to $(S_T,\Pim(S_T))$ is no longer 
possible. This is verified numerically in \Cref{2PF_sec:NumresScenarioB}.

\begin{proposition}
Assume \eqref{2PF_eq:PerHysSb} and let $S_T\in (S_B,1)$ be such that the line $\ell$ through $U_B=(S_B,F(S_B,p_B))$ and 
$U_T=(S_T,F^{(d)}(S_T))$ in the $F$-$S$ plane intersects $F^{(i)}$ at some $S=S_P\in (S_T,\bar{S}_i)$. 
Consider the system \hyperref[2PF_eq:ODE]{(TW)} with the wave-speed
 $$
c_P=\tfrac{F^{(d)}(S_T)-F(S_B,p_B)}{S_T-S_B}= \tfrac{F^{(i)}(S_P)-F(S_B,p_B)}{S_P-S_B}=\tfrac{F^{(i)}(S_P)-F^{(d)}(S_T)}{S_P-S_T}.
$$
For this system, let  $(S^P_{(i)},p^P_{(i)})$  be the orbit that passes through $\Him$ and connects to 
the equilibrium point $(S_B,p_B)$ as $\xi\to -\infty$, described in \Cref{2PF_prop:CaseFhysi}.  Similarly, 
let $(S^P_{(d)},p^P_{(d)})$  be the orbit passing through $\Hdr$ and connecting to $(S_P,\Pim(S_P))$ as $\xi\to -\infty$, 
described in \Cref{2PF_prop:CaseFhysii}.
Assume that $0<\t<\max\{\t^*_i(S_P),\t^*_d(S_T))\}$ where the $\t^*_i(S_P)$ and the $\t^*_d(S_T)$ values correspond 
to the orbits $(S^P_{(i)},p^P_{(i)})$ and  $(S^P_{(d)},p^P_{(d)})$ respectively. 
Then, $(S^P_{(i)},p^P_{(i)})\to (S_P,\Pim(S_P))$ as $\xi\to \infty$ and $(S^P_{(d)},p^P_{(d)})\to (S_T,\Pdr(S_T))$  
as $\xi\to \infty$.\label{2PF_prop:StablePlateau}
\end{proposition}

The proof follows  directly from \Cref{2PF_prop:CaseFhysi,2PF_prop:CaseFhysii}.

\subsection{Entropy solutions}\label{2PF_sec:EntropySolutionScenarioC}
We can now discuss the entropy solutions of the Riemann problem \eqref{2PF_eq:LimitingBL} under the
assumptions of Scenario B. To be more specific, we give a selection criteria for the solutions of the system
\begin{equation}
\begin{cases}
\dfrac{\p S}{\p t}+ \dfrac{\p F(S,p)}{\p z}=0, \\[.5em]
 p \in \frac{1}{2}(\Pdr(S)+\Pim(S))-\frac{1}{2}(\Pdr(S)-\Pim(S))\cdot\sgn\left(\frac{\p S}{\p t}\right),
 \end{cases}
 \text{  in } \R\times[0,\infty) \label{2PF_eq:DefHysRiemann}
\end{equation}
\begin{equation}
\text{ with }
S(z,0)=\begin{cases}
S_T &\text{ for } z<0,\\
S_B &\text{ for } z>0,
\end{cases} \text{ and }
p(z,0)=p_B \text{ for } z>0.
\end{equation}
 We view \eqref{2PF_eq:DefHysRiemann} as the limit of 
\hyperref[2PF_eq:FullRegularisedSystem]{($\mathcal{P}$)} when the capillary effects vanish. However, hysteresis is still present in the model.

Note that, $\t$ still plays a role in determining the entropy solution despite being absent in \eqref{2PF_eq:DefHysRiemann}. 
This is similar to what we saw in \Cref{2PF_sec:CaseA}.
However, the focus here being hysteresis in permeability and capillary pressure, for a fixed $S_B\in (0,1)$ we take
\begin{equation}
0<\t<\min\left\{\inf\limits_{S_T\in (S_B,\bar{S}_i]}\t^*_i(S_T),\inf\limits_{S_T\in [\bar{S}_d,S_B)}\t^*_d(S_T)\right\}.
\label{2PF_eq:SceneCentropyTaubound}
\end{equation}
Observe that, \eqref{2PF_eq:SceneCentropyTaubound} does not provide a void interval for $\t$. 
To see this, note that $\t^*_i(S_T)$ is defined similar to $\t_m$ in \Cref{2PF_prop:Tau_m} and thus, 
it satisfies the inequality in \eqref{2PF_eq:DefTauM}, i.e. it has the positive quantity $\bar{\t}_m$ 
as its lower bound. Although $\bar{\t}_m$  in \Cref{2PF_prop:Tau_m} actually depends on $S_T$, one sees 
from \eqref{2PF_eq:DefTauBarM} that the values of $\bar{\t}_m$ are bounded away from 0 uniformly with respect to $S_T$. 
Hence, $\t^*_i(S_T)$ is also bounded uniformly away from 0. Similar argument holds for $\t^*_d(S_T)$.

We now consider the cases $S_T>S_B$ and $S_T<S_B$ separately. 
\subsection*{\underline{$S_T>S_B$}}
If $S_T\leq \bar{S}_i$ ($\bar{S}_i$ introduced in Definition \ref{2PF_def:FhysQuatities}) and $F^{(i)}(S_T)>F(S_B,p_B)$  
then the entropy solution is a shock:
\begin{equation}
S(z,t)=
\begin{cases}
S_T &\text{ for } z<c_{(i)} t,\\
S_B &\text{ for } z>c_{(i)} t,
\end{cases} 
\text{ with } 
c_{(i)}=\frac{F^{(i)}(S_T)-F(S_B,p_B)}{S_T-S_B}.\label{2PF_eq:EntropyFhysCaseiA}
\end{equation}
For $F^{(i)}(S_T)<F(S_B,p_B)$, from Remark \ref{2PF_rem:FSTsmaller}, the solution is \eqref{2PF_eq:EntropyFhysCaseiA} but with $c_{(i)}=0$, i.e. it is a stationary shock.
However, if $S_T>\bar{S}_i$ then the solution becomes more complex, combining a rarefaction wave with a shock:
\begin{equation}
S(z,t)=\begin{cases}
S_T &\text{ for } z<{F^{(i)}}'(S_T)t,\\
r_{(i)}(z/t) &\text{ for } {F^{(i)}}'(S_T) t<z<{F^{(i)}}'(\bar{S}_i) t,\\
\bar{S}_i &\text{ for } {F^{(i)}}'(\bar{S}_i) t<z<c_{(i)}  t,\\
S_B &\text{ for } z>c_{(i)}  t.
\end{cases}\label{2PF_eq:EntropyFhysCaseiB}
\end{equation}
Here $r_{(i)}(\cdot)$ satisfies
$$
{F^{(i)}}'(r_{(i)}(\zeta))=\zeta, \text{ for } {F^{(i)}}'(S_T)\leq \zeta \leq {F^{(i)}}'(\bar{S}_i).
$$

\subsection*{\underline{$S_T<S_B$}}
If $S_T\geq \bar{S}_d$  then the entropy solution for $F(S_B,p_B)>F^{(d)}(S_T)$ is the shock
\begin{equation}
S(z,t)=
\begin{cases}
S_T &\text{ for } z<c_{(d)} t,\\
S_B &\text{ for } z>c_{(d)} t,
\end{cases} 
\text{ with } 
c_{(d)}=\frac{F(S_B,p_B)- F^{(d)}(S_T)}{S_B - S_T},\label{2PF_eq:EntropyFhysCaseiiA}
\end{equation}
and for $F(S_B,p_B)<F^{(d)}(S_T)$ it is \eqref{2PF_eq:EntropyFhysCaseiiA} with $c_{(d)}=0$. If $S_T<\bar{S}_d$ then 
the solution has a similar structure to \eqref{2PF_eq:EntropyFhysCaseiB}, i.e.
\begin{equation}
S(z,t)=\begin{cases}
S_T &\text{ for } z<{F^{(d)}}'(S_T)t,\\
r_{(d)}(z/t) &\text{ for } {F^{(d)}}'(S_T) t<z<{F^{(d)}}'(\bar{S}_d) t,\\
\bar{S}_d &\text{ for } {F^{(d)}}'(\bar{S}_d) t<z<c_{(d)}  t,\\
S_B &\text{ for } z>c_{(d)}  t,
\end{cases} \label{2PF_eq:EntropyFhysCaseiiB}
\end{equation}
with the function $r_{(d)}(\cdot)$ satisfying
$$
{F^{(d)}}'(r_{(d)}(\zeta))=\zeta, \text{ for } {F^{(d)}}'(S_T)\leq \zeta \leq {F^{(d)}}'(\bar{S}_d).
$$

\section{Numerical results}
\label{2PF_sec:NumRes}

For the numerical experiments, we solve $( \tilde{\mathcal{P}}  )$ (System \eqref{2PF_eq:SystemPtilde})  in a domain $\left(z_{in},z_{out} \right)$, where $z_{in}<0$ and
$z_{out}>0$. As an initial condition for the saturation variable, we choose a smooth and monotone approximation of the Riemann data:
\begin{equation}
\label{eq:initial_condition}
S \left( z,0 \right) =
\begin{cases}
S_T \text{ for } z<-l, \\[.2em]
\frac{\left(S_B+S_T\right)}{2} + \frac{ \left(S_T-S_B\right)}{4 l^3} z \cdot (z^2 - 3 l^2) \text{ for } -l \leq z \leq l, \\[.2em]
S_B \text{ for } z>l.
\end{cases}
\end{equation}
Here, $l$ is a smoothing parameter, $S_T$ denotes the saturation induced by a certain injection rate and $S_B$ is the initial 
saturation within the porous medium.
In order to model the capillary pressure, a van Genuchten parametrisation is considered, i.e.
$$
p_c^{(j)}(S)=\Lambda_j (S^{-\frac{1}{m_j}}-1)^{1-m_j},\; j \in \left\{ i,d \right\}.
$$
In the remainder of this section we use the following parameter set:  $\Lambda_i=3.5$, $m_i=0.92$, $\Lambda_d=7$ and $m_d=0.9$.  
To solve  $( \tilde{\mathcal{P}}  )$ numerically, for $n\in \N\cup\{0\}$ and $t_0=0$, we solve within the time step 
$$
\left[t_n,t_{n+1} \right] \text{ of width } \Delta t_n = t_{n+1}- t_n,
$$ 
 the elliptic problem
$$
- \frac{\partial}{\partial z} \left( F\left(S,p \right) + h\left(S,p \right) \frac{\partial p}{\partial z} \right) 
= \frac{1}{\tau} \mathcal{F}\left(S,p \right),
$$
with respect to the pressure variable $p$. For a given $S$, this is a nonlinear elliptic problem and to solve 
it, a linear iterative scheme is employed 
which is referred to as the L-scheme in literature \cite{pop2004mixed,List2016,MITRA2018}:
$$
L\left( p_n^i - p_n^{i-1} \right) - \frac{\partial}{\partial z} \left( F\left(S_n,p_n^{i-1} \right) + h\left(S_n,p_n^{i-1} \right) \frac{\partial p_n^i}{\partial z} \right)
= \frac{1}{\tau} \mathcal{F}\left(S_n,p_n^{i-1} \right).
$$
Here, $p_n^i$ denotes the pressure at the $i^{\mathrm{th}}$ iteration and $p_n^0 = p\left(z,t_n\right)$. On closer examination, the L-scheme corresponds to a linearization of the
nonlinear problem, since for each iteration a linear equation in the unknown pressure variable $p_n^i$ is solved. For Scenario A the parameter $L$ is set to $L = \frac{1}{\tau}$ to ensure convergence of the L-scheme  \cite{pop2004mixed,mitra2018wetting} and for Scenario B the modified variant of the L-scheme is used \cite{MITRA2018,VANDUIJN2018232} to speed up the convergence, since in this scenario the stiffness matrix has to be recomputed in every iteration. A standard cell centered finite volume scheme is considered for discretizing the 
linearised elliptic problem in space. Having the pressure variable $p_n$ and the saturation variable $S_n$ for $t=t_n$ at hand, 
we update the saturation as follows:
$$
S_{n+1} = S_n + \frac{\Delta t_n}{\tau}  \mathcal{F}\left(S_n,p_n \right).
$$

\subsection{Numerical results for Scenario A} 

First we illustrate the  theoretical findings of Scenario A. The boundary conditions with respect to the pressure 
variable are of Neumann type at $z=z_{in}$ and of Dirichlet type at $z=z_{out}$:
\begin{equation}
\label{eq:boundary_conditions_Scenario_A}
 p^\prime \left( z_{in}, t \right)  =0 \text{ and }  p\left( z_{out}, t \right) = \Pim(S_B) \text{ for all } t>0.
\end{equation}
The boundaries of the domain are given by: $z_{in}= -10$ and $z_{out}= 500$. Since we do not include 
hysteresis in the relative permeabilites, the flux function $F$ depends only on $S$ and is determined by:
$$
f \left(S \right) = \frac{S^2}{S^2 + \left(1-S\right)^2} \text{ and } N_g =1.
$$
The numerical results presented 
in this subsection are related to $t = t_{\text{end}} = 300$. 
For the parameters of the initial condition, we take:
$$
S_B = 0.1,\quad S_T = 0.4 \;\text{ and } l=1.
$$ 
Based on these data, some of the variables and constants occurring in 
\Cref{2PF_sec:prelim} and \Cref{2PF_fig:Sdefs} are computed, i.e:
\begin{equation}
\tilde{S} \approx 0.3138,\; \bar{S} \approx 0.5909,\;S_o \approx 0.4393,\;S_\ast \approx 0.4111 
\text{ and } S^\ast \approx 0.8132. \label{2PF_eq:NumResSdefs}
\end{equation}
Moreover, the curves for $\gamma$ and $\beta$ are determined (see  \Cref{fig:Char_Saturations}). 
Observe that, from our choice, $S_B < S_o$ and $S_T \in \left(S_B,S_*\right]$. 
\begin{figure}[h!]
\centering
\includegraphics[scale=0.4,clip={10 10 0 0}, clip]{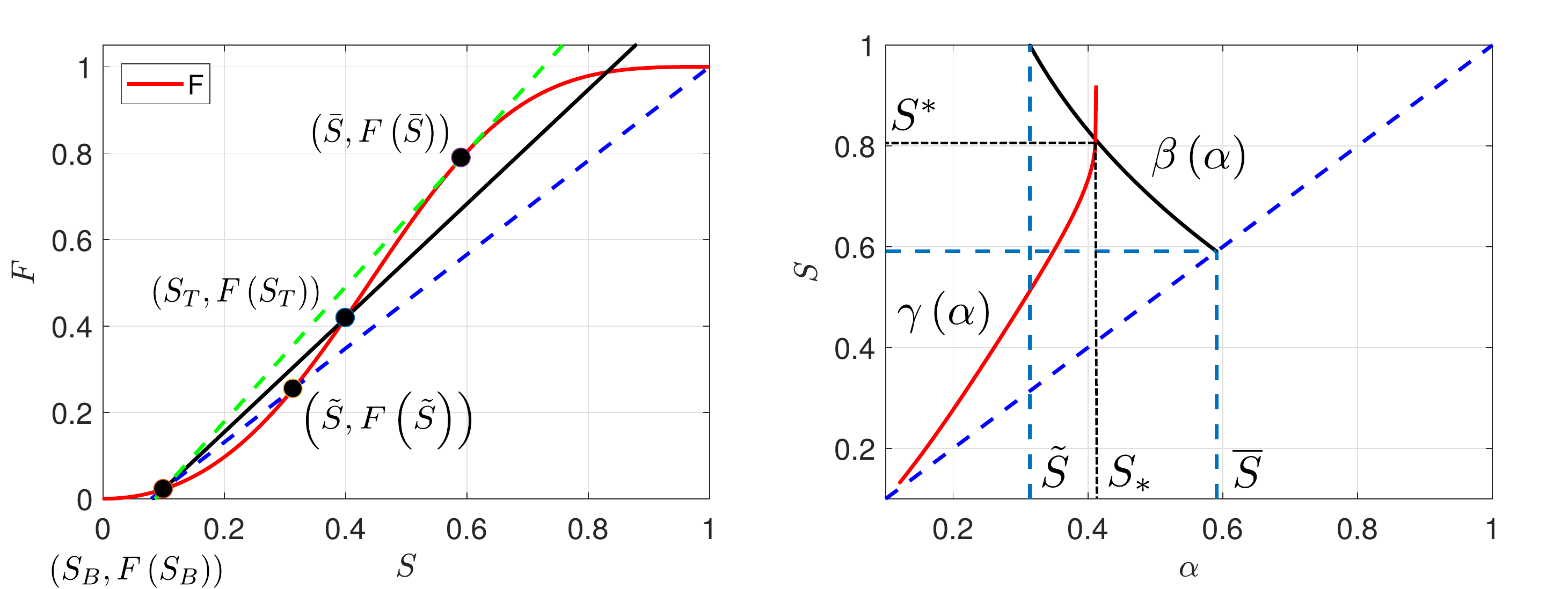}
\caption{\label{fig:Char_Saturations} Fractional flow function $F$ for Scenario A (left). 
The characteristic points $\tilde{S}$, $\bar{S}$ and $S_T$ are shown.  (right)
Curves for $\gamma$ (red) and $\beta$ (black) corresponding to $F$. The intersection point of these curves 
is denoted by $\left(S_\ast,S^\ast\right)$.}
\end{figure}
Next, the characteristic $\tau$-values for drainage and imbibition are computed. Using  \eqref{2PF_eq:ExpressionEigenValues} 
and given parameters, we obtain:
$$
\tau_i = 0.0452 \text{ and } \tau_d = 0.2620.
$$
Since the requirements listed in \Cref{2PF_theo:SleqSstar} are all fulfilled, we can compare the numerical results with the claims contained
in the theorem. For this purpose, we choose $\tau$ from the following set:
$$
\tau \in  \left\{ 0.045,\;0.25,\;1.0,\;2.0 \right\},
$$
and study the resulting $S$-$p$ orbits. Considering \Cref{2PF_fig:orbits}, it can be observed that
for $\tau < \tau_i$ monotone saturation waves are produced by the numerical model linking $E_{S_B}^i$ and $E_{S_T}^i$. 
In the other cases, a saturation overshoot can be detected, where for $\tau_i < \tau < \tau_d$ the orbit 
ends up at the equilibrium point $E_{S_T}^d$ and for $\t>\t_d$ the orbits spiral around the segment $\overline{E_{S_T}^i E_{S_T}^d}$. 
If we choose larger values of $\t$, the corresponding $S_m(\t,S_T)$ value of the orbit increases. This supports the claims of 
\Cref{2PF_cor:ContinuousDependence} and \Cref{2PF_prop:Wtau}. Similar results including variation of saturation with $\xi$ can be found 
in \cite{mitra2018wetting}. 
\begin{figure}[h!]
\centering
\includegraphics[scale=0.395]{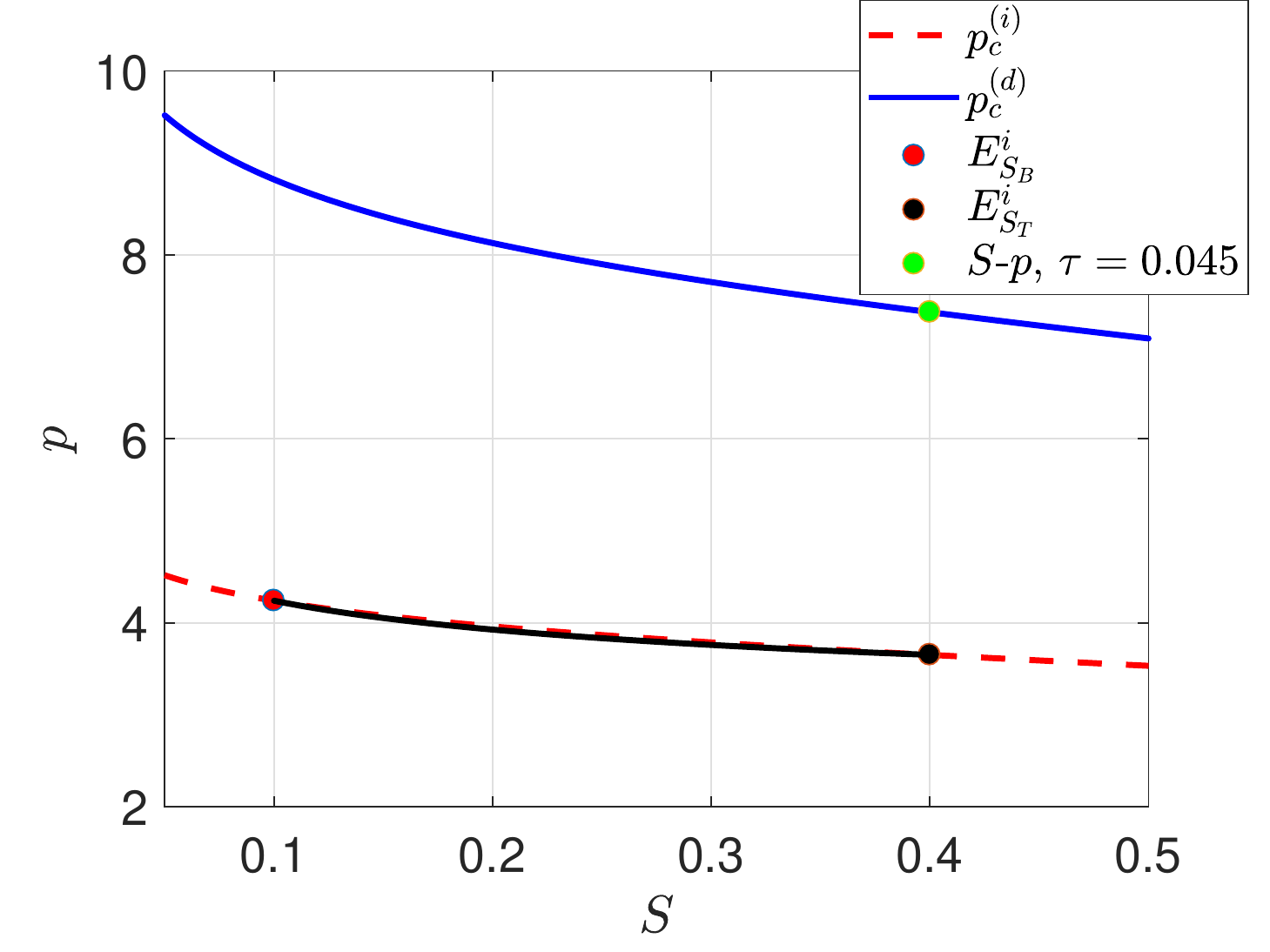}
\includegraphics[scale=0.395]{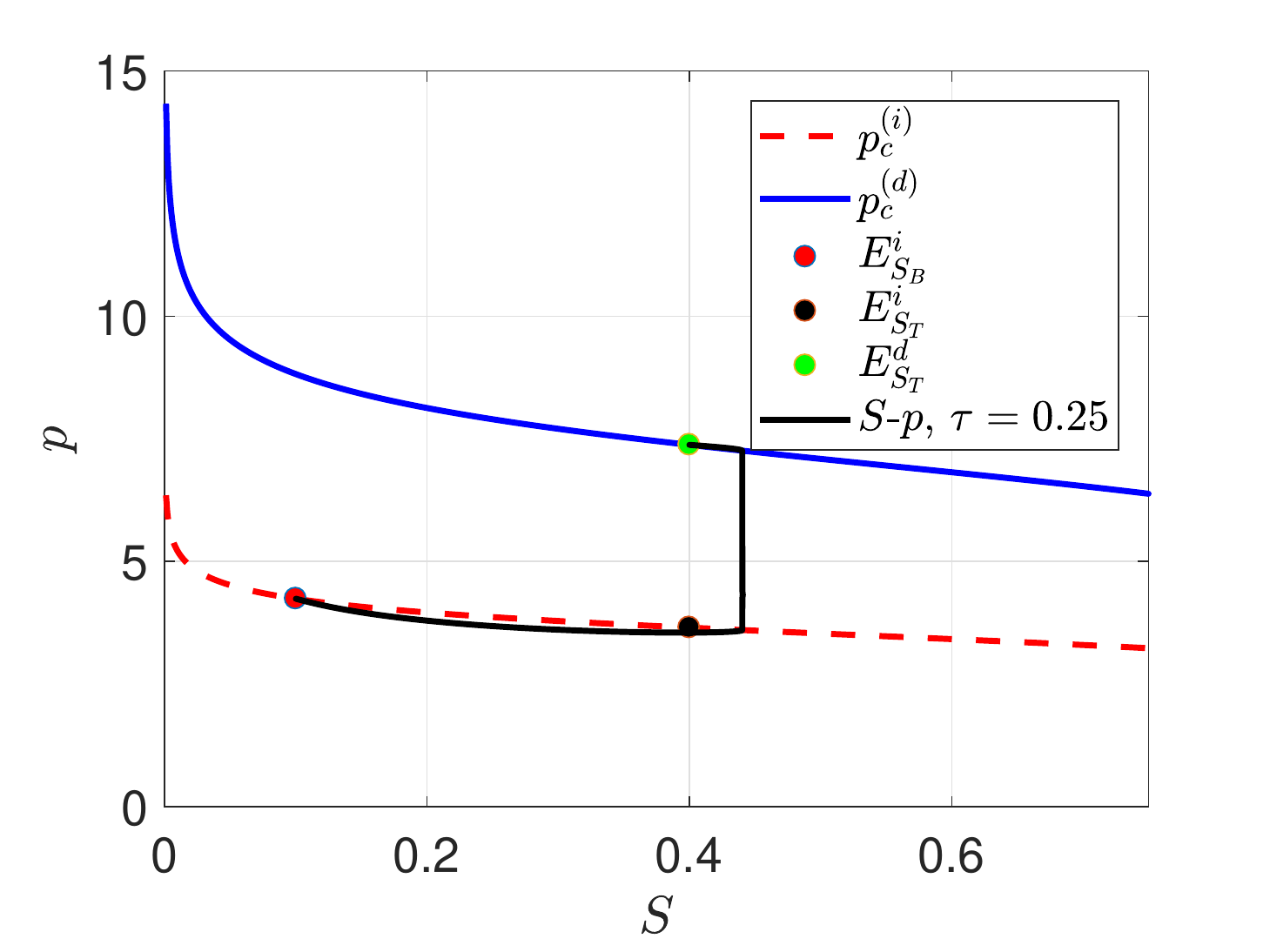}
\\
\includegraphics[scale=0.395]{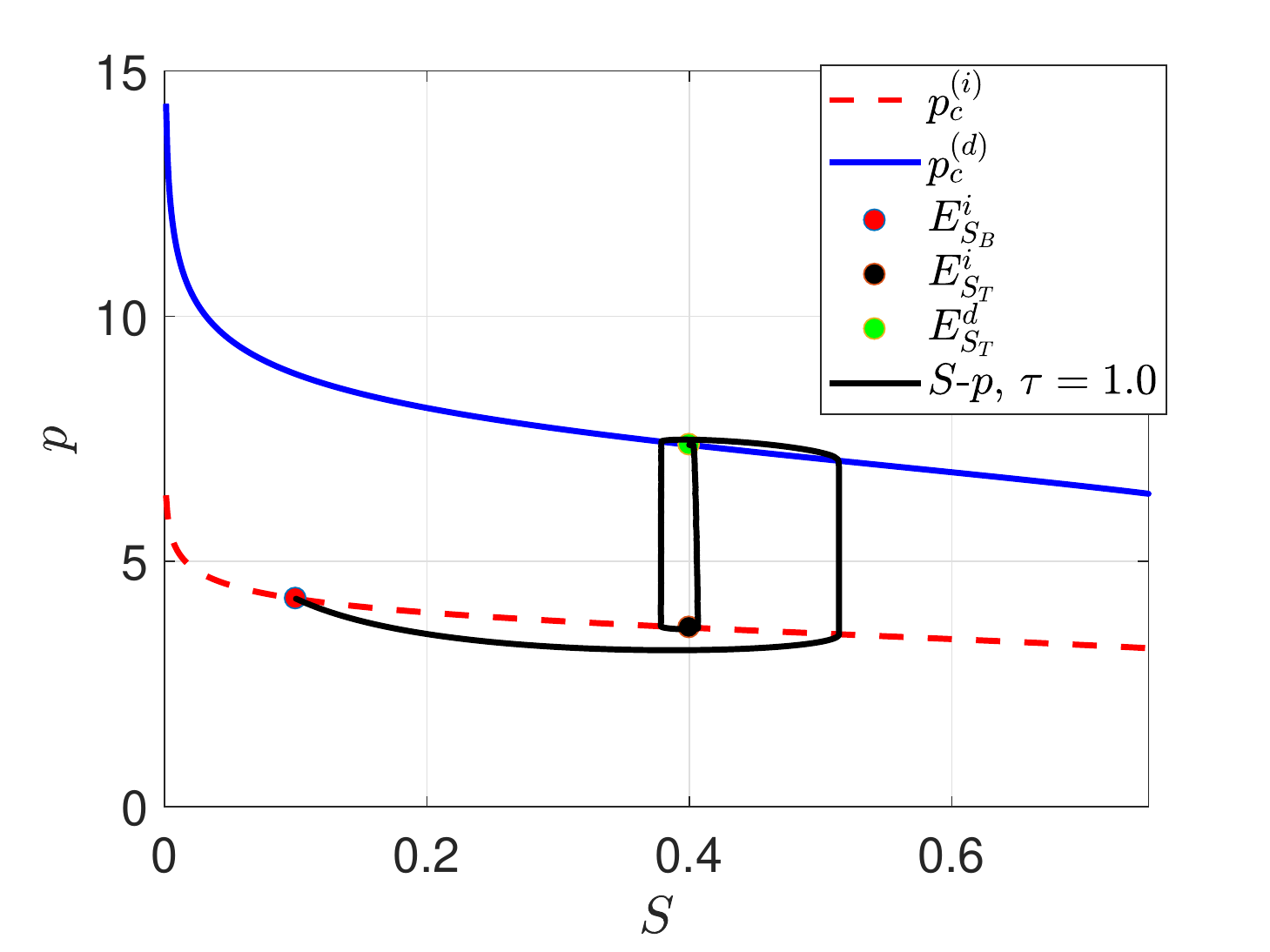}
\includegraphics[scale=0.395]{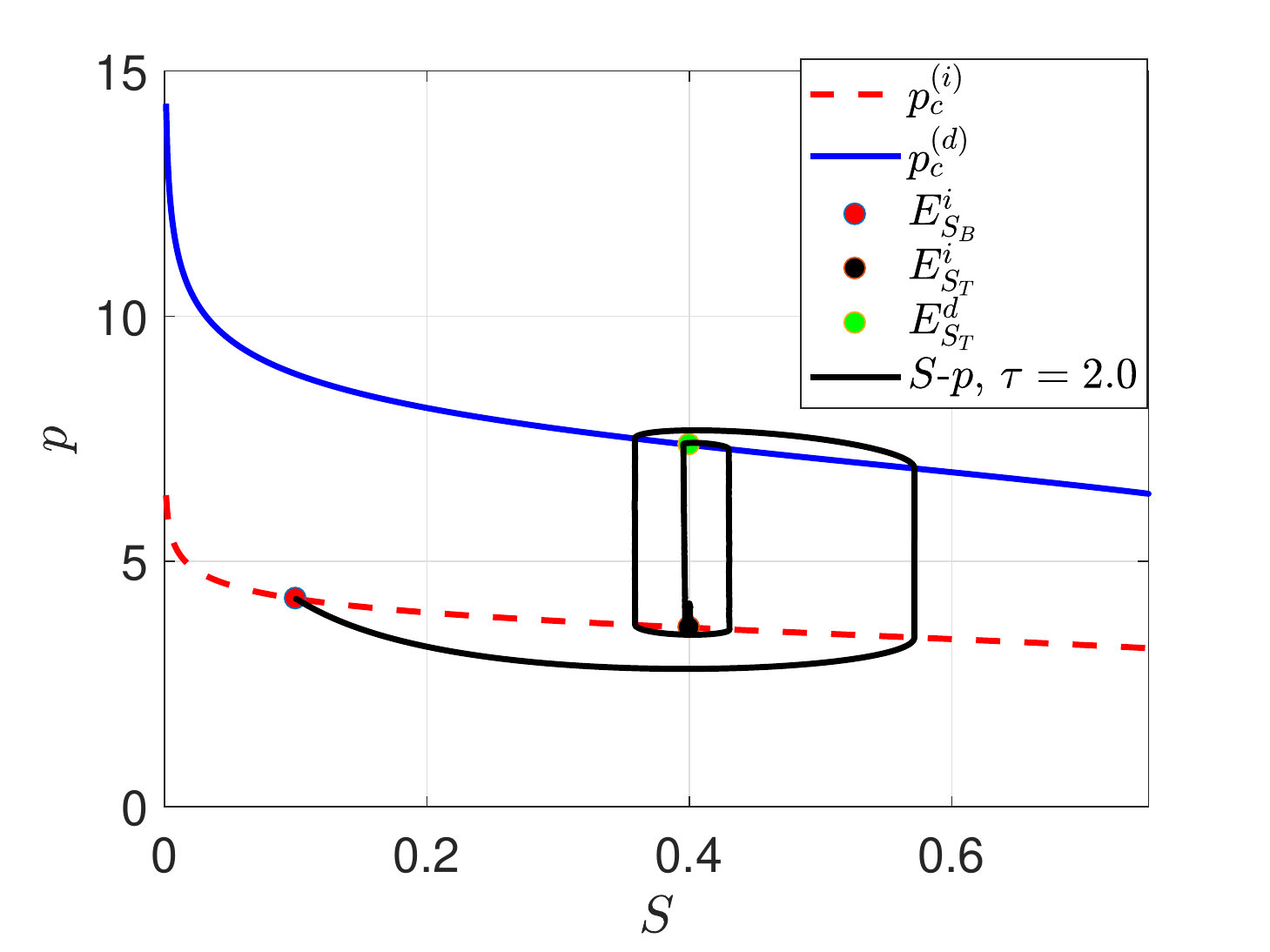}
\caption{\label{2PF_fig:orbits} Orbits for different $\tau$ parameters in the $S$-$p$ plane.}
\end{figure}
The parameter choice considered so far, corresponds to the solution class $\mathcal{A}$ (see \eqref{2PF_eq:solution_classes}), 
whose entropy  solution consists of a single shock without any saturation overshoots 
(see \Cref{2PF_fig:entropy_sol} (top)). However, there are two further solution classes, $\mathcal{B}$ and $\mathcal{C}$ 
(see  \eqref{2PF_eq:solution_classes}), arising in the context of
Scenario A, represented by entropy solutions \eqref{2PF_eq:CaseAentropyStauinB} and \eqref{2PF_eq:SolCaseC}. 
In case of solution class $\mathcal{B}$,
the entropy solution is given by saturation plateau that is formed by an infiltration wave followed by a drainage wave. 
The saturation at plateau level
is denoted by $\hat{S}_B\left(\tau\right)$. For solution class $\mathcal{C}$, 
the entropy solution exhibits a rarefaction wave connecting $S_T$ with $\hat{S}_B\left(\tau\right)$, which is connected to $S_B$ by a shock.
\begin{figure}[h!]
\centering
\includegraphics[scale=.35]{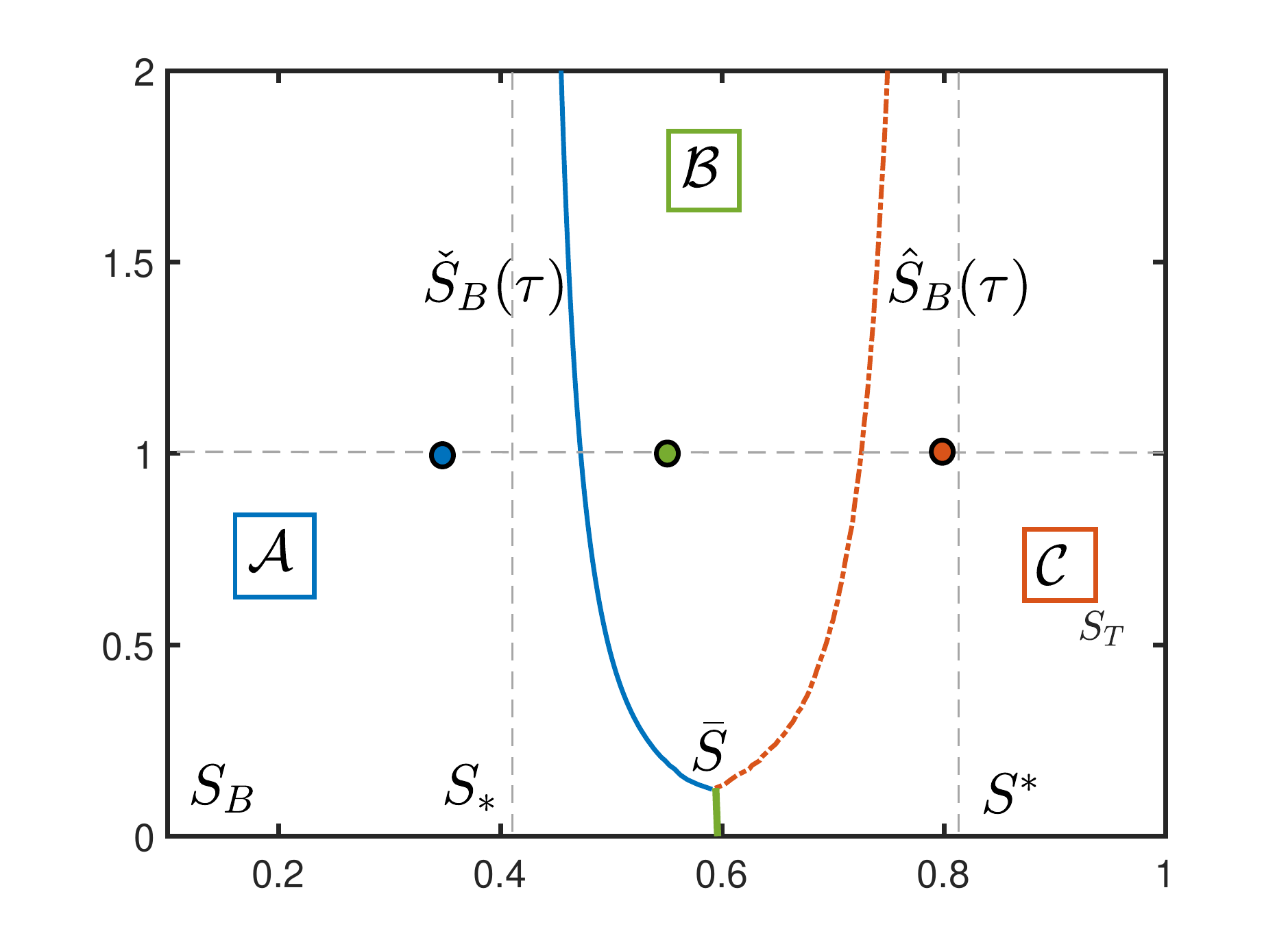}
\caption{\label{2PF_fig:NumResBifurcationDiagram} The $\hat{S}_B(\t)$ and $\check{S}_B(\t)$ curves computed for $S_B=0.1$. 
The characteristic saturations are as in \eqref{2PF_eq:NumResSdefs}. The corresponding Sets $\mathcal{A},\; \mathcal{B}$ and 
$\mathcal{C}$ along with $(S_T,\t)$ test pairs used in \Cref{2PF_fig:entropy_sol} are shown.}
\end{figure}
\begin{figure}[H]
\begin{minipage}{.45\textwidth}
\centering
\includegraphics[scale=.33]{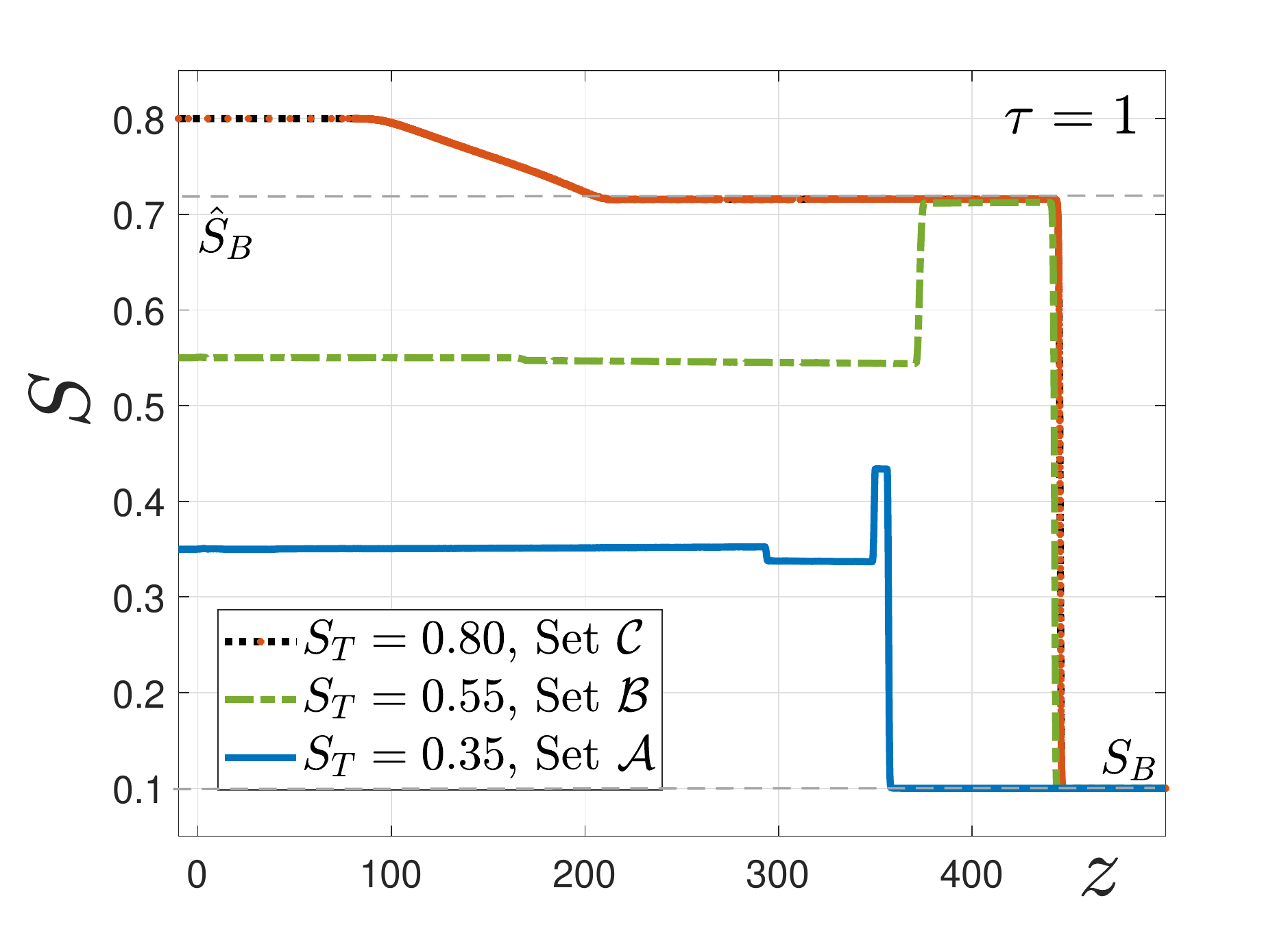}
\end{minipage}
\begin{minipage}{.45\textwidth}
\centering
\includegraphics[scale=.33]{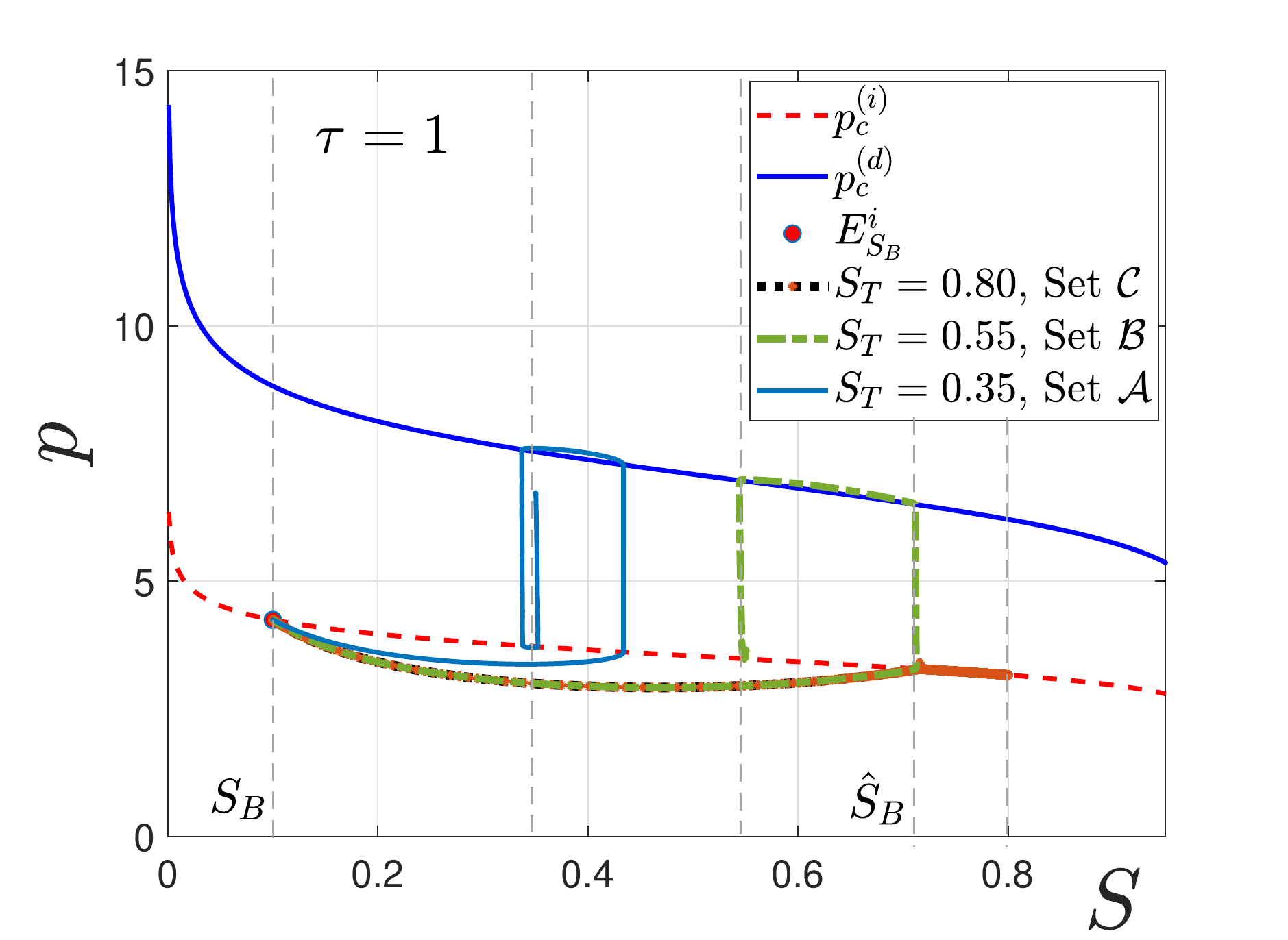}
\end{minipage}
\caption{\label{2PF_fig:entropy_sol} Numerical solutions corresponding to  different $(S_T,\t)$ pairs 
from solution classes $\mathcal{A}$, $\mathcal{B}$ and  $\mathcal{C}$, marked in \Cref{2PF_fig:NumResBifurcationDiagram}. 
Here, $\t=1$ is fixed and $S_T$ is chosen from $\{0.35,0.55,0.8\}$. The (left) plot shows the variation of $S$ with $z$, 
whereas, the (right) plot shows $p$ vs. $S$. The saturation plateau for the Sets $\mathcal{B}$ and  $\mathcal{C}$ is observed at $\hat{S}_B=.7158$.}
\end{figure}
To observe these cases numerically, we compute the $\hat{S}_B(\t)$ and $\check{S}_B(\t)$ curves introduced in Definition 
\ref{2PF_def:Supdown_i}, see \Cref{2PF_fig:NumResBifurcationDiagram}. In the figure we fix $\t=1$ and vary $S_T$ so that the 
pairs $(S_T,\t)$ belong to one of the sets $\mathcal{A},\; \mathcal{B}$ and $ \mathcal{C}$. The results are shown in \Cref{2PF_fig:entropy_sol} 
with the (left) plot showing the variation of $S$ with $z$, and the (right) plot showing the profiles in the $S$-$p$ phase plane.  
The curves corresponding to Set $\mathcal{A}$ show a direct travelling wave connecting $S_B$ and $S_T=0.35$. Some oscillatory 
behaviour around $S_T$ can be observed since $\t$ is comparatively large, however, the existence of a single travelling wave 
between $S_B$ and $S_T$ implies that these states are connectable by an admissible shock in the hyperbolic limit. 
Next, choosing $S_T=0.55$, $(S_T,\t)$ lies in Set $\mathcal{B}$, and a solution consisting of an infiltration wave followed 
by a drainage wave is computed in accordance with the theory. Again, small oscillations are seen in the drainage wave part which 
is expected from \Cref{2PF_cor:DrainageOscillation} since $\t$ is large. The resulting plateau has saturation $0.7158$, whereas, 
the prediction from \Cref{2PF_fig:NumResBifurcationDiagram} is $\hat{S}_B(\t)=0.7254$. 
Finally, for $S_T = 0.8$, the pair $(S_T,\t)$ belongs to the Set $\mathcal{C}$. The numerical solution exhibits 
a shock-like structure followed by a plateau and they coincide with the infiltration wave of Set $\mathcal{B}$ on 
both plots of \Cref{2PF_fig:entropy_sol}. Moreover, a rarefaction wave between $\hat{S}_B\left(\t\right)$ and $S_T$ is detected. 
Thus, we conclude that the saturation profiles in \Cref{2PF_fig:entropy_sol} correspond to the entropy solutions depicted 
in \Cref{2PF_fig:EntropySol} and the numerical results are in agreement with the theory.

\subsection{Numerical results for Scenario B} \label{2PF_sec:NumresScenarioB}

In case of Scenario B, we choose the following boundary conditions with respect to the pressure variable. As in the previous subsection, they are
of Neumann type at $z=z_{in}$ and of Dirichlet type at $z=z_{out}$: 
\begin{equation}
\label{2PF_eq:boundary_conditions_Scenario_B}
 p^\prime \left( z_{in}, t \right)  =0  \text{ and } p\left( z_{out}, t \right) = \Pdr(S_B) \text{ for all } t>0.
\end{equation}
Moreover,
the boundaries of the domain are given by: $z_{in}= -10$ and $z_{out}= 190$.
To make matters interesting, contrary to the previous subsection, we do not start with an infiltration state for $S_B$, but with a drainage state. 
Due to the fact that we consider hysteresis both in the capillary pressure and relative permeabilities, 
fractional flow functions are introduced both for infiltration and for drainage. We use
$$
f^{(i)}(S)=\frac{S^2}{S^2+ 3(1-S^2)},\; f^{(d)}(S)=\frac{S^2}{S^2+ 2(1-S^2)} \text{ with } N_g = 0,
$$
and define $F^{(i)}$ and $F^{(d)}$ accordingly.
We verified numerically that if $S_T>S_B$ and $F^{(i)}(S_T)<F(S_B,p_B)$ then the solution is frozen in time in the sense 
that $S(z,t)=S(z,0)$ for all $t>0$. This is what was discussed in Remark \ref{2PF_rem:FSTsmaller}.
To verify \Cref{2PF_prop:CaseFhysi,2PF_prop:CaseFhysii} and entropy solutions 
\eqref{2PF_eq:EntropyFhysCaseiA}-\eqref{2PF_eq:EntropyFhysCaseiiB}, we show two results:  
$S_B=S_{B,1}=0.3$, $S_T=S_{T,1}=0.95$ and $S_B=S_{B,2}=0.95$, $S_T=S_{T,2}=0.3$ both for $\t=0.02$. 
Let the corresponding solutions be $(S_{(i)},p_{(i)})$ and $(S_{(d)},p_{(d)})$. Since $S_{T,1}>\bar{S}_i$ for the first 
case (see Definition \ref{2PF_def:FhysQuatities}) and $\t$ is small, from \eqref{2PF_eq:EntropyFhysCaseiB} it is expected
that the entropy solution will have a shock from $S_{B,1}$ to $\bar{S}_i$, followed by a rarefaction wave from 
$\bar{S}_i$ to $S_{T,1}$. This is exactly what is seen from the viscous profiles obtained numerically, see \Cref{2PF_fig:FhysTWRW}. 
Similarly, for the second case, since $S_{T,2}<\bar{S}_d$ and $\t$ is small, we see from \Cref{2PF_fig:FhysTWRW} a viscous 
solution resembling a drainage shock followed by a rarefaction wave, as predicted in \eqref{2PF_eq:EntropyFhysCaseiiB}. 
\begin{figure}[h!]
\begin{minipage}{0.45\textwidth}
\centering
\includegraphics[scale=0.33]{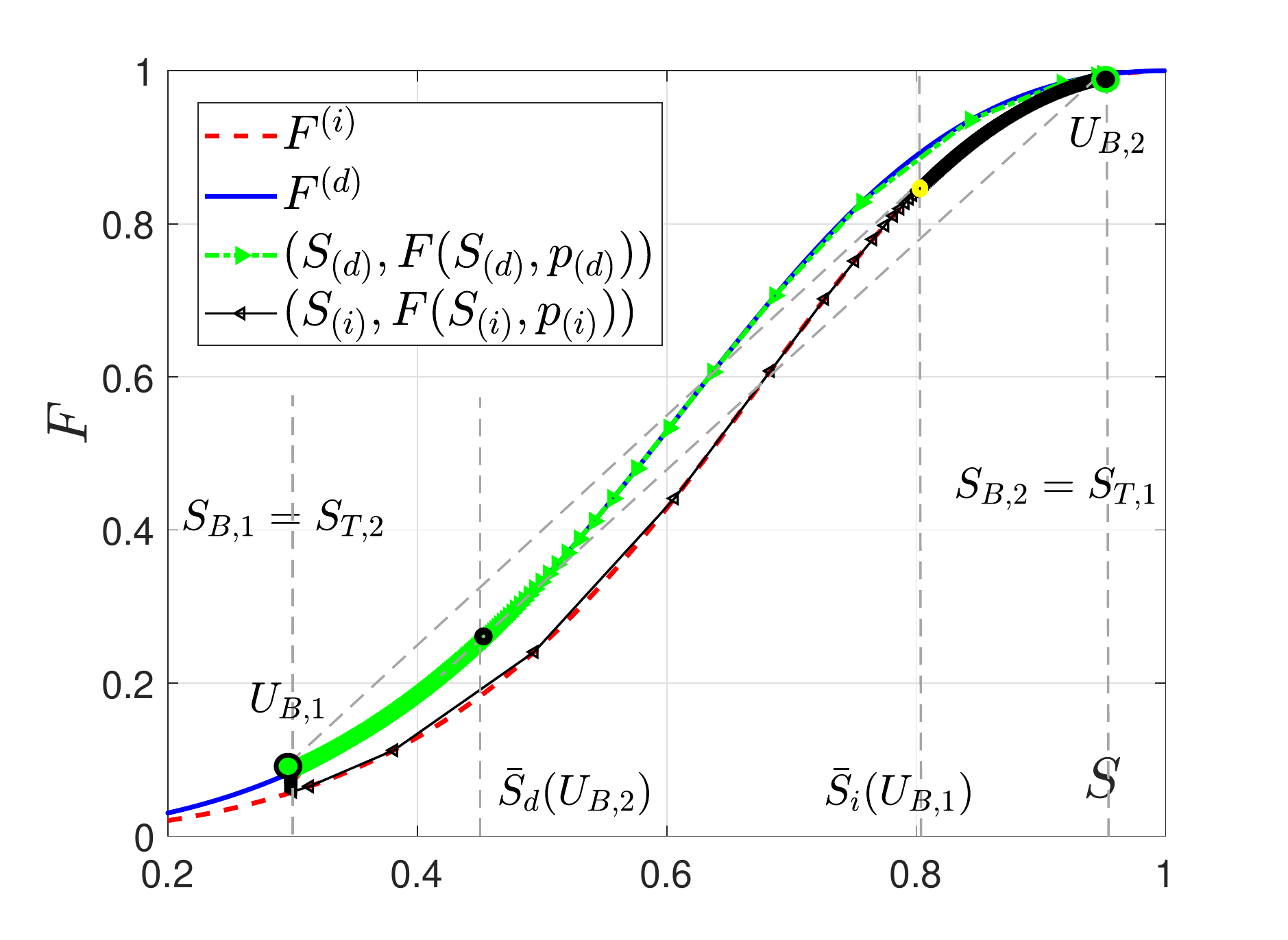}
\end{minipage}
\begin{minipage}{.45\textwidth}
\centering
\includegraphics[scale=0.33]{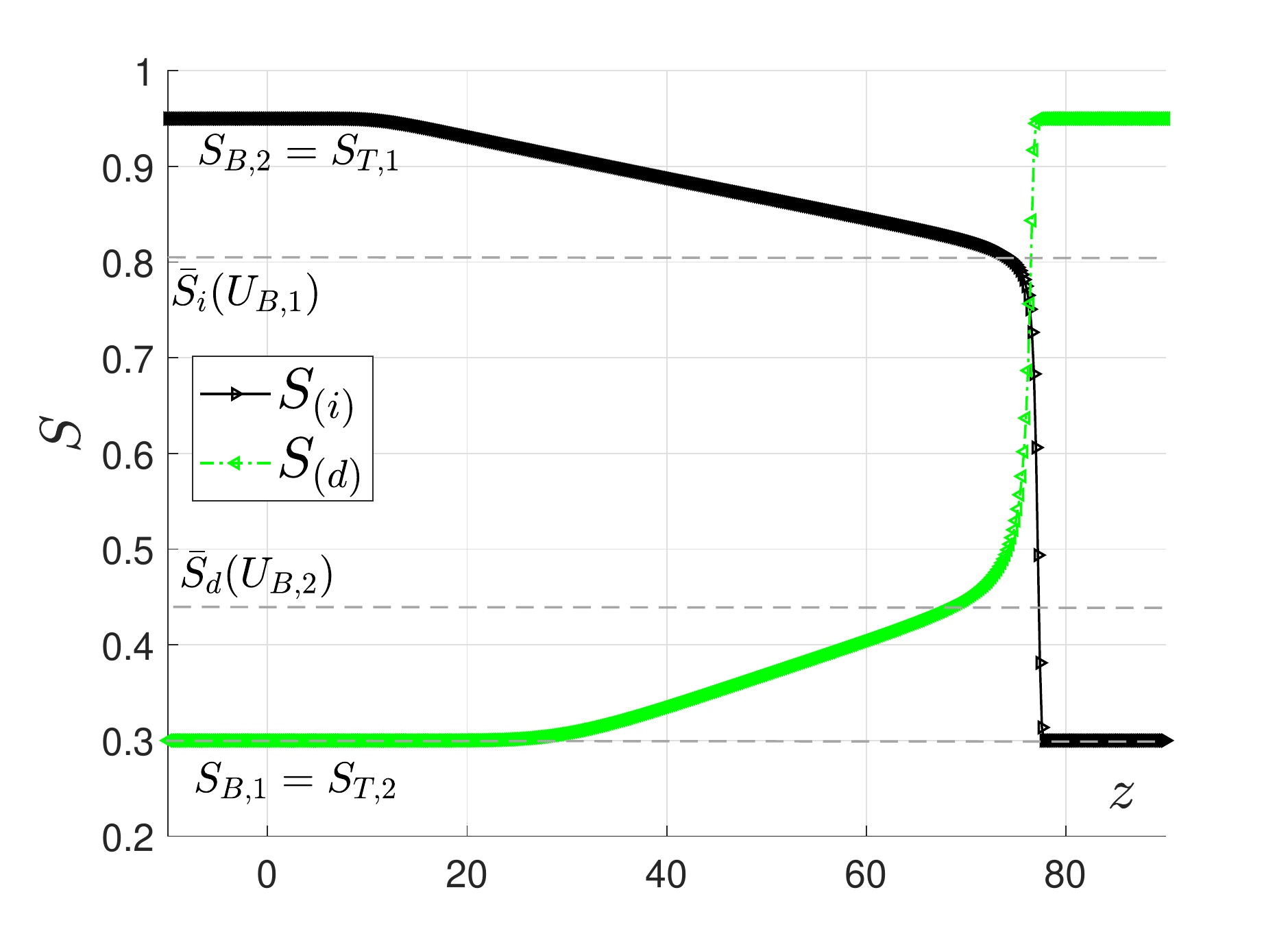}
\end{minipage}
\caption{The viscous solutions for $S_{B,1}=0.3,\; S_{T,1}=0.95$ denoted by $(S_{(i)},p_{(i)})$ and 
$S_{B,2}=0.95,\; S_{T,2}=0.3$ denoted by $(S_{(d)},p_{(d)})$ with boundary conditions \eqref{2PF_eq:boundary_conditions_Scenario_B} 
and $\t=0.02$ fixed. In the (left) plot, the solutions are shown in the $F$-$S$ plane  and in the (right) plot the saturations 
are plotted as functions of $z$. The points $U_{B,1}$ and $U_{B,2}$ and the saturations $\bar{S}_i(U_{B,1})$ and $\bar{S}_d(U_{B,2})$, 
introduced in Definition \ref{2PF_def:FhysQuatities}, are marked. The results agree with the predictions of 
\Cref{2PF_prop:CaseFhysi,2PF_prop:CaseFhysii} and \Cref{2PF_sec:EntropySolutionScenarioC}.}\label{2PF_fig:FhysTWRW}
\end{figure}
Next, we investigate whether a stable plateau is formed  for suitable parameter values by an infiltration wave 
and an ensuing drainage wave, as predicted in \Cref{2PF_prop:StablePlateau}. This happens only if $\t>\t^*_i(S_T)$, 
since in this case, a monotone connection between $(S_B,p_B)$ and $(S_T,\Pim(S_T))$ does not exist. To this end, in 
the numerical experiment we have used the following parameters:
 $$
S_B=0.3,\quad S_T=0.5\; \text{ and } \t=0.5. 
 $$
For a stable saturation plateau, the  velocities of the infiltration wave, connecting $S_B$ and $S_P$, and the drainage wave, 
connecting $S_P$ and $S_T$, have to be equal, i.e. if $c^P_{(i)}$ and $c^P_{(d)}$ are denoting the two wave-speeds, then
$$
c_{(i)}^P = \dfrac{F^{(i)}(S_P)-F^{(d)}(S_B)}{S_P-S_B} = \dfrac{F^{(i)}(S_P)-F^{(d)}(S_T)}{S_P-S_T} = c^P_{(d)},
$$
where $S_P$ stands for the saturation of the plateau. 
Geometrically, this equality is fulfilled, if the points 
$$
\left(S_B,F\left(S_B,p_B\right)\right), \left(S_T,F^{(d)}\left(S_T\right)\right) \text{ and } \left(S_P,F^{(i)}\left(S_P\right)\right)
$$ 
are located on the same line. This is precisely the condition that the solutions $(S^P_{(i)},p^P_{(i)})$ and 
$(S^P_{(d)},p^P_{(d)})$ of \Cref{2PF_prop:StablePlateau} satisfy.
Drawing a line through the given points for $S_B = 0.3$ and $S_T=0.5$ (see \Cref{2PF_fig:F_Points}), 
we obtain that a stable plateau should be located at $S_P \approx 0.634$. As seen  from \Cref{2PF_fig:F_Points}, 
the orbit in the $S$-$F$ plane stabilizes exactly at $S_P \approx 0.634$ and all the three points line up. 
Considering \Cref{2PF_fig:S_Plateaus}, we observe that the
saturation plateau is in a transient state in the beginning, but it stabilizes at $S_P \approx 0.634$ for longer 
times, as the speeds of the infiltration and drainage waves match.
\begin{figure}[h!]
\centering
\includegraphics[scale=0.5]{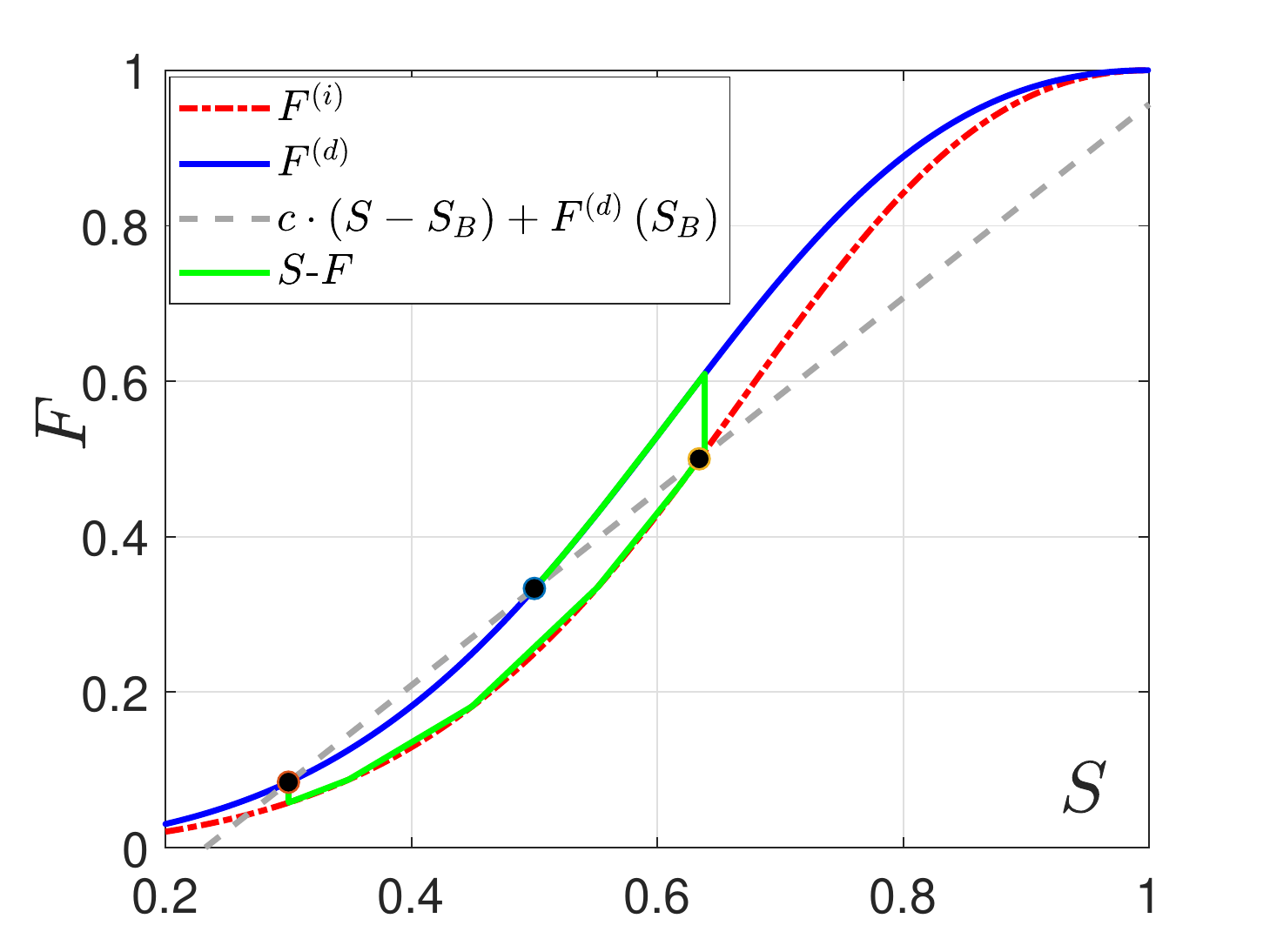}
\caption{\label{2PF_fig:F_Points} The orbit in the $S$-$F$ plane representing  a stable saturation plateau for Scenario B. 
The equilibrium points  for this orbit are shown on the flux curves.}
\end{figure}
\begin{figure}[h!]
\centering
\includegraphics[scale=0.38]{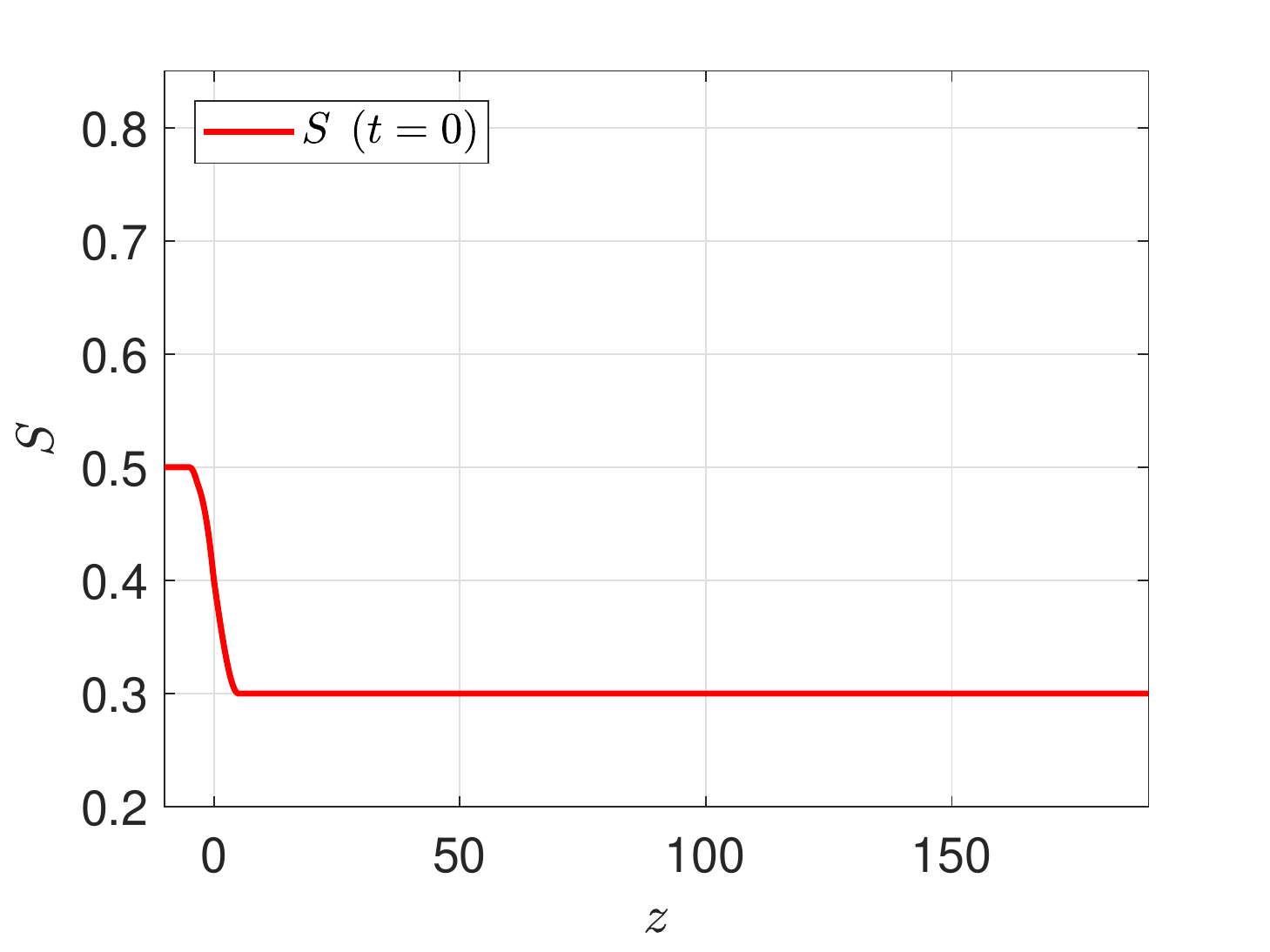}
\includegraphics[scale=0.38]{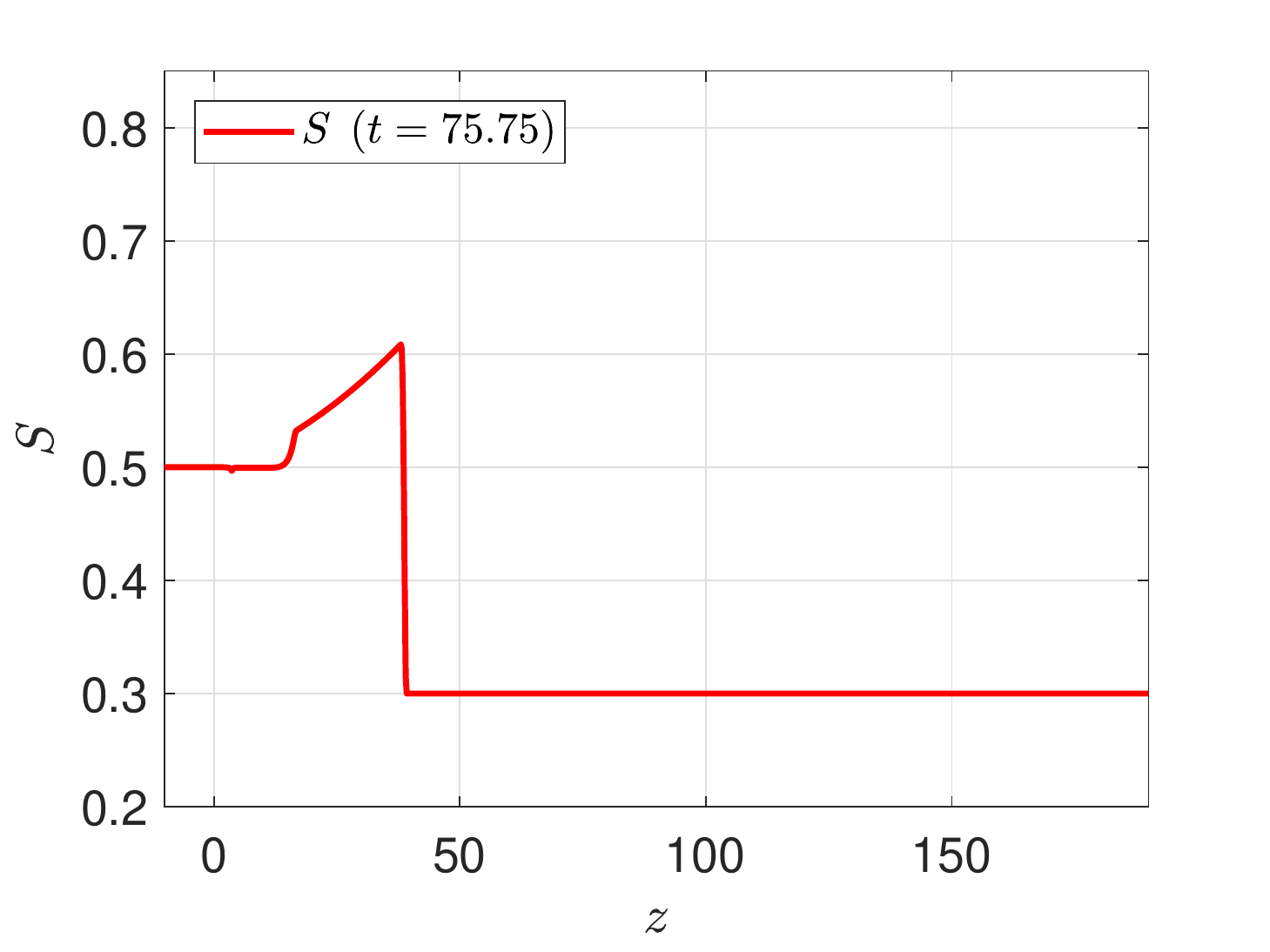}
\\
\includegraphics[scale=0.38]{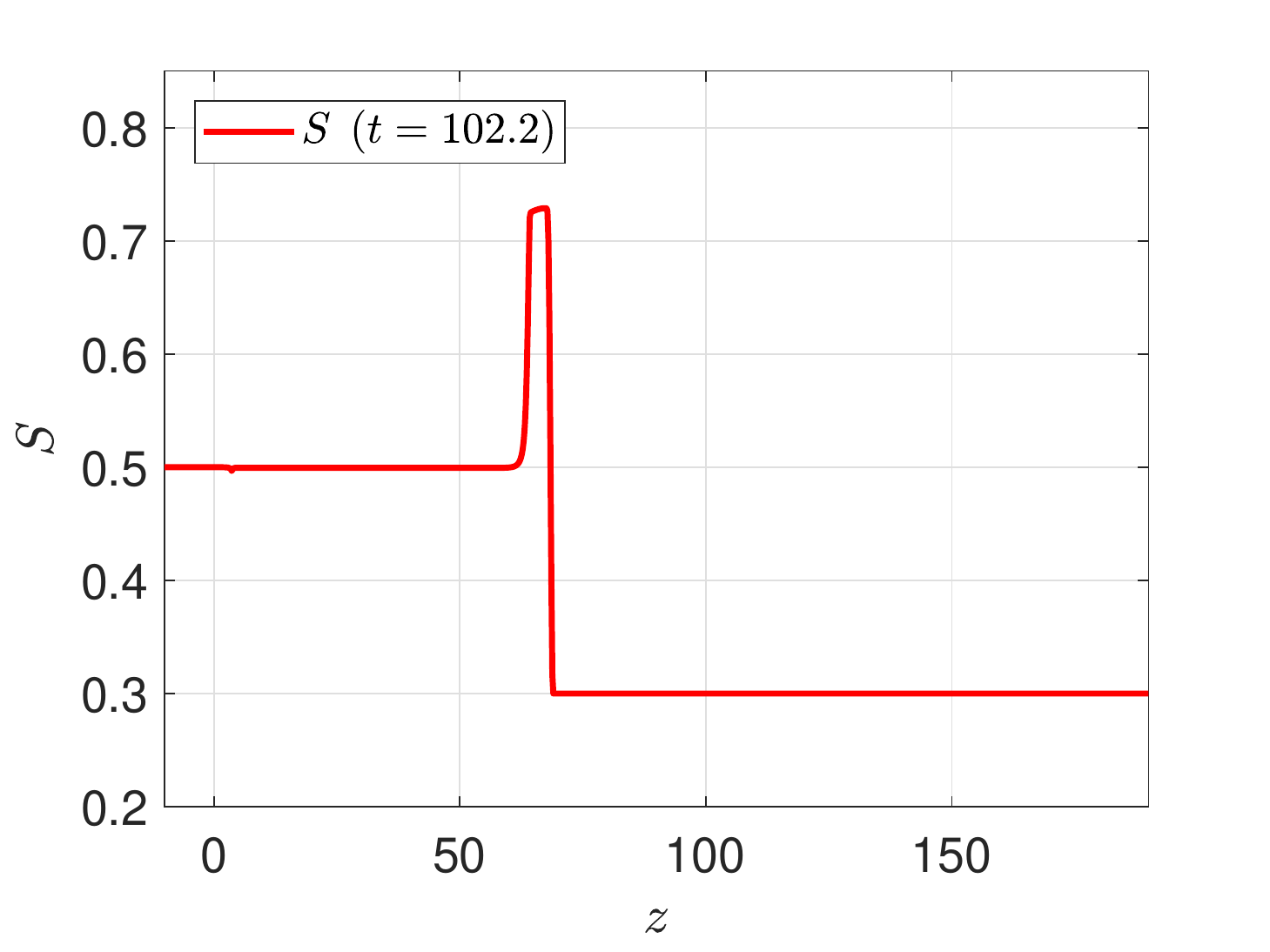}
\includegraphics[scale=0.38]{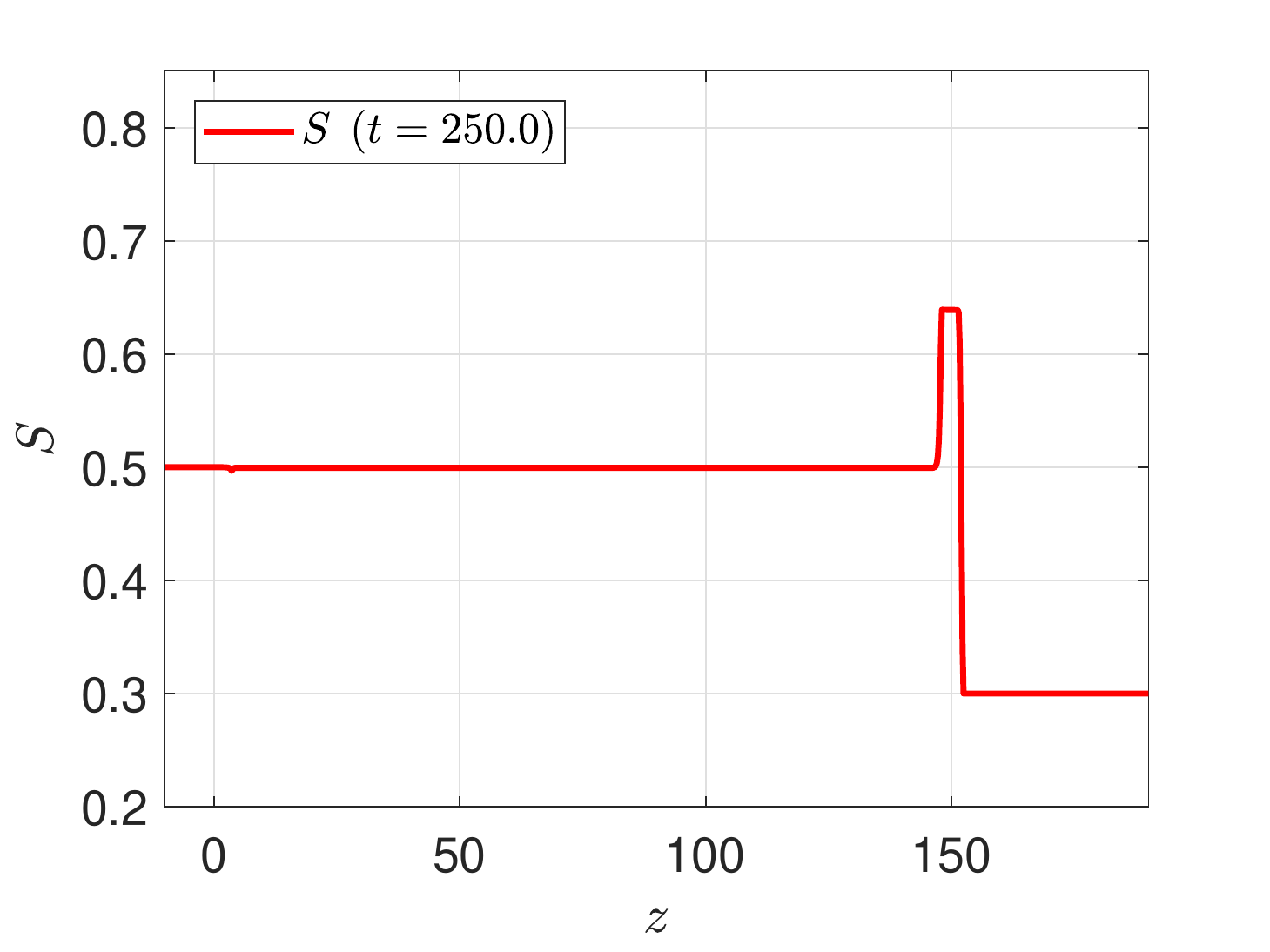}
\caption{\label{2PF_fig:S_Plateaus} Saturation profiles for different time points. Besides the initial condition (top left) and the final 
saturation profile (bottom right),
two intermediate profiles are shown, which have the form of a plateau. Contrary to the final saturation profile their plateaus 
are not stable, since the speeds of the infiltration and the drainage fronts are equal only for $S_P=0.634$.}
\end{figure}

\section{Final remarks and comparison with experiments}
\label{2PF_sec:Conclude}

In this work, a one-dimensional two-phase flow model has been analysed for infiltration problems. For simplicity, 
we have assumed that the medium is homogeneous and a constant total velocity is prescribed at the boundary. Dynamic and hysteretic 
effects are included in the capillary pressure with transitions between drainage and infiltration processes being modelled by a 
play-type hysteresis model having vertical scanning curves. Relative permeabilities are modelled as functions of saturation and 
capillary pressure in order to make their hysteretic nature explicit.

The focus being on travelling waves (TW), the system of partial differential equations is transformed into a dynamical system. This system 
is analysed for two different scenarios, A and B. In Scenario A, the hysteresis appears only in the capillary pressure, and we 
consider a broad range of dynamic capillarity terms, from small to large ones. In Scenario B, hysteresis is included in both the 
relative permeabilities and in the capillary pressure, whereas the dynamic capillary effects are kept small. For each scenario the 
existence of TW solutions is studied. In particular, we show that if the dynamic capillary effects exceed a certain threshold value,  
the TW profiles become non-monotonic. Such results complement the analysis in \cite{cuesta2000infiltration,VANDUIJN2018232,mitra2018wetting}
done for the unsaturated flow case, respectively in \cite{van2007new,van2013travelling} for two-phase flow but without hysteresis. 
From practical point of view, the present analysis provides a criterion for the occurrence of overshoots in two-phase infiltration experiments.

Based on the TW analysis, we give admissibility conditions for shock solutions to the hyperbolic limit of the system. 
Motivated by the hysteretic and dynamic capillarity effects, such solutions do not satisfy the classical entropy condition.
This is because the standard entropy solutions to hyperbolic two-phase flow models are obtained as limits of solutions 
to classical two-phase flow models, thus not including hysteresis and dynamic capillarity. In particular, for the infiltration 
case of Scenario A, apart from the classical solutions, there can be solutions consisting of (i) an infiltration shock followed 
by a rarefaction wave having non-matching speeds, or (ii) an infiltration shock  followed by a drainage shock resulting in a 
growing saturation plateau (overshoot) in between. This is similar to the results in \cite{van2007new,van2013travelling}. 
In Scenario B, the entropy solutions are shown to depend also on the initial pressure. In particular, if certain parametric 
conditions are met, the solutions may include ones featuring a stable saturation plateau between an infiltration front and a drainage 
front, both travelling with the same velocity. Such solutions are obtained e.g. in \cite{schneider2018stable}, but only after 
generating the overshoot through a change in the boundary condition. All cases mentioned above have been reproduced by numerical 
experiments, in which a good resemblance has been observed between the TW results and the long time behaviour of the solutions 
to the original system of partial differential equations.

From practical point of view, we note that the present analysis can also be used to explain experimental results reported e.g. 
in \cite{dicarlo2004experimental,kalaydjian1992dynamic,glass1989mechanism,zhuang2017advanced}. The occurrence of saturation 
overshoots is predicted theoretically for high enough dynamic capillary effects, namely of the $\t$ value in \eqref{2PF_eq:ScalingOfTau}. 
In dimensionless setting this can be assimilated to an injection rate that is sufficiently large. This is in line with the 
experimental results in \cite{dicarlo2004experimental}, where the development of plateau like profiles was observed for high 
enough injection rates, as shown in Figure 5 of \cite{dicarlo2004experimental} and Figure 5.3 of \cite{zhuang2017advanced}. 
Similarly, in the water and oil case, the plateaus are seen to develop and grow in Figures 5-6, 8-9, 18 of \cite{gladfelter1980effect}.
This behaviour is predicted by the analysis in \Cref{2PF_sec:CaseA}. Moreover, Figure 10 of \cite{gladfelter1980effect} might 
be presenting the case when the saturation has developed a plateau between two fronts travelling with the same velocity, 
a situation that is explained by the authors by means of hysteretic effects in the flux functions. Such solutions are 
investigated numerically in \cite{schneider2018stable,hilfer2014saturation}, where it is shown that the plateaus can persist 
in time but without explaining how they are generated. The results in \Cref{2PF_sec:CaseB} partly support the conclusions there, 
but also explain the mechanism behind the development of such plateaus. We mention \cite{kacimov1998nonmonotonic} in this regard, 
where the authors conclude that a similar mechanism must be responsible for observed  stable saturation plateaus inside viscous fingers.

\section*{Acknowledgment}
K. Mitra is supported by Shell and the Netherlands Organisation for Scientific Research (NWO), Netherlands through the
CSER programme (project 14CSER016) and by the Hasselt University, Belgium through the project BOF17BL04. I.S. Pop is supported 
by the Research Foundation-Flanders (FWO), Belgium through the Odysseus programme (project
G0G1316N). C.J. van Duijn acknowledges the support of the Darcy Center of Utrecht University and Eindhoven University of Technology. 
The work of T. K{\" o}ppl and R. Helmig is supported by the Cluster of Excellence in Simulation Technology (EXC 310/2).
Furthermore, R. Helmig acknowledges the support of the Darcy Center and the Deutsche Forschungsgemeinschaft (DFG, German Research Foundation),
SFB 1313, Project Number 327154368.

\bibliography{Literature}
\nocite{*}
\bibliographystyle{plain}

\end{document}